\documentclass[11pt, reqno]{amsart}

\usepackage[margin= 1.24in]{geometry}

\usepackage{amssymb,amsthm,amsfonts,amscd,latexsym, dsfont, color}

\usepackage[dvipsnames]{xcolor}
\usepackage{tikz}
\usetikzlibrary{arrows.meta}

\usepackage{bm}

\usepackage{multirow}

\newcommand{\T}{\mathrm{T}}
\newcommand{\N}{\mathrm{N}}

\newcommand{\im}{\mathrm{image}} 
\newcommand{\coker}{\mathrm{coker}}

\usepackage{dynkin-diagrams}
\usepackage[title]{appendix}
\usepackage[american]{babel} 

\usepackage{array}
\usepackage[citecolor=ForestGreen,linkcolor=blue,urlcolor=blue]{hyperref} 
 \usepackage{graphics}
\usepackage{wrapfig}
\usepackage{pst-node}
\usepackage{pythonhighlight}

\usepackage{tikz-cd} 

\hypersetup{colorlinks}
\usepackage[normalem]{ulem}

\usepackage{setspace} 

\usepackage{tcolorbox}

\usepackage{footnote}

\usepackage{cancel}

\newcommand{\hol}{\mathrm{hol}}

\newcommand{\pr}{\mathrm{pr}}

\usepackage{MnSymbol}
\usepackage{enumitem}

\newcommand{\Span}{\mathrm{span}}

\definecolor{bluegreen}{RGB}{5, 170, 204}

\counterwithin{equation}{section}
\definecolor{defnyellow}{RGB}{255,224,102} 
\definecolor{lightblue}{RGB}{170, 220, 255} 
\definecolor{darkblue}{RGB}{0, 125, 230}

\newcommand{\Gtwosplit}{\mathrm{G}_2'}

\newtcolorbox{bluebox}[1]{colback=lightblue,colframe=darkblue,fonttitle=\bfseries,title=#1}
\newtcolorbox{redbox}[1]{colback=red!5!white,colframe=red!75!black,fonttitle=\bfseries,title=#1}

\usepackage{pgfplots}

\usepgfplotslibrary{colormaps}
\pgfplotsset{
    compat=newest,
    colormap={mycolormap}{color=(lightgray) color=(white) color=(lightgray) } }

\definecolor{mydarkblue}{RGB}{37, 42, 200}

\definecolor{mygreen}{RGB}{0, 150, 50}

\newcommand{\vis}{\partial_{\mathrm{vis}}}
\newcommand{\Ein}{\mathrm{Ein}}

\newcommand{\Mat}{\mathsf{Mat}} 
\newcommand{\Id}{\mathrm{Id}}
 
\newcommand{\SO}{\mathrm{SO}}

\newcommand{\PSL}{\mathrm{PSL}}
\newcommand{\Gr}{\mathsf{Gr}} 
\newcommand{\Stab}{\mathsf{Stab}} 
\usepackage{mathrsfs}
\usepackage[utf8]{inputenc}
\usepackage[T1]{fontenc}
\newcommand{\R}{\mathbb R}
\newcommand{\RP}{\mathbb R \mathbb P}
\newcommand{\CP}{\mathbb C \mathbb P}
\newcommand{\C}{\mathbb C}
\newcommand{\Z}{\mathbb Z}
\newcommand{\Q}{\mathbb Q}

\newcommand{\Ha}{\mathbb{H}} 

\newcommand{\g}{\mathfrak{g}}
\newcommand{\frakk}{\mathfrak{k}} 
\newcommand{\frakp}{\mathfrak{p}} 

\newcommand{\X}{\mathbb{X}}

\newcommand{\End}{\mathsf{End}} 
\newcommand{\D}{\mathbb{D}}
\newcommand{\sphere}{\mathbb{S}}

\newcommand{\sllie}{\mathfrak{sl}}
\newcommand{\SL}{\mathrm{SL}}
\newcommand{\Sp}{\mathrm{Sp}}
\newcommand{\GL}{\mathrm{GL}}

\newcommand{\tr}{\text{tr}}
\newcommand{\id}{\text{id}}

 \newcommand{\Hom}{\mathrm{Hom}}

 \newcommand{\del}{\partial}

\usepackage{mathtools}

\newcommand{\Pho}{\mathrm{Pho}}

\newcommand{\rank}{\text{rank}}

\renewcommand{\im}{\text{im}}

\newcommand{\Flag}{\mathrm{Flag}}

\newcommand{\diag}{\mathrm{diag}}

\newcommand{\Hit}{\mathsf{Hit}}

\newcommand{\Diff}{\mathsf{Diff}}

\newtheorem{mainthm}{Theorem}
\newtheorem{maincor}{Corollary}

\theoremstyle{plain}

\newtheorem{theorem}{Theorem}[section]
\newtheorem{proposition}[theorem]{Proposition}
\newtheorem{lemma}[theorem]{Lemma}
\newtheorem{corollary}[theorem]{Corollary}

\newtheorem{example}{Example}[section]

\newtheorem{definition}[theorem]{Definition}
\newtheorem{remark}[theorem]{Remark}

\newcommand{\lh}[1]{\mathcal{F}_{1,#1-1}}

\usepackage{soul}

\usepackage{graphicx}

\begin{document}

\title{On Line-Hyperplane Structures for Hitchin Representations}
\author{Parker Evans and Andrea Tamburelli}

\begin{abstract}
Let $n \geq 4$ and $\rho: \pi_1S \rightarrow \PSL(n,\R)$ be a Hitchin representation.  
We study the topology of a cocompact domain of discontinuity $\Omega_{\rho}$ in the flag manifold $\mathcal{F}_{1,n-1}$ of line-hyperplane pairs in $\R^n$ defined by Guichard-Wienhard. In particular, we lift $\Omega$ to its $(\Z_2\times \Z_2)$-cover $\hat{\Omega}$ in the Stiefel manifold $V_2(\R^n)$. The domain $\hat{\Omega}$ is highly connected and smoothly fibers over hyperbolic space $\Ha^2$ with unknown fiber $\hat{\mathfrak{F}}$, a closed $(2n-5)$-manifold. 
We determine the homeomorphism type of $\hat{\mathfrak{F}}$ as well as its diffeomorphism type up to connected sum with a homotopy sphere. 
The classification of $\hat{\mathfrak{F}}$ involves computing both its homology and a certain differential topology invariant needed to invoke Wall's classification of highly connected odd-dimensional manifolds.
\end{abstract}

\maketitle 

\vspace{-4ex}
\tableofcontents 

\section{Introduction}

Let $S = S_g$ be a closed oriented surface of genus $g \geq 2$. Recall that \emph{Teichm\"uller space} $T(S)$ can be seen as the moduli space of marked hyperbolic structures on $S$. A hyperbolic metric on a surface $S$ is tantamount to a \emph{locally homogeneous} $(\PSL(2,\R), \Ha^2)$-structure on $S$, in the modern language of $(G,X)$-structures. 
The holonomies of the hyperbolic structures in $T(S)$ are \emph{Fuchsian}, or discrete and faithful, representations $\pi_1S \rightarrow \PSL(2,\R)$, and the locus of all Fuchsian representations comprises a connected component of the \emph{character variety} $\chi(S,\PSL(2,\R))$ of reductive representations up to conjugation. In this way, the holonomy map explicitly identifies these geometric and algebraic models of Teichm\"uller space. 

There are now many moduli spaces of $(G,X)$-structures associated to surface group representations, which are much less understood than Teichm\"uller space. To describe these analogues, 
we replace Fuchsian representations by ($G$-)\emph{Hitchin} representations $\rho:\pi_1S \rightarrow G$, where $G$ is a split real simple Lie group of higher rank, such as $\PSL(n,\R)$ for $n\geq 3$. These $G$-Hitchin representations are deformations of Fuchsian representations $\rho_0:\pi_1S\rightarrow \PSL(2,\R)$ once post-composed by the inclusion $\iota_{\pr}: \PSL(2,\R) \hookrightarrow G$ of the \emph{principal} $\PSL(2,\R)$-subgroup of $G$. When $G=\PSL(n,\R)$, for example, the inclusion $\iota_{\pr}$ is just the unique irreducible representation $\PSL(2,\R)\rightarrow \PSL(n,\R)$ up to conjugation. 
We denote the \emph{$G$-Hitchin component}, the space of Hitchin representations in the $G$-character variety $\chi(S,G)$, by $\Hit(S,G)$. Under this definition, the compelling features of $\Hit(S,G)$  are perhaps hidden. In \cite{Hit92}, Hitchin defined this space and identified interesting analogies with Teichm\"uller space. To name one:  $\Hit(S,G)$ is a cell of dimension $(2g-2)\dim(G)$. 
Another remarkable feature, later discovered, is that Hitchin representations are also discrete and faithful \cite{Lab06, FG06}. Thus, Hitchin components are often referred to as \emph{higher Teichm\"uller spaces} \cite{Wie18}. 

In the groundbreaking work \cite{GW12}, Guichard-Wienhard proved the component $\Hit(S,G)$ parametrizes \emph{some} moduli space $\mathscr{M}$ of $(G,X)$-structures on \emph{some} closed manifold $M$, where $X=G/P$ is a prescribed \emph{flag manifold} of $G$. 
Recently, \cite{AMTW25, Dav25} proved that this mysterious manifold $M$ admits a fiber bundle structure $p:M\rightarrow S$ over the original surface $S$. 
With this fact in hand, the \cite{GW12} result more precisely says the following: the holonomy $\hol: \pi_1M\rightarrow G$ of any geometric structure in $\mathscr{M}$ factors through $p_*: \pi_1M \rightarrow \pi_1S$ to obtain the form $\hol = \overline{\hol} \circ p_*$, and there is a homeomorphism $\mathscr{M}\rightarrow \Hit(S,G)$ given by taking this \emph{descended holonomy} $\overline{\hol}$. Thus, geometric structures in $\mathscr{M}$ on the fiber bundle $M \rightarrow S$ secretly encode representations of the fundamental group of the base surface $S$, and this construction provides a geometric incarnation of $\Hit(S,G)$ analogous to that of $T(S)$. 

Let us now fix $G$ and $X$.  
After seeing this remarkable `geometrization' theorem, one can't help but wonder:  
what is this manifold $M$? Indeed, $M=M_{\rho}$ is indirectly constructed in \cite{GW12} as a compact quotient $M_{\rho} = \rho(\pi_1S)\backslash \Omega_{\rho}$ of an open domain of discontinuity $\Omega_{\rho} \subset X$, whose topology is independent of the Hitchin representation $\rho$. Such domains were constructed more generally for \emph{Anosov representations} in \cite{GW12} and later \cite{KLP18}. 
Now, from the perspective of this construction, the topology of the manifold $M$ is rather unclear. 

In fact, when $\rho$ is Hitchin, the domain $\Omega_{\rho}$ admits a fiber bundle $p:\Omega_{\rho}\rightarrow \Ha^2$, which leads to a projection $p:M_{\rho} \rightarrow S$ \cite{AMTW25, Dav25}. The common fiber $\mathfrak{F}$ of these projections is known to be a closed codimension two submanifold of $X$. 
In principle, $\mathfrak{F}$ depends on the fibration $p$, but its homotopy type is fixed since $\mathfrak{F}$ is homotopy equivalent to the domain $ \Omega$. 

Recently, there has been progress towards understanding the topology of these $(G,X)$-manifolds $M$ and their fibers $\mathfrak{F}$ \cite{CTT19,DS20, ADL24, AMTW25, Dav25, DE26a, DE26, Har26}. See \cite{DE26} for a recent summary of results. On the other hand, \emph{explicit descriptions} of the moduli space $\mathscr{M}$ of $(G,X)$-structures on $M \rightarrow S$ with $G$-Hitchin descended holonomy are much more elusive, and exist only for the Lie groups $\PSL(2,\R)$, $\SL(3,\R)$, $\PSL(4,\R)$, $\SO(2,3)$, and $\Gtwosplit$ \cite{Gol90, CG93, GW08, Bar10, CTT19, NR25, RT25, DE26a}. 

\subsection{Main Results}
In this paper, we consider the case of $G=\PSL(n,\R)$ and $X = \mathcal{F}_{1,n-1}$, the flag manifold of line-hyperplane pairs $(\ell, H)$ in $\R^n$. While Hitchin representations admit a cocompact domain of discontinuity in projective space $\RP^{n-1}$ for $n=3$ and $n\geq 4$ even \cite{GW12}, for $n \geq 5$ odd they admit no such domain in any Grassmannian $\Gr_k(\R^n)$ \cite{Ste23}. On the other hand, Hitchin representations admit a cocompact domain in $\mathcal{F}_{1,n-1}$ for all $n\geq 3$ \cite{GW12}. 
This pair $G=\PSL(n,\R)$ and $X=\mathcal{F}_{1,n-1}$ is thus quite natural to consider in pursuit of a uniform and economic `geometrization'. 

The first nontrivial case is $n=3$. 
Such $(\SL(3,\R),\mathcal{F}_{1,2})$-structures were studied prior to the general work of Guichard-Wienhard by Barbot in \cite{Barb01,Barb10}. He constructed $\mathcal{F}_{1,2}$-structures on circle bundles $\sphere^1\rightarrow M_{\rho} \rightarrow S_g$ for \emph{non-Hitchin} representations $\rho:\pi_1(S_g)\rightarrow \SL(3,\R)$, including  deformations of Fuchsian representations $\rho_0:\pi_1S\rightarrow \SL(2,\R)$ under the post-composition of the reducible inclusion  $\iota_{2+1}:\SL(2,\R)\rightarrow \SL(3,\R)$. 

Now, returning to the Hitchin case, the line-hyperplane structures for $n=3$ were described recently by Nolte and Riestenberg in \cite{NR25}. The domain $\Omega_{\rho} \subset \mathcal{F}_{1,2}$ is disconnected and its quotient $M_{\rho}$ is diffeomorphic to three copies of the projective unit
tangent bundle $\mathbb{P}(\T^1S)$. 
In \cite{NR25}, a detailed description of one component is provided, including the distinguishing synthetic features of these Hitchin $(\SL(3,\R), \mathcal{F}_{1,2})$-structures, in a similar spirit to \cite{GW08}. That is to say, both $M$ and $\mathscr{M}$ are fairly well understood for $n=3$. However, before the present work, nothing was known for $n \geq 4$ about the fiber bundle structure of $M$.

Let us now describe the new results. For $\rho:\pi_1S \rightarrow \PSL(n,\R)$ Hitchin and $n\geq 4$, we lift the \cite{GW12}-domain $\Omega_{\rho} \subset \mathcal{F}_{1,n-1}$ to a $(\Z_2\times \Z_2)$-cover $\hat{\Omega}_{\rho}$ in the Stiefel manifold $V_2(\R^n)$. When $n\geq 5$, the domain $\hat{\Omega}_{\rho}$ is the universal cover of $\Omega_{\rho}$. Again by \cite{AMTW25, Dav25}, $\hat{\Omega}_{\rho}$ also fibers over $\Ha^2$, with fiber $\hat{\mathfrak{F}}$ some closed manifold of dimension $2n-5$. 

It turns out that $\hat{\mathfrak{F}}$ is \emph{highly connected}. Such manifolds were classified by Wall in the odd-dimensional case via surgery theory in \cite{Wal67}. 
The essential difficulty to invoke his results in our case is the computation of a seemingly intractable invariant $S\beta$, described shortly in Theorem \ref{Thm:WallWilkens}. Now, by applying Wall's classification, we obtain our main theorem, which determines $\hat{\mathfrak{F}}$ topologically, and also smoothly up to finite ambiguity. 

\begin{mainthm}\label{thm:A}
Let $\rho:\pi_1S \rightarrow \PSL(n,\R)$ be a Hitchin representation. Denote $\hat{\Omega}_{\rho} \subset V_2(\R^n)$ as the $(\Z_2\times \Z_2)$-cover of the domain $\Omega_{\rho} \subset \mathcal{F}_{1,n-1}$. Consider the smooth manifold $Y_n$ given by 
\begin{align*}
    Y_n:= \#_{i=1}^{k(n)}(\sphere^{n-3}\times \sphere^{n-2}), 
\end{align*}
where $k(n)=n-1$ when $n$ is even and $k(n)=2n-1$ when $n$ is odd. 

Then there is a smooth fiber bundle projection 
$p: \hat{\Omega}_{\rho} \rightarrow \Ha^2$ with fiber $\hat{\mathfrak{F}}$ such that 
\begin{enumerate}[label=(\roman*)]
    \item When $n\geq 6$, $\hat{\mathfrak{F}}$ is diffeomorphic to $ Y_n \#\Sigma$, for some homotopy $(2n-5)$-sphere $\Sigma$. 
    \item When $n=5$, $\hat{\mathfrak{F}}$ is diffeomorphic to $Y_5 = \#_{i=1}^9(\sphere^2\times \sphere^3)$. 
    \item When $n=4$, $\hat{\mathfrak{F}}$ is diffeomorphic to the double of the $T(2,6)$-link complement in $\sphere^3$. 
\end{enumerate}
\end{mainthm}
\setcounter{maincor}{1}

\begin{maincor}\label{cor:B}
Let $n\geq 6$. The fiber $\hat{\mathfrak{F}}$ of $p$ is homeomorphic to $Y_{n}$. 
\end{maincor}

Recall that a \emph{homotopy $m$-sphere} $\Sigma$ is a closed smooth $m$-manifold that is homotopy equivalent to $\sphere^m$. 
By deep work of Kervaire-Milnor and Smale, the set $\Theta_m$ of all homotopy $m$-spheres up to diffeomorphism, an abelian group under connected sum, is \emph{finite} for $m\geq5$ \cite{KM63, Sma62a}. 
Thus, Theorem \ref{thm:A} indeed determines $\hat{\mathfrak{F}}$ smoothly up to finite ambiguity. 

As Theorem \ref{thm:A} makes clear, the cases $n=4$ and $n=5$ are exceptional. Indeed, $n=4$ is the unique case in which $\hat{\mathfrak{F}}$ is not simply connected. For details, see  $\S$\ref{Sec:Intro:n=4}. For further remarks on $n=5$, when $\hat{\mathfrak{F}}$ is 5-dimensional, see $\S$\ref{Sec:Intro:n=5}. 

It is now known that an odd-dimensional sphere $\sphere^{m}$ admits a unique smooth structure exactly when $m\in \{1,3,5,61\}$, with the final case $m=61$ addressed in \cite{WX17}. Thus, $\Theta_{61}$ is trivial. Correspondingly, we conclude the smooth structure of $\hat{\mathfrak{F}}$ in another case.  

\begin{maincor}\label{cor:C}
When $n=33$, the fiber $\hat{\mathfrak{F}}$ of $p$ is diffeomorphic to $\#_{i=1}^{65}(\sphere^{30}\times \sphere^{31})$. 
\end{maincor}

To rephrase Theorem \ref{thm:A}, the quotient $M_{\rho}= \rho(\pi_1S)\backslash \Omega_{\rho}$ fibers over $S$ with smooth fiber $\mathfrak{F}$ that is a free $(\Z_2\times \Z_2)$-quotient of $\hat{\mathfrak{F}}\cong Y_n \#\Sigma$. 

In the next portion of the introduction, we introduce our techniques and explain our decision to work upstairs in the Stiefel manifold, where the more natural results are obtained.

\subsection{Techniques and Further Discussion}\label{Sec:Technique}
One of the main contributions of this work is to employ techniques from the realm of highly connected manifolds to the study of geometric structures associated to surface group representations. 

The fiber $\hat{\mathfrak{F}}$ we seek is a priori known to be a closed codimension two submanifold of $V_2(\R^n)\cong \T^1\sphere^{n-1}$. It is well known that the Stiefel manifold $V_2(\R^n)$ is $(n-3)$-connected. A convenient feature of the domain $\hat{\Omega}$ is that its complement $\hat{K}$ admits the structure of a finite CW complex and hence has bounded cohomological dimension. Using the homotopy equivalence $\hat{\mathfrak{F}} \simeq \hat{\Omega}$, one finds that $\hat{\mathfrak{F}}$ is $(n-4)$-connected. Thus, just like the Stiefel manifold, $\hat{\mathfrak{F}}$ is a \emph{highly connected (odd) manifold}, namely an $(m-1)$-connected $(2m+1)$-manifold. This key fact allows us to apply Wall's classification of such manifolds achieved in the series \cite{Wal63c, Wal63a, Wal63b, Wal65, Wal66, Wal67}. On the other hand, the original fiber  $\mathfrak{F}=\hat{\mathfrak{F}}/(\Z_2\times \Z_2)$, is not highly connected, and cannot currently be understood through such a classification; see $\S$\ref{Subsec:Cover} for further discussion. 

We now outline the main steps involved in the proof of Theorem \ref{thm:A}, namely the computation of the invariants needed to apply Wall's classification.  

The starting point is to realize the fiber $\hat{\mathfrak{F}}$ as a \emph{base of pencil} as in \cite{Dav25}. In many similar cases, such fibers are \emph{themselves fiber bundles}, often sphere bundles; see $\S$\ref{Sec:RelatedWork}.
However, currently, this is not quite the case.
Instead, we prove a \emph{Structure Lemma} that realizes $\hat{\mathfrak{F}}$ as the union of a pair of fiber bundles. 
More precisely, we study a smooth surjective map $p: \hat{\mathfrak{F}} \rightarrow \sphere^{n-1}$, which decomposes both the total space $\hat{\mathfrak{F}}$ and the base $\sphere^{n-1}$ into pieces: a ``generic locus'' $\hat{\mathfrak{F}}_{gen}\rightarrow S_{gen}$, and a ``singular locus'' $\hat{\mathfrak{F}}_{sing} \rightarrow S_{sing}$. That is, we split $\hat{\mathfrak{F}} = \hat{\mathfrak{F}}_{gen} \sqcup \hat{\mathfrak{F}}_{sing}$ and $\sphere^{n-1}=S_{gen} \sqcup S_{sing}$. The map $p$ restricts to an honest fiber bundle with fiber a sphere on the generic locus and the singular locus, though the sphere fibers are of different dimensions for the two fiber bundles. The miracle here is that the failure of the `projection' $p$ to be a submersion is sufficiently tame so as to allow us to understand $\hat{\mathfrak{F}}$ as a sort of bundle nonetheless. Using the Structure Lemma, with only the Mayer-Vietoris and Gysin exact sequences, and a few additional geometric insights, we can compute the homology of $\hat{\mathfrak{F}}$. The important feature that emerges from the calculation is that $H_*(\hat{\mathfrak{F}},\Z)$ is torsion-free.

Of course, in general, the homology $H_*(Y)$ of a topological space $Y$ is too weak of an invariant to classify $Y$ up to homeomorphism or even up to homotopy type. However, in the presence of a generous amount of connectivity, one might not be far off. Indeed, this is exactly the purpose of Wall's classifications of highly connected manifolds \cite{Wal61, Wal67}. 

To finish our proof sketch of Theorem \ref{thm:A}, we describe Wall's classification\footnote{We note that Wall's student Wilkens handled the missing cases $m \in \{3,7\}$ unaddressed by Wall in \cite{Wil72}.} in the odd-dimensional case under the two additional hypotheses of  
torsion-free homology and \emph{stably trivial} tangent bundle, which $\hat{\mathfrak{F}}$ satisfies. 

\begin{theorem}[The Classification: Simplified Invariants {\cite{Wal67, Wil72}}]\label{Thm:WallWilkens}
Let $m \geq 3$ and $P^{2m+1}$ be a closed smooth $(m-1)$-connected manifold with stably trivial tangent bundle and torsion-free integer homology. Then $P$ is classified smoothly up to connected sum with a homotopy $(2m+1)$-sphere by an invariant
\[ S\beta \in \Hom\big( H_{m+1}(P,\Z), \pi_{m}(\SO(m+1)) \big). \] 
\end{theorem}

The most difficult part of the paper is to compute this invariant $S\beta$, which we now describe. Choose any homology class $\sigma \in H_{m+1}(P,\Z)$. By work of Haefliger \cite{Hae62}, this class can be represented by a smoothly embedded sphere $f_{\sigma}: \sphere^{m+1} \hookrightarrow P$. As it turns out, the once-stabilized normal bundle $\N f_{\sigma} \oplus \varepsilon^1_{\R}$, where $\varepsilon^{1}_{\R}$ denotes a trivial real rank one line bundle, is independent of such $f_{\sigma}$. Consequently, the invariant $S\beta$ can be described as follows: 
\[ S\beta(\sigma) = [\N f_{\sigma} \oplus \varepsilon^1_{\R}] \in \pi_{m}(\SO(m+1)),\]
where the right-hand side denotes the clutching function of the given vector bundle. 

To conclude Theorem \ref{thm:A} from Theorem \ref{Thm:WallWilkens}, we compute the Betti number $b_{m+1}(\hat{\mathfrak{F}})$ and prove that the invariant $S\beta$ vanishes for $\hat{\mathfrak{F}}$. 
As $\dim(\hat{\mathfrak{F}})=2n-5$, we have 
$m=n-3$ in our case. 
The verification that $S\beta =0$ requires that we explicitly represent a collection of generating classes $\sigma \in H_{n-2}(\hat{\mathfrak{F}},\Z)$ by smoothly embedded $(n-2)$-spheres $Q_{\sigma} \hookrightarrow \hat{\mathfrak{F}}$ and then prove that the normal bundle of $Q_{\sigma}$ becomes trivial after 1-stabilization. 

\subsubsection{Computing $S\beta$}

Since the computation of the invariant $S\beta$ is the most involved part of the paper, we provide a brief overview of the approach here. 

Recall that we split $\hat{\mathfrak{F}}$ into two pieces: $\hat{\mathfrak{F}} = \hat{\mathfrak{F}}_{gen} \sqcup \hat{\mathfrak{F}}_{sing}$. 
Now, the generators in middle homology $H_{n-2}(\hat{\mathfrak{F}}, \Z)$ come in two types from our perspective: classes that are induced by inclusion from $\hat{\mathfrak{F}}_{gen}$ or $\hat{\mathfrak{F}}_{sing}$, and classes that are not. 

In terms of the former, it turns out we need only include generators from $H_{n-2}(\hat{\mathfrak{F}}_{gen},\Z)$. We compute $S\beta$ on these homology classes rather directly, but only after developing a clear geometric picture of the fibration $\hat{\mathfrak{F}}_{gen} \rightarrow S_{gen}$. Now, the base $S_{gen}$ has the homotopy type of a CW complex and has geometrically evident embedded spherical generators for $H_{n-2}$. We work to show these spheres lift to $\hat{\mathfrak{F}}_{gen}$. It happens that our argument to verify the possibility of lifting leads simultaneously to the proof that $S\beta$ must vanish on these classes.

The homological generators not induced from inclusion provide much more difficulty. In this case, for each such class $\sigma \in H_{n-2}(\hat{\mathfrak{F}},\Z)$, we proceed as follows. First, we produce an $(n-2)$-sphere $Q_{\sigma} \subset \hat{\mathfrak{F}}$ realized in two pieces through the geometry of the almost-fibration $p: \hat{\mathfrak{F}}\rightarrow \sphere^{n-1}$. In particular, $Q_{\sigma}$ is obtained by gluing a smooth copy of $\mathbb{B}^2\times \sphere^{n-4} \subset \hat{\mathfrak{F}}_{gen}$ with a smooth copy of $\mathbb{S}^1 \times \mathbb{D}^{n-3} \subset  \hat{\mathfrak{F}}_{sing}$. However, this construction only produces a $C^0$-embedded sphere $Q_{\sigma}$, as the two pieces meet along a `corner singularity'. We provide an explicit smoothing of $Q_{\sigma}$ to obtain a $C^{\infty}$-submanifold $Q_{\sigma}^{s}$, and
carefully check that an appropriate version of the smoothing process preserves embeddedness.
We can then finally verify that the once-stabilized normal bundle of $Q_{\sigma}^s$ is trivial. Our strategy for this task is somewhat direct and somewhat indirect. Concretely, we produce enough global sections of the bundle $\N Q_{\sigma}^s \oplus \varepsilon^1_{\R}$ to guarantee triviality, but we do not achieve an explicit trivialization. 

Finally, we conclude $S\beta=0$ for $\hat{\mathfrak{F}}$. Theorem \ref{thm:A} then says that $\hat{\mathfrak{F}}$ has the simplest possible topology within its cohort of highly connected companions, aside from homotopy spheres.  

\subsubsection{The exceptional case $n=4$}\label{Sec:Intro:n=4}

The case $n=4$ is exceptional, as $\hat{\mathfrak{F}}$ is not simply connected and it is also a 3-manifold. In this case, we find $\hat{\mathfrak{F}}$ is diffeomorphic to the double of a torus link complement in $\sphere^3$. 
While this case does provide a concrete perspective on some of our general ideas, the fiber $\hat{\mathfrak{F}}$ is quite different from the generic case. See $\S$\ref{Sec:n=4} for further details. 

\subsubsection{The illustrative case $n=5$}\label{Sec:Intro:n=5}

Wall's classification does not apply in the case $m=2$, corresponding to simply connected 5-manifolds $P^5$. However, in this case an even stronger result is known: a complete classification up to diffeomorphism due to Smale \cite{Sma62}, when $P$ is \emph{spin}, and Barden \cite{Bar65} in general. Our manifold $\hat{\mathfrak{F}}$ is indeed spin, and so \cite{Sma62} asserts that the second homology group over the integers is a complete invariant. We thus conclude Theorem \ref{thm:A}(ii) from the calculation of $H_2(\hat{\mathfrak{F}},\Z) \cong \Z^9$. 

We emphasize: to this date, there is \emph{no complete classification} of closed 5-manifolds up to diffeomorphism, even imposing finite cyclic fundamental group, or even $\pi_1(P)=\Z_2$. However, there is considerable related work in this direction. For partial results classifying closed 5-manifolds $P$ with fundamental group $\Z_2$, including certain circle bundles over simply connected 4-manifolds, see \cite{HS13, Ott26}. In \cite{Su11}, Su classified free involutions on $\sphere^2\times \sphere^3$, including determining all possible quotients. Very recently, \cite{Jin2026} computed the mapping class group of $\#_k (\sphere^2\times \sphere^3)$. To the best of our knowledge, despite these results, the classification of quotients of $\#_k(\sphere^2\times \sphere^3)$ under free finite group actions remains open. 

\subsubsection{Working in the cover}\label{Subsec:Cover}
Given the above model case $n=5$ and the status of finite quotients of simply connected 5-manifolds, we do not see how to determine $\mathfrak{F}=\hat{\mathfrak{F}}/(\Z_2\times \Z_2)$ in that case.
More generally, for $m\geq 3$, there is currently no classification of free $G$-actions on a fixed highly connected odd-dimensional manifold $M^{2m+1}$ for finite groups $G$, or even finite cyclic groups. In the related setting where $M=M^{2m}$ is even-dimensional, the theory is much more restricted, as any such $G$-action must respect the \emph{intersection form} on $H^m(M,\Z)$. For work in this direction, see \cite{SuYang21, FT24a, FT24b}. 
On the other hand, in the odd case, there is more flexibility as $H^m(M,\Z)$ is linked with $H^{m+1}(M,\Z)$ rather than itself. As a testament to this freedom, Johnson showed in \cite{Joh18} that \emph{any} finite group $G$ can act freely on $M_{k,m} := \#_{k}(\sphere^m\times \sphere^{m+1})$ for some $m,k$. However, the classification problem of free $G$-actions on a fixed space $M_{k,m}$, or further $M_{k,m}\#\Sigma$, with $\Sigma$ an exotic sphere, is still open. 

An additional issue in our case is that the $(\Z_2\times \Z_2)$-action of deck transformations on $\hat{\mathfrak{F}}$ does not respect the almost-fibration $p: \hat{\mathfrak{F}} \rightarrow \sphere^{n-1}$, so our main source of insight is lost. Thus, in all cases $m \geq 2$, we believe it is currently out of reach to determine $\mathfrak{F} = \hat{\mathfrak{F}}/(\Z_2\times \Z_2)$.

\subsection{Related Work}\label{Sec:RelatedWork}

We now discuss the results and strategies in some related works.

An important feature exhibited in many cases by the fibers $\mathfrak{F}$ of cocompact domains of discontinuity $\Omega \rightarrow \Ha^2$ of surface group representations, when realized as bases of pencils, is that $\mathfrak{F}$ itself has a fiber bundle structure. 
In \cite{CTT19}, Collier-Tholozan-Toulisse, for $G=\SO_0(2,n+1)$ and $X=\Pho(\R^{2,n+1})$, read off the topology of the fiber $\mathfrak{F}\cong \Pho(\R^{2,n})$ of a domain $\Omega \subset X$ for \emph{maximal representations} via the relationship between the domain of discontinuity and an (equivariant) maximal spacelike surface $\sigma: (\tilde{S},J)\rightarrow \Ha^{2,n}$ in pseudo-hyperbolic space. We recall that 
$V_2(\R^{n}) \cong_{\mathbf{Diff}} \Pho(\R^{2,n})$. 
As explained in \cite[Appendix A]{DE26a}, we can reinterpret work of \cite{CTT19} through the lens of bases of pencils. 
In fact, we can do the same for results of Alessandrini-Davalo-Li in \cite{ADL24} on fibers of projective structures for Hitchin and quasi-Hitchin representations. The base of pencil approach was also employed in recent work by Davalo and the first named author in \cite{DE26a, DE26} to study the fibers of domains in the \emph{Einstein universe} $\Ein^{p-1,p}$ and in certain $\Gtwosplit$-partial flag manifolds, where $\Gtwosplit$ is the split real simple (adjoint) group of type $G_2$. The results on the fibers $\mathfrak{F}$ of cocompact domains of discontinuity $\Omega$ in the latter three works can be derived by realizing $\mathfrak{F}$ as a sphere bundle, and then using some keen geometric intuition to identify this sphere bundle explicitly. In fact, in \cite{CTT19, ADL24, DE26a, DE26}, this geometric understanding of the fiber $\mathfrak{F}$ is further harnessed to yield descriptions of the fiber bundle $\mathfrak{F} \rightarrow M \rightarrow S$ achieved under the quotient of $\Omega$.

In these related cases, there was also one additional feature to note: the good fortune that $\mathfrak{F}$ was a readily identifiable object, namely a Stiefel manifold, or a finite quotient of one. The loss of these luxuries in the current case, the fibration of $\hat{\mathfrak{F}}$ and the capacity for its `global recognition', presents significant complications for identifying the fiber.

Our strategy is inspired in part by contrasting recent work  \cite{AMTW25}, where Alessandrini-Maloni-Tholozan-Wienhard show the fiber $\mathfrak{F}$ of a domain $\Omega$ in the space of complex Lagrangians $X=\mathrm{Lag}(\C^4)$ for $G=\Sp(4,\C)$-quasi-Hitchin representations to be homeomorphic to $\CP^2\# \overline{\CP^2}$ by using the classification of (topological) simply connected 4-manifolds. 
The group $H^2(\mathfrak{F},\Z) \cong \Z\oplus \Z$ was originally computed by Dumas-Sanders in \cite{DS20}. In \cite{AMTW25}, the authors recompute this group in order to further determine the intersection form on $H^2(\mathfrak{F},\Z)$ and then invoke the aforementioned classification. 
Recently, Hart upgraded the \cite{AMTW25} result on $\mathfrak{F}$ to `diffeomorphism' in \cite{Har26} by developing and implementing a classification of smooth circle actions on simply connected 4-manifolds. For $G=\SL(3,\C)$, he also studied the fibration of a domain $\Omega_{\rho}$ in the \emph{complex} flag manifold $\mathcal{F}_{1,2}^{\C}=\Flag(\C^3)$ for $\rho:\pi_1S \rightarrow \SL(3,\C)$ any $\iota$-Fuchsian. In that case, he showed the fiber is diffeomorphic to $(\sphere^2\times \sphere^2)\#(\sphere^2\times \sphere^2)$ also using circle actions. In particular, his approach is quite different from ours presently. It would be interesting to see if one can use circle actions to study the fibers $\mathfrak{F}$ of cocompact domains of discontinuity in other situations as well. 

Surprisingly at first sight, the case $G=\Sp(4,\C)$ and $X=\mathrm{Lag}(\C^4)$ mentioned above matches the previous patterns: $\CP^2\# \overline{\CP^2}$ is also a sphere bundle, namely the twisted product $\sphere^2\tilde{\times} \sphere^2$ over $\sphere^2\cong\CP^1$. Through the exceptional isogeny between $\Sp(4,\C)$ and $\SO(5,\C)$, which identifies $\mathrm{Lag}(\C^4)$ with the complex quadric $\mathcal{Q}^3\subset \CP^4$, we can also reinterpret this case via the base of pencil approach. 

Finally, we recall that in \cite{DS20}, Dumas-Sanders, among many other things, computed the homology of any cocompact domain $\Omega = \Omega_{\rho, I}$ in a flag manifold $\mathcal{F}=G^{\C}/P$ constructed via the \cite{KLP18}-framework when $\rho: \pi_1S \rightarrow G^{\C}$ is an Anosov representation to a complex Lie group. In particular $H_*(\Omega,\Z)$ is torsion-free, nonzero only in even dimensions, and depends only on the ideal $I$ and not on the representation $\rho$. 
Their calculation of $H_*(\Omega,\Z)$ in complete generality in terms of the balanced ideal $I$ and the cellular decomposition of the ambient flag manifold is impressive. The less surprising part of their result is 
the vanishing of the torsion for $H_*(\Omega,\Z)$, given that the ambient flag manifold $\mathcal{F}$ and the thickening $K$ admit cellular structures with cells only in even dimensions. On the other hand, there is no analogue of the \cite{DS20} result when $G$ is not complex, as there are already cases in $\SL(3,\R)$ when the homology of $\Omega_{\rho,I}$ depends on $\rho$ and not just $I$ \cite[Example 6.9]{Dav25}. Hence, our calculation of $H_*(\hat{\Omega})$ is indeed necessary.
Furthermore, as our ambient flag manifold is real, the torsion-free nature of $H_*(\hat{\Omega},\Z)\cong H_*(\hat{\mathfrak{F}},\Z)$ is rather delicate, and requires careful verification. \newpage

\subsection{Organization}

The layout of the paper is as follows:
\begin{itemize}[noitemsep]
    \item In $\S$\ref{Sec:Preliminaries}, we introduce the basic objects: the flag manifold $\mathcal{F}$, the domain of discontinuity $\Omega$, and the fibration of $\Omega$ over $\Ha^2$ via \emph{bases of pencils} as in \cite{Dav25}. 
    \item In $\S$\ref{Sec:FiberGeometry}, we produce the almost-fibration $p:\hat{\mathfrak{F}}\rightarrow \sphere^{n-1}$ and explore the consequences. This includes computing the homology via the almost-fibration. 
    \item In $\S$\ref{Sec:FiberTopology}, we recall Wall's classification and apply it to the fiber $\hat{\mathfrak{F}}$. In particular, we explain the a priori vanishing of some of his invariants and then compute $S\beta$. 
    \item In $\S$\ref{Sec:SpecialCases}, we address the special cases $n\in \{3,4,5\}$. The case $n=3$ is a sanity check. The case $n=4$ is exceptional as $\hat{\mathfrak{F}}$ is a 3-manifold with nontrivial fundamental group. When $n=5$, $\hat{\mathfrak{F}}$ is a 5-manifold and we appeal to the Smale-Barden classification. 
\end{itemize}

\subsection*{Acknowledgments}
    We thank Mason Hart for suggesting to use $\hat{\Omega}$ to verify the high connectivity of $\hat{\mathfrak{F}}$, which simplified our argument. We thank Alex Nolte and Jesus Sanchez for comments and Sara Maloni and Max Riestenberg for their support. A.T. acknowledges support from the Italian Ministry of University and Research (MUR) through the FIS3 project “Geometry and Topology of higher Teichm\"uller spaces”, grant no. FIS-2024-05142.

\section{Preliminaries}\label{Sec:Preliminaries}

In this section, we introduce the necessary preliminaries. To start, we describe the symmetric space $\X$ of $\PSL(n,\R)$ geometrically and then equivariantly embed the flag manifold $\mathcal{F}_{1,n-1}$ into the visual boundary $\vis \X$ of $\X$. We then recall the definition of the domain of discontinuity $\Omega \subset \mathcal{F}_{1,n-1}$ of interest from \cite{GW12}, along with the base of pencil approach to fiber the domain $\Omega$ over $\Ha^2$ from \cite{Dav25}. In particular, we first introduce bases of pencils in generality, then discuss this construction concretely for the flag manifold $\mathcal{F}_{1,n-1}$.

\subsection{The Symmetric Space}\label{Sec:SymmetricSpace}

Let us now introduce some basic geometry of the $\PSL(n,\R)$-symmetric space $\X$ and of non-compact symmetric spaces more generally.  

\subsubsection{A geometric model} 
We begin by recalling a standard identification of $\X$ with the following model space $\mathcal{M}_n$: 
\[ \mathcal{M}_n=\{A \in \Mat_n(\R) \mid A> 0,\, \det(A) =1, \,A=A^T\}. \] 
That is, $\mathcal{M}_n$ consists of all matrix representatives of unit volume Euclidean inner products on $\R^n$. Note that $\PSL(n,\R)$ acts on $\mathcal{M}_n$ by pushforward: $g\cdot A = g_{*}A=(g^{-1})^TAg^{-1}$, upon which $\mathcal{M}_n$ and $\X =G/K$, the homogeneous space of the maximal compact, become isomorphic $G$-spaces. 
Going forward, we may abusively write $Q \in \X$ to mean $Q \in \mathcal{M}_n$. 

Recall that given a point $x \in \X$, the subgroup $K_x =\Stab_G(x)$ is a copy of the maximal compact subgroup and induces a Cartan decomposition $\g=\frakk_x \oplus \frakp_x$, where $\frakk_x$ is the Lie subalgebra of $K_x$. We may then identify $\T_x\X \cong \frakp_x$. Equivalently, one can make the following standard identification of the tangent space: 
\[ \T_x\X = \{ \psi \in \End(\R^n) \mid \psi^{*x} = \psi\}. \] 
In other words, $\T_x\X$ consists of all $x$-symmetric endomorphisms of $\R^n$. 

We now fix the basepoint $o =I$ as the identity matrix in $\X$ using the model $\mathcal{M}_n$.
Choose once-and-for-all an $o$-orthonormal basis $(e_i)_{i=1}^n$. Correspondingly, we obtain a Cartan subalgebra $\mathfrak{a} < \mathfrak{sl}_n(\R)$ in this basis $(e_i)$, consisting of all trace-free diagonal transformations: 
\begin{align}\label{Model_CSA}
    \mathfrak{a} = \bigg\{ \diag(\lambda_1,\dots, \lambda_n) \in \End(\R^n)\mid \lambda_i \in \R, \,\sum_{i=1}^n \lambda_i = 0 \bigg\}. 
\end{align}
Observe that $\mathfrak{a}\subset \T_o\X$. 
We fix the usual Weyl chamber: 
\begin{align}\label{Model_WeylChamber}
    \mathfrak{a}^+ = \{ \diag(\lambda_1,\dots, \lambda_n) \in \mathfrak{a} \mid \lambda_1 > \lambda_2 > \cdots > \lambda_n \}. 
\end{align}
We shall always write $\mathfrak{a}^+$ for an open Weyl chamber and $\overline{\mathfrak{a}^+}$ for its closure.   
We shall also use the following standard choice of simple restricted real roots $\Delta = (\alpha_i)_{i=1}^{n-1}$ of the corresponding restricted root system of $(\g,\mathfrak{a})$.  
Namely, $\alpha_i(X) = \lambda_{i}-\lambda_{i+1}$, for $1 \leq i \leq n-1$. \medskip 

\subsubsection{Cartan projection} Let $G$ be a non-compact simple Lie group and $\X=\X_{G}$ its symmetric space. The following projection is a fundamental operation of interest. 

\begin{definition}[Cartan Projection]
The \textbf{Cartan projection} $\mu:\T\X\rightarrow \overline{\mathfrak{a}^+}$ is the map that sends $X \in \T_x\X$ to the unique element in its $G$-orbit in $\overline{\mathfrak{a}^+}$.  
\end{definition}

For $G =\PSL(n,\R)$, the Cartan projection is quite explicit. Any tangent vector $\psi \in \T_{Q}\X$ is a $Q$-self-adjoint endomorphism $\psi: \R^n \rightarrow \R^n$. Thus, $\psi$ is diagonalizable over $\R$ with real eigenvalues $\{\psi_i\}_{i=1}^n$. The Cartan projection $\mu(\psi)$ is simply the matrix $\diag(\psi_1,\psi_2,\dots, \psi_n)$, with the eigenvalues listed in monotonically decreasing order. 

\subsubsection{The visual boundary} Let $G$ be a non-compact simple Lie group. The associated symmetric space $\X_{G}$ is a complete, simply connected, non-positively curved Riemannian manifold. There are many compactifications of such spaces. We shall use the \emph{visual compactification} of $\X$, denoted $\overline{\X} = \X \sqcup \vis\X$, where we compactify by adding the \emph{visual boundary} $\vis\X$. 
Recall that $\vis \X$ consists of all geodesic rays $\gamma:[0,\infty) \rightarrow \X$ up to the equivalence relation of being at bounded Hausdorff distance. Since $G$ acts by isometries on $\X$, there is an induced action on $\vis\X$, which is also continuous. 

We briefly describe the topology on $\vis {\X}$. To this end, fix any point $x \in \X$ and $v \in \T^1_{x}\X$. We shall denote $\gamma_{x,v}:[0,\infty)\rightarrow \X$ as the unique geodesic with $\gamma'_{x,v}(0)=v$. There is a map $\iota_x: \T^1_{x}\X \rightarrow \vis\X$ given by $v\mapsto [\gamma_{x,v}]$ and the topology on $\T^1_{x}\X$ is characterized by the fact that $\iota_x$ is a homeomorphism for one (in fact, any) $x \in \X$. We shall use the terminology that $v \in \T_x\X$ \emph{points towards} $a \in \vis\X$ when $[\gamma_{x,v}]=a$. We may also use the suggestive notation $\gamma_{x,v}(\infty) = a$ to denote this equality.
We refer to \cite{BH99,Ebe96} for further details on the visual compactification.

\subsection{Bases of Pencils}\label{Sec:BaseOfPencil}
In this subsection, we recall the geometric notion of a \emph{base of pencil} and the relation between these objects and cocompact domains of discontinuity for surface group representations as in \cite{Dav25}. Most of the discussion in this subsection is general: $G$ can be any non-compact simple real Lie group. \medskip 

To start, we introduce a simple notion. 
\begin{definition}
A \textbf{pencil} $\mathcal{P}$ is a 2-plane $\mathcal{P}\subset \T_x\X$ for some $x \in \X$. 
\end{definition}

Recall that a \emph{flag manifold} $\mathcal{F}$ is a homogeneous $G$-space of the form $G/P$, where $P$ is a \emph{parabolic subgroup}. Geometrically, a parabolic subgroup $P$ is exactly the stabilizer of a point in the visual boundary $\vis \X$ of the $G$-symmetric space. 

Our goal in this subsection shall be to associate to a given pencil $\mathcal{P} \subset \T_x\X$ an expected codimension two submanifold of a given flag manifold $\mathcal{F}$ called the \emph{base} of the pencil. However, for the construction of such a base of pencil, we must first fix an equivariant embedding of a given flag manifold in the visual boundary $\vis\X$.

To define such embeddings, we recall some Lie theory.
First, recall that any parabolic subgroup $P < G$ is conjugate to a `standard' parabolic subgroup $P_{\Theta}$, where $\Theta \subseteq \Delta$ is a subset of the simple restricted real roots $\Delta$, defined as follows. 
The Lie algebra $\mathfrak{p}_{\Theta}$ consists of the sum of all root spaces of non-negative $\Theta$-height and $P_{\Theta}$ is the normalizer of $\mathfrak{p}_{\Theta}$ in $G$. 

We now describe a bijection between equivariant embeddings of flag manifolds and certain directions in the model Weyl chamber, which is essentially implicit in \cite{Dav25}. 

\begin{proposition}[Embedding Flag Manifolds]\label{Prop:CoherentEmbedding}
Let $\Theta \subseteq \Delta$ be given. Choose any direction $\tau \in \sphere \big(\,\overline{\mathfrak{a}^+} \,\big)$ such that 
\[ \Xi(\tau)= \{ \alpha \in \Delta \mid \alpha(\tau)\neq 0\} \] 
satisfies $\Xi(\tau)=\Theta$. 

Then there is a canonically associated equivariant embedding $\iota_{\tau}: G/P_{\Theta} \hookrightarrow \vis\X$. Moreover, every $G$-equivariant embedding $G/P_{\Theta}\hookrightarrow \vis \X$ obtains this form for some $\tau$. 
\end{proposition}

\begin{proof}
Let $\tau$ be given. Construct $\gamma_{\tau}:[0,\infty)\rightarrow \X$ with initial velocity $\tau$.
The stabilizer subgroup $P_{\tau} = \Stab_{G}([\gamma_{\tau}])$ is conjugate to $P_{\Theta}$.
Thus, the $G$-orbit $\mathcal{F}_{\tau}$ of $[\gamma_{\tau}]\in \vis \X$ is an embedded copy of $G/P_{\Theta}$. If we regard $G/P_{\Theta}$ as the space of all subgroups in $G$ conjugate to $P_{\Theta}$, then $\iota_{\tau}$ is described by $\iota_{\tau}(gP_{\tau}g^{-1} ) = g\cdot [\gamma_{\tau}]$. 

Conversely, it is clear that every such $G$-equivariant embedding obtains this form. 
\end{proof}

Proposition \ref{Prop:CoherentEmbedding} says there is a ($|\Theta|-1$)-dimensional collection of equivariant embeddings of $G/P_{\Theta}$ in the visual boundary. In particular, such an embedding is unique exactly when $|\Theta|=1$, or equivalently, when $P_{\Theta}$ is a maximal parabolic.  

\begin{definition}
Given $\tau \in \T^1_x\X$, we denote by $\bm{\mathcal{F}_{\tau}} \subset \vis \X$ the $G$-orbit of $[\gamma_{x,\tau}]$ in $\vis\X$. By Proposition \ref{Prop:CoherentEmbedding}, $\mathcal{F}_{\tau} \cong G/P_{\Theta(\tau)}$, where $\Theta(\tau) = \{ \alpha \in \Delta \mid \alpha(\mu(\tau)) \neq 0\}$. 
\end{definition}

We now define a \emph{base of pencil} as in \cite{Dav25}. Recall that $\gamma_{x,v}:[0,\infty) \rightarrow \X$ denotes the unique geodesic with $\gamma_{x,v}'(0)=v$. We will also use this new notation $\mathcal{F}_{\tau}$ for a chosen equivariantly embedded flag manifold. 

\begin{definition}[Base of Pencil]\label{Defn:TauBasePencil}
Let $\mathcal{P} \subset \T_{x} \mathbb{X}$ be a pencil and $\tau \in \sphere(\overline{\mathfrak{a}^+})$. Then the $\bm{\tau}$\textbf{-base of} $\boldsymbol{\mathcal{P}}$, denoted $\mathcal{B}_{\tau}(\mathcal{P})$, is given by 
\[ \mathcal{B}_{\tau}(\mathcal{P}) = \{ \gamma_{x, v}(\infty) \in \mathcal{F}_{\tau} \subset \vis \X\mid v \in \T_x\X, \; v\, \bot \,\mathcal{P} \}.\]
\end{definition}

In other words, the base $\mathcal{B}_{\tau}(\mathcal{P})$ contains all $\Theta(\tau)$-flags reached in $\mathcal{F}_{\tau}$ inside the visual boundary by traveling from $x$ via directions orthogonal to $\mathcal{P}$ in $\T_x\X$.

We emphasize: the base of pencil restricts tangent vectors in \emph{two ways}: $v \in \T_x\X$ must point towards $\mathcal{F}_{\tau}$ and $v$ must be orthogonal to the pencil $\mathcal{P}$.  

\subsubsection{Regularity of Pencils}
We now use the Cartan projection to define a notion of regularity for tangent vectors $X \in \T \X$, which will be important for the study of bases of pencils. 

\begin{definition}[$\Theta$-regularity]\label{Defn:Regular}
Let $\Theta \subseteq \Delta$ be a subset of simple roots. 
\begin{itemize}
    \item A tangent vector $X \in \T \X$ is $\mathbf{\Theta}$\textbf{-regular} when $\alpha(\mu(X)) \neq 0$ for all $\alpha \in \Theta$. 
    \item A pencil $\mathcal{P} \subset \T_x\X$ is $\mathbf{\Theta}$\textbf{-regular} when every element $\psi \in \mathcal{P}$ is $\Theta$-regular. 
    \item An immersion $u:M \rightarrow \X$ is $\mathbf{\Theta}$\textbf{-regular} when the tangent pencils $du(\T_pM)$ are each $\Theta$-regular for all $p \in M$. 
    \item An $\sllie_2\R$-subalgebra $\mathfrak{s} \subset \g$ is $\mathbf{\Theta}$\textbf{-regular} when the inclusion $\iota_{\mathfrak{s}}: \Ha^2_{\mathfrak{s}} \hookrightarrow \X$ of the sub-symmetric space of the corresponding analytic subgroup is $\Theta$-regular.  
\end{itemize}
\end{definition}

We shall be particularly interested in $\Delta$-regular pencils in the $\PSL(n,\R)$-symmetric space, which are easy to characterize.

\begin{proposition}[$\Delta$-regularity, concretely]
\label{Prop:DeltaRegularity}
Let $\X$ be the $\PSL(n,\R)$-symmetric space. 
A tangent vector $\psi \in \T_Q\X$ is $\Delta$-regular if and only if $\psi: \R^n \rightarrow \R^n$ has distinct eigenvalues.
\end{proposition}

\begin{proof}
The vector $\psi$ is $\Delta$-regular if and only if its Cartan projection $\mu(\psi) = \diag(\psi_1,\dots, \psi_n)$ has distinct eigenvalues. However, as we have already noted in $\S$\ref{Sec:SymmetricSpace}, the entries $\psi_i$ are exactly the eigenvalues of $\psi$.  
\end{proof} 

\begin{remark}[Regular rank one sub-symmetric spaces]\label{Remk:RegularRankOne}
Let $u: \Ha^2 \rightarrow \X_G$ be a totally geodesic embedding. Fix any nonzero vector $X \in \T \Ha^2$. Then $u$ is $\Theta$-regular if and only if $du(X)$ is $\Theta$-regular. Indeed, the Cartan projection sends the whole tangent bundle $\T\Ha^2$ to a single ray: $\mu( du(\T\Ha^2))= \R_{+}\{\tau\} \subset \overline{\mathfrak{a}^+}$, for some $\tau \in \sphere(\overline{\mathfrak{a}^+})$.  
\end{remark}

As a consequence of the previous remark, to verify whether an $\sllie_2\R$-subalgebra $\mathfrak{s} = \Span_{\R} \langle E,F,H\rangle $ is $\Theta$-regular, we need only check whether the semisimple element $H$ is $\Theta$-regular. In particular, we may apply this verification procedure to the principal $\sllie_2$.  

\begin{corollary}[$\T \Ha^2_{\pr}$ is $\Delta$-regular]\label{Cor:PrincipalRegularity}
The inclusion map $\Ha^2_{\pr} \hookrightarrow \X$ is $\Delta$-regular, where $\Ha^2_{\pr} \subset \X$ is the sub-symmetric space of a principal $\PSL(2,\R)$-subgroup of $\PSL(n,\R)$. 
\end{corollary}

\begin{proof}
It is well-known by work of Kostant that a semisimple element $H \in \mathfrak{s}_{\pr}$ has distinct eigenvalues \cite{Kos59}. The result follows by Remark \ref{Remk:RegularRankOne}, since we may regard $H$ also as a tangent vector in $\T \Ha^2_{\pr}$. 
\end{proof}

\subsubsection{Nearest Point Projection}\label{Subsec:NearestPointProjection}
In \cite{Dav25}, Davalo explains the relation between bases of pencils and fibrations of certain cocompact domains of discontinuity $\Omega_{\rho}$ for surface group representations. We briefly summarize his results in the case of \emph{Fuchsian-Hitchin} representations $\rho:\pi_{1}S \rightarrow \PSL(n,\R)$.\medskip

First, we recall that these Fuchsian-Hitchin representations obtain the form $ \rho=\iota_{\pr} \circ \rho_{0}$, where $\rho_0:\pi_1S\rightarrow \PSL(2,\R)$ is Fuchsian and $\iota_{\pr}: \PSL(2,\R)\rightarrow \PSL(n,\R)$ is the unique irreducible representation up to conjugation, called \emph{principal}. 

Associated to the principal $\PSL(2,\R)$-subgroup is the sub-symmetric space $\Ha^2_{\pr}$ of the $\PSL(n,\R)$-symmetric space $\X$. Given a Fuchsian-Hitchin representation $\rho$, we obtain a \emph{uniformization}, namely a $\rho$-equivariant homeomorphism $f: \tilde{S} \rightarrow \Ha^2_{\pr}$. 
We may instead view $f$ as an embedding $f: \tilde{S} \rightarrow \mathbb{X}$ with a totally geodesic image $\Ha^2_{\pr}$. 
Crucially, Corollary \ref{Cor:PrincipalRegularity} says the map $f$ is $\Delta$-regular.

To define the domain of interest, we use the fixed equivariant embedding $\lh{n} \hookrightarrow \vis\X$ described in the following subsection. 
After fixing an arbitrary basepoint $o \in \X$, we can define a domain $\Omega_{f}$ in the flag manifold $\mathcal{F}_{1,n-1}$ using Busemann functions by 
\begin{align}\label{Omega_Busemann}
    \Omega_{f} := \{a \in \mathcal{F}_{1,n-1} \mid b_{a,o} \circ f\; \text{is proper, bounded below} \}.
\end{align}
In fact, $\Omega_f$ is independent of choice of $o$. 
There is a natural projection $\pi: \Omega_f \rightarrow \tilde{S}$ that maps $a$ to the unique point $x \in \tilde{S}$ such that $b_{a,o} \circ f$ has a critical point at $f(x)$. This critical point is unique by \cite[Lemma 7.2]{Dav25}, making $\pi$ well-defined. Geometrically, this map $\pi$ is the extension of the \emph{nearest point projection} $\X \rightarrow \Ha^2_{\pr}$ to the domain $\Omega_{f} \subset \mathcal{F}_{\tau} \subset \vis\X$, where we identify $\tilde{S}$ and $\Ha^2_{\pr}$. The domain $\Omega_f$ is a cocompact domain of discontinuity for $\pi_1S$ by \cite[Theorem 7.8]{Dav25}. 

The key properties of $\pi$ applied to our setting are summarized in the following lemma. 

\begin{lemma}[Nearest Point Projection]\label{Lem:NearestPointProjection}
Let $f: \tilde{S} \rightarrow \mathbb{X}$ be a $\rho$-equivariant embedding with totally geodesic image $\Ha_{\pr}^{2}$, where $\rho:\pi_1S\rightarrow \PSL(n,\R)$ is Fuchsian-Hitchin. Then the smooth map $\pi: \Omega_{f} \rightarrow \tilde{S}$ satisfies the following: \begin{enumerate}[noitemsep, label=(\roman*)]
    \item $\Omega_{f}$ is open in $\mathcal{F}_{1,n-1}$, 
    \item $\pi$ is a fiber bundle projection,
    \item $(\Omega_f)|_x = \mathcal{B}(\mathcal{P}_x)$, where $\mathcal{P}_x \subset \T_{f(x)}\mathbb{X}$ is the tangent pencil $ df(\T_{x}\tilde{S})$. 
\end{enumerate}
\end{lemma} 

We will recall in $\S$\ref{Subsec:ComparisonDomains} the relation between $\Omega_f$ defined here and the domain of discontinuity $\Omega_{1,n-1}$ defined in \cite{GW12}. That is, for $\rho$ Fuchsian-Hitchin $\Omega_f =\Omega_{1,n-1}$; see Theorem \ref{thm:EquivalentDomains}. 

The Stiefel manifold $V_2(\R^n)$ of orthonormal pairs in $\R^n$ is a four-fold cover of $\mathcal{F}_{1,n-1}$, as described in the next subsection. 
Later, we consider the domain $\hat{\Omega} \subset V_2(\R^n)$ obtained by lifting $\Omega_{f}$. 
Hence, we record a consequence of the nearest point projection upstairs.
\begin{corollary}[Fibering the Lifted Domain]\label{Cor:NearestPointProjUpstairs}
Let $f$ be given as in Lemma \ref{Lem:NearestPointProjection} and $\Omega_f$ be the domain \eqref{Omega_Busemann}.
Given a covering $q: V_2(\R^n) \rightarrow \lh{n}$, define $\hat{\Omega} := q^{-1}(\Omega_f)$. 
Then the map $\pi\circ q: \hat{\Omega} \rightarrow \tilde{S}$ defines a fiber bundle projection with fiber $\hat{\Omega}|_{x} =  q^{-1}(\mathcal{B}(\mathcal{P}_x))$. 
\end{corollary}

\subsection{The Flag Manifold \texorpdfstring{$\lh{n}$}{F(1,n-1)}}\label{Sec:FlagManifold}

In this section, we introduce the flag manifold $\mathcal{F}_{1,n-1}$. 
We first describe the diffeomorphism type of this space as well as its realization as a fiber bundle. After choosing an equivariant embedding $\iota_{\tau}$ of $ \mathcal{F}_{1,n-1}$ in the visual boundary $\vis \X$ of the $\PSL(n,\R)$-symmetric space $\X$, we describe the resulting tangent vectors pointing towards $\mathcal{F}_{1,n-1}$. We then write the equations for the $\tau$-\emph{base of pencil} in $\mathcal{F}_{1,n-1}$.

\subsubsection{Geometry of $\mathcal{F}_{1,n-1}$} The flag manifold $\mathcal{F}_{1,n-1}$ consists of all line-hyperplane pairs:
\[ \mathcal{F}_{1,n-1} = \{(\ell, H) \in \Gr_1(\R^n) \times \Gr_{n-1}(\R^n) \mid \ell \subset H\}.\]
We begin with a basic description of the flag manifold in terms of the tangent bundle of projective space $\RP^{n-1}$. 

\begin{proposition}[Line-Hyperplane Pairs, Geometrically]\label{Prop:LHTopologicalType}
Let $Q \in \X$ be a Euclidean inner product. Then there is an associated diffeomorphism
$F_Q:\mathbb{P}\T^1\RP^{n-1}\rightarrow \mathcal{F}_{1,n-1}$ between the projective unit tangent bundle $\mathbb{P}\T^1\RP^{n-1}$ of projective space $\RP^{n-1}$ and $\lh{n}$. 
\end{proposition}

While elementary, the proof involves seeing the Stiefel manifold $V_2(\R^n) =\T^1\sphere^{n-1}$ as a finite cover of $\lh{n}$, which will be useful later. We also wish to emphasize the dependency of these identifications on the choice of Euclidean inner product $Q$. 

\begin{proof}
We first build the Stiefel manifold of orthonormal pairs with respect to $Q$:
\[ V_2^Q(\R^n) = \{ (u,v) \in \R^n\times \R^n\mid \langle u,u\rangle_Q=1, \langle v,v\rangle_Q=1, \langle u,v\rangle_Q=0\}.\]
Note that $V_2^Q(\R^n)$ is diffeomorphic to $\T^1\sphere^{n-1}$. 

Consider the map 
\begin{align*}
    \pi: V_2^Q(\R^n)&\longrightarrow \lh{n} \\
     (u,v) & \longmapsto ([u], [v]^{\bot_Q}). 
\end{align*}
Observe that $\pi$ is an $O(Q)$-equivariant surjective submersion, with fibers given by 
\[ \pi^{-1}(\ell,H)=\{(u,v),(u,-v),(-u,v),(-u,-v)\}\]
for some pair $(u,v) \in V_2^Q(\R^n).$ 
Now, let $\Gamma < \Diff(V_2^Q(\R^n))$ be the $(\Z_2\times \Z_2)$-subgroup generated by $(u,v)\mapsto (u,-v)$ and $(u,v)\mapsto (-u,v)$. The map $\pi$ descends to a diffeomorphism $\T^1\sphere^{n-1}/\Gamma \cong \lh{n}$. However, $\mathbb{P}\T^1\RP^{n-1}$ is diffeomorphic to $\T^1\sphere^{n-1}/\Gamma$.
\end{proof}

\subsubsection{An Equivariant embedding of $\mathcal{F}_{1,n-1}$} In fact, Proposition \ref{Prop:LHTopologicalType} describes one possible equivariant embedding $\mathcal{F}_{1,n-1} \hookrightarrow \vis\X$. 
Given $f=(\ell,H)$ and $Q \in \X$, consider the 4-tuple of representatives $\pi^{-1}(f) \subset  V_2^Q(\R^n)$ for the map $\pi$ from Proposition \ref{Prop:LHTopologicalType}. For any preimage $(u,v) \in \pi^{-1}(f)$, we define the following linear map
$\phi_{u,v}:\R^n \rightarrow \R^n$: 
\begin{align}\label{PointingLH_Simple}
    \phi_{u,v}=\begin{cases}
        u &\mapsto u \\
        v &\mapsto -v \\
        \langle u,v\rangle^{\bot_Q} &\mapsto 0
    \end{cases}
\end{align}
Since $\phi$ is $Q$-symmetric, we may regard $\phi$ as a tangent vector $\phi \in \T_Q\X$. Using these tangent vectors, we introduce the following definition. 

\begin{definition}
Let $f = (\ell, H) \in \lh{n}$. 
Observe that $\phi_{u,v}$ in \eqref{PointingLH_Simple} is independent of $(u,v) \in \pi^{-1}(f)$. Hence, we may define $\phi_{Q,f} :=\phi_{u,v}$ for any $(u,v) \in \pi^{-1}(f)$. 
\end{definition}

\begin{lemma}[Pointing Towards $\lh{n}$]\label{Lem:LHEmbedding}
Consider the map $\Psi: \lh{n} \times \X\rightarrow \T\X$ given by $(f, Q) \mapsto \phi_{Q,f}$. Then 
$\Psi$ is an $\PSL(n,\R)$-equivariant map satisfying 
\[ \Stab_{\PSL(n,\R)}([\gamma_{Q, \phi_{Q,f}}])= \Stab_{\PSL(n,\R)}(f). \]
\end{lemma}

\begin{proof}
The map is evidently $\PSL(n,\R)$-equivariant. Now, it suffices to prove the claim for $o \in \X$ the model basepoint $o =I$, the identity matrix. For our chosen  Weyl chamber $\overline{\mathfrak{a}^+}$, we consider the model line-hyperplane pair
$f= ([e_1], [e_n]^\bot)$, which has 
$\phi_{o,f}= \diag(1,0,0,\dots, 0, -1)$. 
The stabilizer of the geodesic ray $[\gamma_{o, \phi_{o,f}}]$ is the stabilizer of the flag $f$, which completes the proof. 
\end{proof}

Lemma \ref{Lem:LHEmbedding} makes explicit the embedding $\lh{n}\hookrightarrow \vis\X$. 
In short, we take $\tau$ to be:
\begin{align}\label{SpecialTau}
    \tau = \diag(1,0,0,\dots,-1) \in \overline{\mathfrak{a}^+} \subset \T_o\X,
\end{align}
and we identify its $\PSL(n,\R)$-orbit $\mathcal{F}_{\tau}$ in the visual boundary with the flag manifold $\lh{n}$ as in Proposition \ref{Prop:CoherentEmbedding}. Moreover, the lemma provides a compass: given any point $Q\in \X$ and desired flag $f \in \lh{n}$ at infinity, we obtain the unique direction $\phi_{Q,f}$, up to positive scalars, that points towards $f$ in the visual boundary.

\begin{remark}
Going forward, when we regard $f \in \lh{n}$ as an element of $\vis \X$, we shall mean 
under the embedding from Lemma \ref{Lem:LHEmbedding}. Moreover, the only base of pencil we shall consider is $\mathcal{B}_{\tau}(\mathcal{P})$ for $\tau$ from \eqref{SpecialTau}. 
\end{remark}

\subsubsection{Bases of Pencil in $\lh{n}$} 
We now consider bases of pencils in the line-hyperplane space. 

We first describe a sufficient condition for a base of pencil in $\lh{n}$ to be a smooth submanifold. 
The relation between regularity of pencils and smoothness of the associated bases of pencils is described in greater generality in \cite{Dav25}.

\begin{lemma}\label{Lem:SmoothBase}
Let $\mathcal{P}\subset \T_Q\X$ be a $\Delta$-regular pencil. Then 
\begin{enumerate}[noitemsep, label=(\roman*)]
    \item The base of pencil $\mathcal{B}(\mathcal{P}) \subset \lh{n}$ is a closed smooth codimension two submanifold.
    \item Let $\pi: V_2(\R^n) \rightarrow \lh{n}$ be a smooth $(\Z_2\times \Z_2)$-covering map. The lifted base 
    $\pi^{-1}(\mathcal{B}(\mathcal{P})) \subset V_2(\R^n)$ is a closed smooth codimension two submanifold of $V_2(\R^n)$. 
\end{enumerate}
\end{lemma}

\begin{proof}
(i) This follows from \cite[Lemma 6.7]{Dav25} and subsequent remarks. 

(ii) This is an immediate consequence of (i).
\end{proof}

We now provide our first description of the equations for a base of pencil in $\lh{n}$. 

\begin{proposition}[Quadratic Orthogonality] \label{Prop:PencilOrtho_Quadratic}
Let $\mathcal{P} \subset \T_Q\X$ be a pencil. Then a tangent vector $\phi_{u,v} \in \T_Q\X$ of the form \eqref{PointingLH_Simple} is orthogonal to $\mathcal{P}$ if and only if 
\begin{align}\label{OrthoCondition_Quadratic}
 \langle \psi(u),u\rangle_Q =\langle \psi(v),v\rangle_Q, \qquad \forall \psi \in \mathcal{P}.
\end{align}
\end{proposition}

\begin{proof}
Under the identification of $\X_{\PSL(n,\R)}$ with the model space $\mathcal{M}_n$, the Riemannian metric on $\X$ induced by the Killing form identifies up to a multiplicative constant $c\neq 0$ with the metric $g$ on $\mathcal{M}_n$ given by $g_Q(\phi,\psi) = \tr(\psi^{*Q}\circ \phi)$. 

Hence, for any element $\psi \in \mathcal{P}$ and tangent vector $\phi=\phi_{Q,f}$, written in the form \eqref{PointingLH_Simple},
\[ c\langle \phi_{Q,f}, \psi\rangle_{\X} = \langle (\psi^{*Q} \circ \phi)(u),u\rangle_Q + \langle(\psi^{*Q} \circ \phi)(v),v\rangle_Q
= \langle u,\psi(u)\rangle_Q -\langle v, \psi(v)\rangle_Q .
\]
The claim follows. 
\end{proof}

Let us make a trivial but important remark: the equations \eqref{OrthoCondition_Quadratic} are quadratic in $u$ and $v$. For this reason, it is essential to change perspective in search of linearity. 
Observe that by setting 
$x = \frac{1}{\sqrt{2}}(u+v)$ and $y = \frac{1}{\sqrt{2}}(u-v)$, the tangent vector $\phi_{Q,f}$ from \eqref{PointingLH_Simple} can be equivalently described by  
\begin{align}\label{PointingLH_Better}
    \phi_{Q,f}= \begin{cases}
        x &\mapsto y,\\
        y &\mapsto x, \\
        \langle x,y\rangle^{\bot_Q} &\mapsto 0.
    \end{cases}
\end{align}

This change of variables $(u,v)\mapsto (x,y)$ leads to a different covering map $V_2(\R^n) \rightarrow \lh{n}$. 

\begin{proposition}[Pointing Towards Line-Hyperplane, Variant]\label{Prop:LHEmbedding_Better}
Fix $Q \in \X$. The map $G_Q:V_2^Q(\R^n) \rightarrow \lh{n}$ by $(x,y)\mapsto ([x+y], [x-y]^{\bot_Q})$ 
is a $(\Z_2\times \Z_2)$-covering map. 
Moreover, the tangent vector $\phi_{Q,f} \in \T_Q\X$ points towards the flag $G_Q(f)$. 
\end{proposition}

\begin{remark}
Note that the preimage of a line-hyperplane pair $f=(\ell,H) \in \lh{n}$ under the map $G_Q$ instead obtains the form 
$G_Q^{-1}(f)=\{(x,y),(y,x),(-x,-y),(-y,-x)\}$. 
\end{remark}

We now obtain the desired linear equations for the base of pencil. 

\begin{proposition}[Linear Orthogonality] \label{Prop:PencilOrtho_Linear}
Let $\mathcal{P} \subset \T_Q\X$ be a pencil and $\phi_{Q,f}$ a vector pointing towards $\lh{n}$, written as in \eqref{PointingLH_Better}. Then $\phi_{Q,f}$ is orthogonal to $\mathcal{P}$ if and only if 
\begin{align}\label{OrthoCondition_Linear}
    \langle y, \psi(x) \rangle_Q = 0, \qquad \forall \psi \in \mathcal{P}.
\end{align}
\end{proposition}

\begin{proof}
We proceed as in Proposition \ref{Prop:PencilOrtho_Quadratic}. 
Take $\psi \in \mathcal{P}$ and write a tangent vector $\phi =\phi_{Q,f}$ as in \eqref{PointingLH_Better}. We compute the pairing: again for some constant $c\neq 0$, 
\[ c\langle \phi, \psi\rangle_{\X} = \langle (\psi^{*Q} \circ \phi)(x),x\rangle_Q + \langle( \psi^{*Q} \circ \phi)(y),y\rangle_Q
= \langle y,\psi(x)\rangle_Q +\langle x, \psi(y)\rangle_Q = 2\langle x,\psi(y)\rangle_Q.\]
The claim follows. 
\end{proof}

In particular, the \emph{Structure Lemma} \ref{Lem:StructureLemma} we shall prove for the base of pencil of interest requires us to look for pairs $(x,y)$ as in \eqref{PointingLH_Better}, rather than pairs $(u,v)$ as in \eqref{PointingLH_Simple}.

\subsection{Comparison of Domains}\label{Subsec:ComparisonDomains}

We now introduce the \cite{GW12}-construction of a cocompact domain of discontinuity $\Omega_{1,n-1} \subset \lh{n}$ for Hitchin representations $\rho$. We recall the equivalence with the domain of discontinuity introduced in $\S$\ref{Subsec:NearestPointProjection} with Busemann functions. In the case of even ambient dimension $n=2m$, we show the domain $\Omega_{1,2m-1}$ is distinct from another domain in $\lh{2m}$ defined by pullback of a domain in $\RP^{2m-1}$. 

\subsubsection{The Tits Metric Thickening Domain}\label{Sec:TitsDomain}

A key property of Hitchin representations is that they admit \emph{limit maps}. We briefly recall the relevant features. 

The surface group $\pi_1S$, for $S=S_g$ closed of genus $g \geq2$, is hyperbolic and thus has a well-defined \emph{Gromov boundary} $\partial \pi_1S$, which is topologically a circle and carries a natural $\pi_1S$-action. We shall need also the \emph{full flag manifold} $\Flag(\R^n) $ consisting of all full flags 
\[ F^{\bullet} = \big[ F^1 \subset F^2 \subset \cdots \subset F^{n-1} ] ,\]
where $F^i$ is an $i$-plane in $\R^n$. 

\begin{proposition}\label{Prop:HitchinLimitMap}
Let $\rho:\pi_1S\rightarrow \PSL(n,\R)$ be Hitchin. Then there exists a unique $\rho$-equivariant continuous map $\xi^{\Delta}: \partial \pi_1S\rightarrow \Flag(\R^n)$ that is transverse and dynamics-preserving. 
\end{proposition}

We refer to \cite{GGKW17} for the details regarding the \emph{dynamics-preserving} condition. Here, we call two full flags $F_{1}^{\bullet}, F_{2}^{\bullet} \in \Flag(\R^n)$ \emph{transverse} exactly when $F_1^{i} + F_2^{j}=\R^n$ for all indices $i+j=n$. Equivalently, from a more abstract perspective, a pair $(F_1^{\bullet},F_{2}^{\bullet})$ is transverse exactly when it lies in the unique open $\PSL(n,\R)$-orbit in $\Flag(\R^n)\times \Flag(\R^n)$. 

Now, out of the data of this limit map, \cite{GW12} defined the following open cocompact domain of discontinuity $\Omega_{1,n-1}= \Omega_{1,n-1}(\rho)$ for $\rho$, which was reinterpreted later by \cite{KLP18} in a more general framework: 
\begin{align}\label{Omega_Thick}
    \Omega_{1,n-1} = \lh{n} \backslash \bigcup_{x \in \partial \pi_1S} K_{\xi^{\Delta}(x)},
\end{align}
where here the \emph{thickening} $K_{F^{\bullet}} \subset \lh{n}$ of a full flag $F^{\bullet}= (F^i)_{i=1}^{n-1}$ is as follows:  
\[ K_{F^{\bullet} } = \{ (\ell, H) \mid \exists i: \ell \subset F^i \subset H \}.\]
From the \cite{KLP18} perspective, the thickening $K_{F^{\bullet}}$ is exactly the intersection of $\lh{n} \subset \vis\X$ with a $\frac{\pi}{2}$-neighborhood of $F^{\bullet}$ in $\vis\X$ with respect to the Tits angle metric. 

In our context, \cite[Theorem 7.11]{Dav25} says that the domain \eqref{Omega_Busemann} defined via Busemann functions is the same as this domain of discontinuity \eqref{Omega_Thick} defined via limit maps in the Fuchsian-Hitchin case. 

\begin{theorem}[Equivalence of Domains]\label{thm:EquivalentDomains}
Let $\rho: \pi_1S \rightarrow \PSL(n,\R)$ be Fuchsian-Hitchin, $\Omega_f$ be the domain \eqref{Omega_Busemann} and $\Omega_{1,n-1}$ the domain \eqref{Omega_Thick}. Then $\Omega_f=\Omega_{1,n-1}$. 
\end{theorem}

The following fact of topological invariance will be useful. 
\begin{corollary}[{{\cite[Theorem 2.10]{AMTW25}}}]
Let $\rho:\pi_1S \rightarrow \PSL(n,\R)$ be any Hitchin representation for $n \geq 4$ and let $\Omega_{\rho}$ be the \cite{GW12} domain in $\mathcal{F}_{1,n-1}$. Then 
\begin{enumerate}[noitemsep, label=(\roman*)]
    \item The diffeomorphism type of the domain $\Omega_{\rho}$ is independent of $\rho$. 
    \item The diffeomorphism type of the quotient $M_{\rho}=\rho(\pi_1S)\backslash  \Omega_{\rho}$ is independent of $\rho$. 
\end{enumerate}
\end{corollary}

Thus, it is sufficient to understand the differential topology of the domain for Fuchsian-Hitchin representations, and by Theorem \ref{thm:EquivalentDomains}, we can use $\Omega_f$ in this case. 

We note a similar result for the lifted domain in the Stiefel manifold. 
\begin{corollary}
Fix a covering $q: V_2(\R^n) \rightarrow \mathcal{F}_{1,n-1}$. 
The diffeomorphism type of the lifted domain $\hat{\Omega}_{\rho} = q^{-1}(\Omega_{\rho})$ is independent of the Hitchin representation $\rho: \pi_1S \rightarrow \PSL(n,\R)$. 
\end{corollary}

\begin{proof}
We first recall that, although the statement of \cite[Theorem~2.10]{AMTW25} only asserts the existence of a smooth equivariant diffeomorphism between the domains associated with two Anosov representations in the same connected component, the proof of the underlying deformation theorem gives a local continuous family; see in particular the proof of \cite[Theorem~9.12]{GW12}. We may thus take $\varphi_t:\Omega_0\rightarrow \Omega_t$ to be a continuously varying $(\rho_0,\rho_t)$-equivariant diffeomorphism. 

We now prove the corollary. Let $j_t:\Omega_t\hookrightarrow \mathcal{F}_{1,n-1}$ denote the inclusion. The map
\[
    H:\Omega_0\times[0,1]\longrightarrow \mathcal{F}_{1,n-1},
    \qquad
    H(x,t)=j_t\bigl(\varphi_t(x)\bigr),
\]
is continuous and satisfies
\[
    H(\,\cdot\,,0)=j_0,
    \qquad
    H(\,\cdot\,,1)=j_1\circ\varphi_1.
\]
In other words, the inclusion maps $j_i:\Omega_i \rightarrow \mathcal{F}_{1,n-1}$ are isotopic. Thus, the domains $\hat{\Omega}_i$, which are the total space of the fiber bundles $j_i^*q$, are diffeomorphic.
\end{proof}

\subsubsection{The Pullback Domain}

In the case of ambient dimension $n=2m$, we introduce another domain via limit maps, namely the pullback of a domain of discontinuity in projective space. We quickly clarify that this domain is distinct from the domain $\Omega_{1,n-1}$. \medskip

Unlike in odd ambient dimensions \cite{Ste23}, Hitchin representations $\rho: \pi_1S\rightarrow \PSL(2m,\R)$ admit a cocompact domain of discontinuity in projective space, denoted $\Omega_{1}\subset \RP^{2m-1}$. Under the \cite{KLP18}-framework, this domain can also be described by removing a certain thickening of limit sets. Here, we denote $\xi^m: \partial \pi_1S \rightarrow \Gr_m(\R^{2m})$ as the projection of the limit map $\xi^{\Delta}$ of $\rho$ from Proposition \ref{Prop:HitchinLimitMap} to the Grassmannian $\Gr_m(\R^{2m})$. 

The domain in projective space is then given by: 
\[ \Omega_{1} = \RP^{2m-1} \setminus \bigcup_{x\in \partial \pi_1S} K_{\xi^m(x)},\]
where here the \emph{thickening} $K_{F^m} \subset \RP^{2m-1}$ of an $m$-plane is similar to the prior thickenings: 
\[ K_{F^m} = \{ \ell \in \RP^{2m-1} \mid \ell \subset F^m\}.\]

Let us denote $\pr_1: \lh{2m} \rightarrow \RP^{2m-1}$ as the natural projection $(\ell, H) \mapsto \ell$. By pullback, we obtain an open and cocompact domain of proper discontinuity in $\mathcal{F}_{1,2m-1}$ by 
$\pr_1^{-1}(\Omega_1)$. 

We now show that the two domains $\Omega_{1,n-1}$ and $\pr_1^{-1}(\Omega_1)$ are distinct. While this result is well known to experts, we provide a short elementary and geometric proof for the convenience of the reader (circumventing a discussion of \emph{balanced ideals} from \cite{KLP18}). In particular, the result of \cite{ADL24} about the fibers of $\Omega_1$ for Hitchin representations in projective space does not help us to find the fibers of $\Omega_{1,2m-1}$ in this case.

We proceed with an argument using only transversality that immediately implies the desired result. 

\begin{proposition}
Let $m \geq 2$ be an integer and $\xi:\sphere^1 \rightarrow \Flag(\R^{2m})$ be a continuous and transverse map. Then the open domains $\pr_1^{-1}(\Omega_1(\xi))$ and $\Omega_{1,n-1}(\xi)$ are not equal. 
\end{proposition}

\begin{proof}
Fix any point $x_0 \in \sphere^1$. Let us define 
$A : = \mathbb{P}(\xi^{m+1}(x_0))\setminus \mathbb{P}(\xi^{m}(x_0))$. We first note that if $A \cap \Omega_1 \neq \emptyset $, then the proof is complete. Indeed, if $\ell \in A\cap \Omega_1$, then for
any hyperplane $H$ satisfying $\xi^{m+1}(x_0) \subset H$, we have $(\ell, H) \in \pr_1^{-1}(\Omega_1(\xi))$, yet $(\ell, H) \notin \Omega_{1,n-1}(\xi)$. 

Now, we prove that $A \cap \Omega_1 \neq \emptyset $. 
For any point $x \in \sphere^1$, note that 
$\xi^m(x_0) \cap \xi^{m}(x) =\{0\}$ by transversality. 
Hence, $\dim (\xi^{m+1}(x_0) \cap \xi^{m}(x_0))= 1$. 
Let us denote $K_1 := \RP^{2m-1}\setminus \Omega_1(\xi)$ as the thickening of $\mathrm{im}(\xi^m)$ in $\RP^{2m-1}$. 
Thus, 
\[ A \cap K_1  = \bigcup_{x \in \sphere^1 \setminus \{x_0\}} \mathbb{P}\big(\xi^{m+1}(x_0)\cap \xi^{m}(x)\big)\]
is one-dimensional. Hence, $A\not \subset K_1$, which means $A\cap \Omega_1 \neq \emptyset$.  
\end{proof}

\begin{corollary}[Comparison of Domains]
Let $\rho:\pi_1S \rightarrow \PSL(2m,\R)$ be $\Delta$-Anosov. Then the domains $\pr_1^{-1}(\Omega_{1}(\rho))$ and $\Omega_{1,n-1}(\rho)$ are distinct. In particular, this holds for $\rho$ Hitchin. 
\end{corollary}

\section{Geometry of the Fiber}\label{Sec:FiberGeometry}

In this section, we begin our study of the domain of discontinuity $\Omega_{\rho} \subset \lh{n}$ for Hitchin representations $\rho$. By topological invariance, we may consider only the case of Fuchsian-Hitchin representations. We first lift the domain to the Stiefel manifold $V_{2}(\R^n)$. Corollary \ref{Cor:NearestPointProjUpstairs} tells us that the lift $\hat{\Omega}$ fibers over $\Ha^2$ with fiber the base of pencil $\hat{\mathcal{B}}(\mathcal{P}) \subset V_2(\R^n)$, where $\mathcal{P} = \T_{x}\Ha^{2}_{pr}$ is the tangent pencil of a principal $\PSL(2,\R)$-subgroup. 
In this section, we prove the following facts about $\hat{\mathcal{B}}(\mathcal{P})$ that will be useful for everything to follow: 
\begin{itemize}[noitemsep]
    \item $\T \hat{\mathcal{B}}(\mathcal{P})$ is stably trivial. 
    \item $\hat{\mathcal{B}}(\mathcal{P})$ is highly connected. 
    \item $H_*(\hat{\mathcal{B}}(\mathcal{P}),\Z)$ is torsion-free.
\end{itemize}
The key to everything is the ``almost-fibration'' introduced in $\S$\ref{Sec:StructureLemma}.

\subsection{The Structure Lemma}\label{Sec:StructureLemma}

In this section, we prove a key lemma that shows the base of pencil $\hat{\mathcal{B}}(\mathcal{P})$ in $V_2(\R^n)$ almost fibers over the sphere $\sphere^{n-1}$. In particular, the base of pencil is naturally the union of two honest fiber bundles: one with base a disjoint union of $r(n)$ circles $S_{sing} \subseteq \sphere^{n-1}$ and one over the complementary set $S_{gen} =\sphere^{n-1}\setminus S_{sing}$. 

After introducing the Structure Lemma, we must perform a tedious calculation of the number of connected components $r(n)$ of the singular locus. \medskip

Now, here is the main lemma of interest. For the result, recall the $(\Z_2 \times \Z_2)$-covering map $G_Q: V_2(\R^n) \rightarrow \lh{n}$ from Proposition \ref{Prop:LHEmbedding_Better}.

\begin{lemma}[Structure Lemma]\label{Lem:StructureLemma}
Suppose $n \geq 4$. Let $\mathcal{P}\subset \T_Q\X$ be the $\Delta$-regular pencil 
$\mathcal{P} = \T_Q\Ha^2_{\pr}$ and define the lifted base of pencil $\hat{\mathcal{B}}(\mathcal{P}) :=G_Q^{-1}(\mathcal{B}(\mathcal{P}))$ in $V_2(\R^n)$. 

The projection $p:\hat{\mathcal{B}}(\mathcal{P})\rightarrow \sphere^{n-1}$ given by $(x,y)\mapsto x$ is a smooth surjective map. The total space $\hat{\mathcal{B}}(\mathcal{P})$ and base $\sphere^{n-1}$ each split into two pieces as 
\[ \hat{\mathcal{B}}(\mathcal{P}) = \hat{\mathcal{B}}(\mathcal{P})_{gen} \sqcup \hat{\mathcal{B}}(\mathcal{P})_{sing} \;\;\text{and}\;\;\;
    \sphere^{n-1} = S_{gen}\sqcup S_{sing},\]
where $\hat{\mathcal{B}}(\mathcal{P})_{gen} = p^{-1}(S_{gen})$ and $\hat{\mathcal{B}}(\mathcal{P})_{sing} = p^{-1}(S_{sing})$, 
such that the following hold: 
\begin{enumerate}[noitemsep, label=(\roman*)]
    \item $S_{sing}$ is a disjoint union of $r(n) \geq 1$ circles, 
    \item $p:\hat{\mathcal{B}}(\mathcal{P})_{gen} \rightarrow S_{gen}$ is an orientable $\sphere^{n-4}$-fiber bundle.  
    \item $p:\hat{\mathcal{B}}(\mathcal{P})_{sing} \rightarrow S_{sing}$ is an $\sphere^{n-3}$-fiber bundle. 
\end{enumerate}
\end{lemma}

\begin{proof}
All orthogonal projections going forward are made with respect to $Q \in \X$. In particular, for $x \in \sphere^{n-1}$, we have an orthogonal projection $\pi_{x^\bot}: \R^n \rightarrow x^\bot \cong\T_x\sphere^{n-1}$. 

The proof hinges upon defining a splitting 
of $\T \sphere^{n-1}$ by   
\[ \T \sphere^{n-1} = \mathcal{R} \oplus \mathcal{R}^\bot, \]
where 
\[ \mathcal{R}_x := \pi_{x^\bot}(\mathcal{P}(x))=\{\pi_{x^\bot}(\psi(x)) \mid  \psi \in \mathcal{P} \}.\]
We then define $\mathcal{R}^\bot$ as the orthogonal complement of $\mathcal{R}$ in $x^\bot$. While each of $\mathcal{R}, \mathcal{R}^\bot$ are not vector bundles over the whole sphere $\sphere^{n-1}$, as their dimensions are non-constant, we shall see they \emph{are} vector bundles over smaller loci $S_{sing}$ and $S_{gen}$, to be defined shortly.

Here is the key observation: the base of pencil $\hat{\mathcal{B}}(\mathcal{P})$ identifies with the following space 
\[ \sphere(\mathcal{R}^\bot) = \{(x,y) \in \sphere^{n-1}\times \sphere^{n-1} \mid y \in \mathcal{R}^\bot|_{x}\}. \] 
Indeed, a pair $(x,y) \in V_2^Q(\R^n)$ gives $\phi_{x,y} \in \T_Q\X$ as in \eqref{PointingLH_Better}. By definition $(x,y) \in \hat{\mathcal{B}}(\mathcal{P})$ if and only if $\phi_{x,y} \in \mathcal{B}(\mathcal{P})$, which happens if and only if 
$(x,y) \in \sphere(\mathcal{R}^\bot)$ by Proposition \ref{Prop:PencilOrtho_Linear}. Now, write $\mathcal{P} = \Span \{\psi_1,\psi_2\}$ and we make the following claim. \medskip 

\textbf{Claim 1}: For any $x \in \sphere^{n-1}$,  we have $\dim \Span \{\psi_1(x),\psi_2(x), x\} \geq 2$. \medskip 

Suppose for contradiction that $x \in \sphere^{n-1}$ satisfies
$\dim \Span\{ \psi_1(x), \psi_2(x),x\}=1$. Note that this means $\Span \{\psi_1(x), \psi_2(x) \} \subseteq \R\{x\}$, so that $\R\{x\}$ is a mutual eigenline of $\psi_1$ and $\psi_2$. Observe that for $n \geq 4$, the pencil $\mathcal{P} = \T_o\Ha^2_{\pr}$ has no such eigenline. Indeed, for $\{E,F,H\}$ an $\sllie_2$-triple of the principal $\sllie_2\R$, the pencil $\mathcal{P}$ is $\mathcal{P} = \Span \langle E+F, H \rangle$, the matrix $H$ is diagonal with distinct eigenvalues, and $E+F$ has no eigenlines in the standard basis. $\square_{claim}$ \medskip 

Hence, we consider 
\begin{align}
    S_{gen} &\coloneq \{ x \in \sphere^{n-1} \mid \dim \mathcal{R}|_x = 2\}, \\
    S_{sing} & \coloneq \{x \in \sphere^{n-1}\mid \dim \mathcal{R}|_x=1\},
\end{align}
which yields a partition $\sphere^{n-1} = S_{gen} \sqcup S_{sing}$. 
Note that for $x \in S_{gen}$, we have $p^{-1}(x) \cong \sphere^{n-4}$ and for $x \in S_{sing}$, we have $p^{-1}(x)\cong \sphere^{n-3}$, a sphere of one higher dimension. 
Next, we make a small observation: 
$x \in S_{sing}$ if and only if $x$ is a unit eigenvector of some (nonzero) $\psi \in \mathcal{P}$. 
Indeed, $\dim \mathcal{P}(x)/\R\{x\}=\dim \mathcal{R}|_x=1$ means $\psi(x) \in \R\{x\}$ for some $\psi \in \mathcal{P}\setminus \{0\}$. \medskip 

We now prove (i). Observe that since the pencil $\mathcal{P}$ has no common eigenlines, if $\psi_1, \psi_2 \in \mathcal{P}$ are linearly independent and $E_i$ are any nonzero eigenvectors of $\psi_i$, for $i \in \{1,2\}$, then $E_1 \neq E_2$. We now define the following space of unit tangent vector--eigenvector pairs
\[ \hat{\mathbb{E}}(\mathcal{P}) := \{ (\psi, x) \in \sphere(\mathcal{P})\times \sphere^{n-1} \mid \exists \lambda \in \R,\, \psi x =\lambda x \}\]
There are natural projections $\pr_1: \hat{\mathbb{E}}(\mathcal{P}) \rightarrow \sphere(\mathcal{P})$ and $\pr_2:\hat{\mathbb{E}}(\mathcal{P})\rightarrow \sphere^{n-1}$ to each factor. Note that $S_{sing}$ is exactly the image of $\pr_2$. Now, since each element $\psi \in \mathcal{P}$ is diagonalizable with distinct eigenvalues, the map $\pr_1$ defines a $2n$-fold cover of the circle $\sphere(\mathcal{P})\cong \sphere^1$. Thus,  $\hat{\mathbb{E}}(\mathcal{P})$ is a disjoint union of circles. On the other hand, by the distinctness of eigenvectors, for any $x \in S_{sing}$, we see 
$\pr_2^{-1}(x) = \{ (\psi,x), (-\psi,x)\}$ for some $\psi \in \sphere(\mathcal{P})$. Hence, consider
$\hat{\mathbb{E}}(\mathcal{P})/(-\id, \id)$, which is a disjoint union of some nonzero number $r(n)$ of circles, and note that  
$\pr_2$ descends to a homeomorphism $\pr_2:\hat{\mathbb{E}}(\mathcal{P})/(-\id, \id) \rightarrow S_{sing}$. The claim (i) follows. \medskip 

We now head towards a proof of (iii). To this end, we make the following claim: \medskip 

\textbf{Claim 3}: For $n \geq 4$, the restriction $\mathcal{R}^\bot \rightarrow S_{sing}$ defines a vector bundle of rank $n-2$. \medskip 

Note that $\T\sphere^{n-1}|_{S_{sing}} \cong S_{sing} \times \R^{n-1}$ is a trivial vector bundle. Let us decompose the singular locus as $S_{sing} = \bigsqcup_{i=1}^{r(n)} C_i$, where each component $C_i$ is a circle. Fix any index $1\leq i\leq r(n)$. We consider the restriction of $\mathcal{R} \rightarrow S_{sing}$ to $C_i$, which defines a real line bundle over $C_i$. It follows that $\mathcal{R}^\bot \rightarrow S_{sing}$ is a vector bundle of rank $n-2$ and thus $\sphere(\mathcal{R}^\bot) \rightarrow S_{sing}$ is an $\sphere^{n-3}$-bundle, whose total space identifies with $\hat{\mathcal{B}}(\mathcal{P})_{sing}$. Thus, the claim and (iii) hold. \medskip 

We now prove (ii). To this end, we make the following claim: \medskip

\textbf{Claim 2}: $\mathcal{R}^\bot|_{S_{gen}} \rightarrow S_{gen}$ is a stably trivial vector bundle of rank $n-3$. Hence, the restriction $\sphere(\mathcal{R}^\bot)|_{S_{gen}} \rightarrow S_{gen}$ is an orientable $\sphere^{n-4}$-fiber bundle.  \medskip 

We observe that the rank two vector bundle $\mathcal{R} \rightarrow S_{gen}$ is trivial. Indeed, from any basis $(\psi_1, \psi_2)$ for $\mathcal{P}$, we obtain sections $s_i \in \Omega^0(S_{gen}, \mathcal{R}|_{S_{gen}})$ by $s_i(u) = \pi_{u^{\perp}}(\psi_i(u))$. These sections are globally linearly independent by definition of $S_{gen}$. Thus, $\mathcal{R}^\bot|_{S_{gen}} \rightarrow S_{gen}$ is a rank $n-3$ vector bundle. However, $\T\sphere^{n-1}|_{S_{gen}}$ is trivial, since $S_{sing} \neq \emptyset$. Thus, denoting by $\varepsilon_\R^i$ a trivial vector bundle of rank $i$ (over $S_{gen}$), we see 
\[ \varepsilon^{n-1}_{\R} \cong \T\sphere^{n-1}|_{S_{gen}} \cong \mathcal{R}\oplus \mathcal{R}^\bot \cong \varepsilon^2_{\R} \oplus \mathcal{R}^\bot, \]
so $\mathcal{R}^\bot$ is stably trivial. As the orthogonal complement of an orientable sub-bundle of an orientable bundle, $\mathcal{R}^\bot$ is orientable. Hence, $\sphere(\mathcal{R}^\bot)|_{S_{gen}} \rightarrow S_{gen}$ is an orientable $\sphere^{n-4}$-bundle. 
\end{proof}

We now proceed towards counting the number $r(n)$ of singular circles in Lemma \ref{Lem:StructureLemma}. 
To this end, we perform a small deformation to simplify the pencil. 
The desired deformation of pencils has a special structure. Thus, we recall that a \emph{Jacobi matrix} is a square matrix whose nonzero entries appear only on the sub-diagonal, main diagonal, and super-diagonal. 
We recall an elementary linear algebra fact about such matrices. 

\begin{lemma}[Symmetric Jacobi matrices are $\Delta$-regular]\label{lem:JacobiMatrices} A symmetric Jacobi matrix $T$ with nonzero off-diagonal entries is diagonalizable with distinct eigenvalues.
\end{lemma}
\begin{proof} By the spectral theorem, $T$ is diagonalizable. Thus, we need only prove that all eigenvalues have multiplicity one. Suppose, for contradiction, that $\lambda$ is an eigenvalue with multiplicity at least $2$. Then take two linearly independent $\lambda$-eigenvectors $v$ and $w$. One can then find a linear combination $u=c_1v+c_2w$ such that $u=(u_1, \dots, u_n)$ has vanishing first entry $u_1$. Now, by hypothesis the matrix $T$ obtains the following form
\[ T= \begin{pmatrix} a_1 & b_1 &  & & &\\
                      b_1 & a_2 & b_2& & &\\
                       & b_2 & a_3 & \ddots& & \\
                     & & \ddots & \ddots & &\\
                     & & & & &b_{n-1}\\
                     & & & & b_{n-1}&a_n\\
                      
\end{pmatrix},\]
where the entries $b_i$ are nonzero for $1\leq i\leq n-1$. 
The eigenvector equation $Tu=\lambda u$ gives the following system of linear equations:
\[
    \begin{cases}
        a_1u_1+b_1u_2=\lambda u_1, \\
        b_{k-1}u_{k-1}+a_ku_k+b_ku_{k+1}=\lambda u_k, \ \ \ \ \;(2 \leq k \leq n-1) \\
        b_{n-1}u_{n-1}+a_nu_n=\lambda u_n \ .
    \end{cases}
\]
We can easily observe that $u_1=0$ implies $u_k=0$ for all $k\in \{1, \dots, n\}$, since $b_k$ is always nonzero. This implies that $v$ and $w$ are linearly dependent,  contradicting our assumption. Thus, all eigenvalues have multiplicity one and the proof is complete.
\end{proof}

We now deform through pencils of symmetric Jacobi matrices to ensure regularity. 

\begin{lemma}[Deformation of Pencils]\label{lem:def_pencil}
Let $n \geq 3$ be an integer. Define 
\begin{align}
    \psi_1 &= \frac{1}{2} \diag ( n-1,\, n-3,\,\dots,\, -n+3,\, -n+1 ) \\
    \psi_2 &= J_n+J_n^T \label{psi2},
\end{align}
where $J_n$ is the lower triangular nilpotent $n\times n$-Jordan block. Then there is a $\Delta$-regular homotopy $(\mathcal{P}_t)_{t \in [0,1]}$ of pencils in $\T_I\X$ between $\mathcal{P}_0= \Span\{\psi_1,\psi_2\}$ and $\mathcal{P} = \T_{I}\Ha^2_{\pr}$. 
\end{lemma}

\begin{proof}
Recall that $\mathcal{P} = \Span \{H, E+F\}$ relative to the principal $\mathfrak{sl}_2$-triple $\{E,F,H\}$. 
We use the basis $\phi_1 = H$ and $\phi_2=E+F$ for $\mathcal{P}$. By well-known $\sllie_2$-representation theory of the principal $\mathfrak{sl}_2$, we can take $\phi_1 =\psi_1$ and $\phi_2$ to be a symmetric Jacobi matrix with positive off-diagonal entries and zero diagonal entries. The result  follows by taking a linear homotopy $\mathcal{P}_t = \Span \{\psi_1, \psi_2(t)\}$ such that $\psi_2(0) =\psi_2$, and $\psi_2(1) = \phi_2$. Hence, $\psi_2(t)$ is a Jacobi matrix with positive off-diagonal entries for all time. By construction, $\mathcal{P}_t \subset \T_I\X$ and by Lemma \ref{lem:JacobiMatrices}, the pencil is $\Delta$-regular for all time $t \in [0,1]$. 
\end{proof}

We now consider the eigenvalues and eigenvectors of the particular Jacobi matrix $\psi_2$ from \eqref{psi2}, which will be essential for the count of $r(n)$ in the proof of Lemma \ref{lem:number_circles}.

\begin{lemma}[Eigenvalues$+$Eigenvectors of $\psi_2$]\label{lem:eigenvectors_psi2} Let $\psi_{2}=J_n+J_{n}^{T}$ where $J_{n}$ is the nilpotent $n \times n$-Jordan block. Then $\psi_{2}$ has $n$ distinct eigenvalues $2 > \lambda_{1} > \lambda_{2} > \dots > \lambda_{n}>-2$, where 
\[
    \lambda_{k}=2\cos\left(\frac{k\pi}{n+1}\right) \ .
\]
The eigenline relative to $\lambda_k$ is spanned by the unit eigenvector $v^{(k)} =(v^{(k)}_1, \dots, v^{(k)}_n)$, where
\[
    v^{(k)}_{j}=\sqrt{\frac{2}{n+1}}\sin\left(\frac{jk\pi}{n+1}\right), \qquad 1\leq j\leq n.
\]
\end{lemma}
\begin{proof} Since $\psi_{2}$ is a symmetric Jacobi matrix, it has $n$ distinct eigenvalues by Lemma \ref{lem:JacobiMatrices}. Let $v=(v_{1}, \dots, v_{n})^{T}$ be a nonzero eigenvector for $\psi_{2}$ relative to an eigenvalue $\lambda \in \R$. We start by showing the a priori bound $|\lambda|<2$:
\begin{align*}
    |\lambda|=\frac{|v^{T}\psi_{2}v|}{\|v\|^{2}}&=\frac{1}{\|v\|^{2}}\left|\sum_{j=1}^{n-1}2v_{j}v_{j+1}\right| \leq \frac{1}{\|v\|^{2}}\sum_{j=1}^{n-1}|2v_{j}v_{j+1}| \\
    &\leq \frac{1}{\|v\|^{2}}\sum_{j=1}^{n-1} (v_{j}^{2}+v_{j+1}^{2}) = 2 -\frac{v_{1}^{2}+v_{n}^{2}}{\|v\|^{2}} \leq 2 .
\end{align*}
Equality holds only if $v_{1}=v_{n}=0$ and $2|v_{j}v_{j+1}|=v_{j}^{2}+v_{j+1}^{2}$ for $1\leq j\leq n-1$, which easily implies that $v=0$, against our assumptions. 

We now proceed with the explicit computation of eigenvalues and eigenvectors of $\psi_{2}$. The equation $\psi_{2}v=\lambda v$ yields the following recurrence relation between the components of $v$:
\[ 
    v_{j-1}+v_{j+1}=\lambda v_{j}, \qquad 1\leq j \leq n,
\]
where, by convention, $v_{0}=v_{n+1}=0$. Looking for nonzero solutions of the form $v_{j}=r^{j}$, we obtain that $r$ must satisfy $r^{2}-\lambda r+1=0$. Since $|\lambda|<2$, we can write $\lambda=2\cos(\theta)$ for some $\theta  \in (0,\pi)$ so that $r=e^{\pm i\theta}$. Now, we know $\psi_2$ has an orthogonal basis of real eigenvectors. The real solutions $v=(v_1,\dots, v_n)^T$ of the recurrence are thus of the form 
\[
    v_{j}=A\cos(j\theta)+B\sin(j\theta), 
\]
for some constants $A,B \in \mathbb{R}$ to be determined from the initial conditions $v_{0}=v_{n+1}=0$:
\begin{align*}
    0&=v_{0}=A\cos(0)+B\sin(0)=A \\
    0&=v_{n+1}=B\sin((n+1)\theta).
\end{align*}
Thus, $\theta \in \frac{\pi}{n+1}\Z$ since $B=0$ would imply $v=0$. 
We deduce that the eigenvalues of $\psi_{2}$ are 
\[
    \lambda_{k}=2\cos\left(\frac{k \pi}{n+1}\right), \qquad 1 \leq k\leq n.
\]
Note that $\lambda_1,\dots, \lambda_n$ are distinct. A unit $\lambda_{k}$-eigenvector $v^{(k)} = (v^{(k)}_1, v^{(k)}_2,\dots, v^{(k)}_n)$ is given by  
\[
    v_{j}^{(k)}= \sqrt{\frac{2}{n+1}}\sin\left(\frac{jk\pi}{n+1}\right) \ .
\]
\end{proof}

We will not invoke the precise formula for the eigenvectors of $\psi_2$ later. Instead, we will only need the following symmetries of the $\psi_2$-eigenvectors, which follow from the explicit descriptions. 

\begin{lemma}[Symmetries of Eigenvectors]\label{lem:tau-symmetry} Let $\tau:\mathbb{R}^{n} \rightarrow \mathbb{R}^{n}$ be the involution $\tau(x_{1}, \dots, x_{n})=(x_{n}, \dots, x_{1})$. The eigenline of $\psi_{2}$ relative to the $k^{th}$ largest eigenvalue $\lambda_{k}$ is spanned by a vector $v^{(k)}$ that is $\tau$-symmetric if $k$ is odd and $\tau$-anti-symmetric if $k$ is even.
\end{lemma}

\begin{proof} First note that since $\tau\psi_{2}=\psi_{2}\tau$, the involution $\tau$ preserves the eigenlines of $\psi_2$. 
Since $\tau$ is an isometry, it also sends unit vectors to unit vectors. Now, let $v^{(k)}$ be the unit eigenvectors of $\psi_2$ from Lemma \ref{lem:eigenvectors_psi2}. Then $\tau v^{(k)}= \varepsilon_{k} v^{(k)}$ for $\varepsilon_{k} \in \{-1,1\}$. We determine the correct sign from the explicit formulas for $v^{(k)} = (v^{(k)}_j)_{j=1}^n$: up to the (positive) multiplicative constant $\sqrt{2/(n+1)}$, we have
\begin{align*}
    (\tau v^{(k)})_{j}=v^{(k)}_{n-j+1} &= \sin \left( \frac{(n-j+1)k\pi}{n+1} \right) = \sin\left(k\pi -\frac{jk\pi}{n+1} \right) \\
    &= - \cos(k\pi)\sin\left(\frac{jk\pi}{n+1} \right) = (-1)^{k+1}v^{(k)}_{j}.
\end{align*}
The claim follows.
\end{proof}

Equipped with these symmetries, we can deduce the number of singular circles. 
\begin{lemma}[Singular Circle Count]\label{lem:number_circles} 
Let $n\geq 4$ be an integer. 
The singular locus $S_{sing} \subset \sphere^{n-1}$ from Lemma \ref{Lem:StructureLemma} is the disjoint union of $r(n)$ circles, where 
\[ r(n)=\begin{cases} \frac{n}{2} & n \in \{0,2\} \bmod 4 \\
                n+1 & n \equiv 1 \bmod 4\\
                n   & n\equiv 3 \bmod 4
\end{cases} \] 
\end{lemma}

\begin{proof}
By Lemma \ref{lem:def_pencil}, we have a homotopy of $\Delta$-regular pencils from $\mathcal{P}=\T_I\Ha^2_{\pr}$ to $\mathcal{P}_0=\Span \{\psi_1,\psi_2\}$. We once again consider the unit tangent vector-unit eigenvector space 
\[ \hat{\mathbb{E}}(\sphere(\mathcal{P})) = \{(\psi, x) \in \sphere(\mathcal{P})\times \sphere^{n-1} \mid \exists \lambda, \psi x = \lambda x\} . \]
Define $\hat{\mathbb{E}}(\sphere(\mathcal{P}_0))$ analogously. Due to the $\Delta$-regular homotopy, we conclude $\hat{\mathbb{E}}(\sphere(\mathcal{P}))$ and $\hat{\mathbb{E}}(\sphere(\mathcal{P}_0))$ are homeomorphic. Now, by the proof of Lemma \ref{Lem:StructureLemma}, the space 
\[\hat{\mathbb{E}}(\mathbb{P}(\mathcal{P})) :=\hat{\mathbb{E}}(\sphere(\mathcal{P}))/(-\id, \id)\] 
is a $(2n)$-fold cover of $\mathbb{P}(\mathcal{P})\cong \sphere^1$ and hence a union of circles. Moreover, the singular locus $S_{sing}$ is homeomorphic to $\hat{\mathbb{E}}(\mathbb{P}(\mathcal{P}))$. In particular,  $|\pi_0(S_{sing})| = |\pi_0(\hat{\mathbb{E}}(\mathbb{P}(\mathcal{P}_0)))|$. \medskip

We now introduce some notation. The matrix $\psi_1$ is diagonal, so we can use it to label the fibers of the covering $\pi: \hat{\mathbb{E}}(\mathbb{P}(\mathcal{P}_0))\rightarrow \mathbb{P}(\mathcal{P}_0)$. 
Indeed, we define the circle $C_i^{\pm}$ to be the connected component of $\hat{\mathbb{E}}(\mathbb{P}(\mathcal{P}_0))$ containing the pair $( [\psi_1], \pm  e_i)$. Note that these labels are \emph{not} necessarily distinct, and our whole task is to precisely understand the redundancy.

Now we begin the task of determining which circles $C_{i}^{\pm}$ actually coincide, to reveal the number of connected components of $\hat{\mathbb{E}}(\mathbb{P}(\mathcal{P}_0))$. We make the following observation: 
let $\gamma:[-1,0]\rightarrow \hat{\mathbb{E}}(\sphere(\mathcal{P}_0))$ be a continuous path such that $\pi \circ \gamma(-1) = \psi_1$ and $\pi \circ \gamma(0) =\psi_2$. Let us write $\gamma(t) = (\psi(t), v(t))$. Then, by Lemma \ref{lem:tau-symmetry}, we see that $\gamma$ can be extended to a continuous path $\hat{\gamma}:[-1,1]\rightarrow \hat{\mathbb{E}}(\sphere(\mathcal{P}_0))$ as follows: 

\[\hat{\gamma}(t) = \begin{cases} 
\gamma(t), & -1 \leq t\leq 0,\\
(\tau \psi(-t) \tau , \varepsilon\, \tau v(-t) \big) & \,\;\;0 < t\leq 1.
\end{cases}\] 
where $\varepsilon \in \{+,-\}$. Here, the sign of $\varepsilon$ is determined by symmetries. Indeed, for $t < 0$, the eigenvectors $v(t)$ are each eigenvectors with respect to the $k^{th}$-largest eigenvalue $\lambda_k(t)$ of $\psi(t)$ for some constant $k$. Considering $t=0$, Lemma \ref{lem:tau-symmetry} implies $\varepsilon =1$ if $k$ is odd and $\varepsilon = -1$ if $k$ is even. The next observation is that two circles $C_{i}^{\sigma_i}$ and $C_j^{\sigma_j}$ can only have the opportunity to coincide if $j = n+1-i$. Indeed, $\hat{\gamma}(t)$ has $v(t)$ a $\lambda_k(t)$-eigenvector of $\psi(t)$. Now, $\psi(1)= - \psi_1$, and the eigenline of the $k^{th}$ largest eigenvalue of $-\psi_1$ is spanned by $e_{n+1-k}$. 

We can use the double covering 
$\hat{\mathbb{E}}(\sphere(\mathcal{P}_0)) \rightarrow \hat{\mathbb{E}}(\mathbb{P}(\mathcal{P}_0))$ to study connected components downstairs. 
Every closed loop in $\hat{\mathbb{E}}(\mathbb{P}(\mathcal{P}_0))$ obtains one of the two mutually exclusive forms: 
\begin{enumerate}[noitemsep,label=(\alph*)]
    \item a closed loop in $\hat{\mathbb{E}}(\sphere(\mathcal{P}_0))$ that projects down to a loop in $\hat{\mathbb{E}}(\mathbb{P}(\mathcal{P}_0))$.
    \item an open segment in $\hat{\mathbb{E}}(\sphere(\mathcal{P}_0))$ lifting a path from $\psi_1$ to $-\psi_1$ but projecting to a closed loop in $\hat{\mathbb{E}}(\mathbb{P}(\mathcal{P}_0))$.
\end{enumerate}

We now show every connected component of $\hat{\mathbb{E}}(\mathbb{P}(\mathcal{P}_0))$ contains precisely 1, 2 or 4 branches of the covering $\pi$. The symmetries of $\tau$ from Lemma \ref{lem:tau-symmetry} tell us exactly what we see. To explain this key point, we introduce one more piece of notation. An injective path $\overline{\gamma}_t:[0,1]\rightarrow  \sphere(\mathcal{P}_0)$ from $\psi_1$ to $-\psi_1$ passing through $\psi_2$ lifts uniquely to a path $\gamma_t:[-1,1] \rightarrow  \hat{\mathbb{E}}(\sphere(\mathcal{P}_0))$ such that $\gamma(-1) =(\psi_1, e_k)$. This path 
satisfies $\gamma(1) = (-\psi_1, \varepsilon_k\, e_{n+1-k})$, where $\varepsilon_k \in \{-1,+1\}$. As noted above, $\varepsilon_k=+1$ when $k$ is odd and $\varepsilon_k=-1$ when $k$ is even. 
The possibilities for the distinctness of the connected components $C_k^+, C_k^-, C_{n+1-k}^+, C_{n+1-k}^-$ of $\hat{\mathbb{E}}(\mathbb{P}(\mathcal{P}_0))$ are entirely governed by the signs of $\varepsilon_i$'s as follows. 
\begin{itemize}
    \item \textbf{Case 1: $\bm{\varepsilon_k\varepsilon_{n+1-k} =1}$ and $\bm{n+1-k\neq k}$.} In this case, we achieve a closed loop in $\hat{\mathbb{E}}(\sphere(\mathcal{P}_0))$ by traveling from $(\psi_1, e_k)$ to $(-\psi_1, \varepsilon_k e_{n+1-k})$ to $(\psi_1, \varepsilon_k\varepsilon_{n+1-k}e_k)$. In particular, this means $C_{k}^+=C_{n+1-k}^{\varepsilon_k}$ and $C_{k}^-=C_{n+1-k}^{-\varepsilon_k}$. 
    Thus, we found in total two connected components of $\hat{\mathbb{E}}(\mathbb{P}(\mathcal{P}_0))$. This case is of type (a). 
    \item \textbf{Case 2: $\bm{\varepsilon_k\varepsilon_{n+1-k} =-1}$.} In this case, all four circles $C_{k}^{+}, C_{k}^-, C_{n+1-k}^{+}, C_{n+1-k}^{-}$ coincide and comprise one connected component of $\hat{\mathbb{E}}(\mathbb{P}(\mathcal{P}_0))$. 
    Indeed, 
    \[ (\psi_1,e_k), (-\psi_1, \varepsilon_{k} e_{n+1-k}),
    (\psi_1,\varepsilon_{n+1-k} \varepsilon_ke_k), (-\psi_1, -\varepsilon_{k} e_{n+1-k})\] 
    all lie on the same connected component of $\hat{\mathbb{E}}(\sphere(\mathcal{P}_0))$ and hence also for $\hat{\mathbb{E}}(\mathbb{P}(\mathcal{P}_0))$. This case is of type (a).
    
     \item \textbf{Case 3: $\bm{k=n+1-k}$.} Thus, $\varepsilon_k\varepsilon_{n+1-k}=1$. 
    This case is, of course, exceptional. Here, $n$ must be odd and of the form $n=2j+1$ with $k=j+1$. In fact,
     we must distinguish two further sub-cases.
     \begin{itemize}[label=$\circ$]
         \item \textbf{Case 3(i): $\mathbf{n \equiv 1 \,\textbf{mod} \, 4}$}.  Thus, $j$ is even, $k=j+1$ is odd and $\varepsilon_k=+1$. Thus, 
     we see an open loop upstairs in $\hat{\mathbb{E}}(\sphere(\mathcal{P}))$ from $(\psi_1, e_k)$ to $(-\psi_1,  e_k)$ that closes up downstairs. Hence, the branches $C_k^+$ and $C_k^-$ comprise different connected component of $\hat{\mathbb{E}}(\mathbb{P}(\mathcal{P}_0))$. This case is of type (b).
        \item \textbf{Case 3(ii): $\bm{n \equiv 3 \,\textbf{mod}\, 4}$}.  Thus, $j$ is odd, $k=j+1$ is even and $\varepsilon_k=-1$. Thus, 
     we see a closed loop upstairs from $(\psi_1, e_k)$ to $(-\psi_1, -e_k)$ and back to $(\psi_1, e_k)$. 
     Hence, $C_k^+$ and $C_k^-$ comprise the same connected component of $\hat{\mathbb{E}}(\mathbb{P}(\mathcal{P}_0))$. This case is of type (a).
     \end{itemize}
\end{itemize}

\begin{figure}[ht]
\centering
\resizebox{\textwidth}{!}{
  \begin{tikzpicture}[scale=1]
  \coordinate (A) at (28.26, 20.76);
  \coordinate (B) at (33.26, 20.76);
  \coordinate (C) at (33.26, 17.76);
  \coordinate (D) at (28.26, 17.76);
  \coordinate (E) at (17.00, 19.50);
  \coordinate (F) at (23.00, 19.50);
  \coordinate (G) at (19.00, 21.90);
  \coordinate (H) at (21.00, 21.90);
  \coordinate (I) at (18.75, 17.00);
  \coordinate (J) at (21.25, 17.25);
  \coordinate (K) at (38.75, 19.50);
  \coordinate (L) at (38.75, 19.50);
  \coordinate (M) at (36.98, 19.52);

  \draw (28.26,17.76) rectangle (33.26,20.76);
  \draw[-Stealth] (E) .. controls (19.00, 21.90) and (21.00, 21.90) .. (F);
  \draw[] (E) .. controls (18.75, 17.00) and (21.25, 17.25) .. (F);
  \draw (K) circle (1.82);
  \draw[-Stealth] (A) -- (B);
  \draw[-Stealth] (B) -- (C);
  \draw[-Stealth] (C) -- (D);
  \draw[-Stealth] (D) -- (A);
  \fill[black] (A) circle (4pt);
  \node[above, black, font=\LARGE] at (A) {$(\psi_1,e_{k})$};
  \fill[black] (B) circle (4pt);
  \node[above, black, font=\LARGE] at (B) {$(-\psi_1,\varepsilon_k e_{n+1-k})$};
  \fill[black] (C) circle (4pt);
  \node[below, black, font=\LARGE] at (C) {$(\psi_1,-e_{k})$};
  \fill[black] (D) circle (4pt);
  \node[below, black, font=\LARGE] at (D) {$(-\psi_1,-\varepsilon_{k}e_{n+1-k})$};
  \fill[black] (E) circle (4pt);
  \node[left, black, font=\LARGE] at (E) {$(\psi_1,e_{k})$};
  \fill[black] (F) circle (4pt);
  \node[right, black, font=\LARGE] at (F) {$(-\psi_1,\varepsilon_{k}e_{n+1-k})$};
  \fill[black] (M) circle (4pt);
  \node[right, black, font=\LARGE] at (M) {$(\,[\psi_1]\,,e_{k})$};
\end{tikzpicture}
}
\caption{Decomposition of the singular circles into branches $C_{k}^{\pm}$. The left figure shows Case 1, the middle figure Case 2, each drawn upstairs, and the right figure Case 3(i), drawn downstairs. Case 3(ii) looks just like Case 1. }
\label{Fig:Branches}
\end{figure}
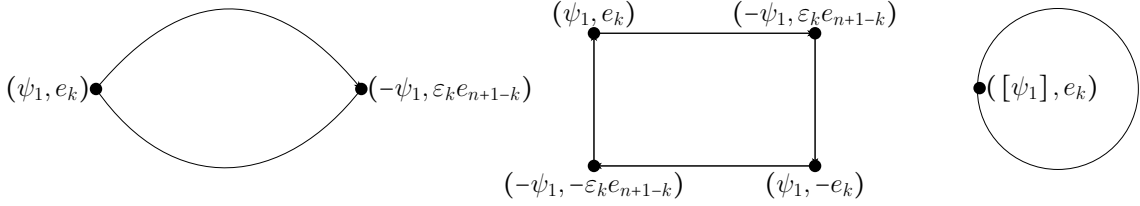

All that remains is to assemble the result. \medskip 

Suppose $n$ is even. For any index $1\leq k\leq n$, we have that $k$ and $n+1-k$ have different parity mod 2. Thus, 
$\varepsilon_k \varepsilon_{n+1-k} =-1$ always occurs, and Case 2 transpires for every pair of indices $(k, n+1-k)$. We conclude that the $2n$-fold cover $\hat{\mathbb{E}}(\mathbb{P}(\mathcal{P}_0))$ of $\mathbb{P}(\mathcal{P}_0)\cong \sphere^1$ has $\frac{n}{2}$ connected components. \medskip 

Now, suppose $n \equiv 1 \bmod 4$. Let us write $n=4j+1$. For any index $k \neq 2j+1$, we see that the pair of indices $(k, n+1-k)$ leads to Case 1, as $k$ and $n+1-k$ have the same parity mod 2. This leads to $n-1$ connected components. Now, for the middle value $k=2j+1$, Case 3(i) occurs and $\varepsilon_{j+1}=1$. Thus, the circles $C_{j+1}^+$ and $C_{j+1}^-$ comprise different connected components of $\hat{\mathbb{E}}(\mathbb{P}(\mathcal{P}))$. In total, we obtain $n+1$ connected components. \medskip 

Finally, suppose $n \equiv 3 \bmod 4$. Again, for any index $k$ other than the middle index, we see the pair of indices $(k, n+1-k)$ land in Case 1. Once more, this provides $n-1$ connected components. On the other hand, we see a difference for the circles $C_{k}^+, C_{k}^-$ when $k = \frac{n+1}{2}$ is the middle index. Here, we land in Case 3(ii) so that $C_k^+ =C_k^-$ are the same. We obtain $n$ connected components in total in this case.   
\end{proof}

The proof of the Lemma \ref{lem:number_circles} has implications for the orientability of the sphere bundle $X_{sing}\rightarrow S_{sing}$, which depends on the connected component of $S_{sing}$ in some cases.  

\begin{corollary}[Orientability of the Singular Fibration]\label{cor:orientability_singular}
Let $p: \hat{\mathcal{B}}(\mathcal{P})\rightarrow\sphere^{n-1}$ be the map from Lemma \ref{Lem:StructureLemma} and $C \subset S_{sing}$ be a singular circle. Then the $\sphere^{n-3}$-bundle $p:p^{-1}(C) \rightarrow C$ is trivial exactly when $C$ falls within Case 1, 2, or 3(ii) in Lemma \ref{lem:number_circles}. In particular, 
\begin{enumerate}[noitemsep,label=(\roman*)]
    \item the fiber bundle $p:X_{sing} \rightarrow S_{sing}$ is trivial when $n \in \{0,2,3\} \bmod 4$, 
    \item the fiber bundle $p^{-1}(C) \rightarrow C$ is non-trivial for exactly two circles $C$ when $n \equiv 1 \bmod 4$. 
\end{enumerate}
\end{corollary}

\begin{proof}
\textbf{Cases 1, 2, 3(ii).} These cases correspond to case (a) in Lemma \ref{lem:number_circles}. 
We reason as in the Structure Lemma \ref{Lem:StructureLemma}, splitting once again $\T\sphere^{n-1} = \mathcal{R} \oplus \mathcal{R}^\bot$. 
For $x \in S_{sing}$, we have $\dim\mathcal{R}_x=1$ by definition. Now, the bundle $\T\sphere^{n-1}|_{S_{sing}}$ is trivial and hence orientable. For any circle $C \subset S_{sing}$, the vector bundle $\mathcal{R}^\bot|_{C} \rightarrow C$ is orientable if and only if the sphere bundle $p^{-1}(C) \rightarrow C$ is orientable. 
Since $\T\sphere^{n-1}$ is orientable, $\mathcal{R}|_{C} \rightarrow C$ is orientable if and only if $\mathcal{R}^{\bot}|_C \rightarrow C$ is orientable. 
For a circle $C$ of type (a), we can show that $\mathcal{R}|_{C} \rightarrow C$ is a trivial line bundle as follows. Denote $J:\mathcal{P}\rightarrow \mathcal{P}$ as an arbitrary almost-complex structure. By hypothesis of type (a), we can define a continuous map $\psi: C \rightarrow \sphere(\mathcal{P})$ 
such that $x$ is an eigenvector of $\psi(x)$. 
We then define a non-vanishing section $s:C \rightarrow \mathcal{R}$ by $s(u) = \pi_{u^\bot}\big(J(\psi(u))(u)\big)$. Since $\psi(x)$ has $x$ as an eigenvector, then $J(\psi(u)) \in \mathcal{P}$ does not have $u$ as an eigenvector, and hence $J(\psi(u))(u)$ has a non-vanishing projection to $u^\bot$ and hence to $\mathcal{R}|_{u}$. The non-vanishing section $s$ verifies that the line bundle $\mathcal{R}|_{C} \rightarrow C$ is trivial. \medskip 

\textbf{Case 3(i).} Now, we suppose $C \subset S_{sing}$ is a circle as in Case 3(i) of Lemma \ref{lem:number_circles}, in which case we are of type (b) from that lemma. This means that there is no continuous map $\psi: C \rightarrow \sphere(\mathcal{P})$ such that $\psi(x)$ has $x$ an eigenvector. We need to explain how this condition manifests as the non-orientability of the line bundle $\mathcal{R}|_{C} \rightarrow C$, contrary to the previous case. 

Now, let us introduce another splitting. Consider the trivial rank two vector bundle $\underline{\mathcal{P}} = C \times \mathcal{P}$ over $C$. We split $\underline{\mathcal{P}}$ into a sum of line bundles as follows. 
Let $E \rightarrow C$ be the eigenline of the position vector: $E|_{x} = \{ \psi \in \mathcal{P} \mid \psi(x)\in \R\{x\}\}$. Place an arbitrary Euclidean inner product on $\mathcal{P}$ and denote $E^\bot$ as the orthogonal complement such that $\underline{\mathcal{P}} = E \oplus E^\bot$. Then since $\underline{\mathcal{P}}$ is trivial, $E$ is orientable if and only if $E^\bot$ is orientable. 
However, observe that the vector bundle map $E^\bot \rightarrow \mathcal{R}|_{C}$ by $(x,\psi)\mapsto (x,\pi_{x^\bot}(\psi(x)))$ is non-vanishing by definition of $E^\bot$ and hence defines an isomorphism. Thus, $\mathcal{R}$ is non-orientable if and only if $E$ is non-orientable. However, the non-existence of a continuous map $\psi: C \rightarrow \sphere(\mathcal{P})$ such that $\psi(x) \in \R\{x\}$ implies the non-existence of a non-vanishing section $\psi \in \Omega^0(C,E)$. Hence, the line bundle $E \rightarrow C$ is non-trivial and thus non-orientable. Consequently, $\mathcal{R}|_{C} \rightarrow C$ is also non-orientable. We conclude that $\sphere(\mathcal{R}^{\bot})|_{C} \rightarrow C$ is also non-orientable. Again identifying $p^{-1}(C)\rightarrow C$ and $ \sphere(\mathcal{R}^\bot)|_{C}\rightarrow C$ as fiber bundles over $C$, the proof is complete. 
\end{proof}

We now write the singular locus $S_{sing}$ as the disjoint union of $S_{sing}^{o}$ and $S_{sing}^{no}$, the union of all singular circles for which the fibration discussed in Corollary \ref{cor:orientability_singular} is orientable (hence trivial) or non-orientable, respectively. In particular, $S_{sing}^{no}=\emptyset$ when $n \in \{0,2,3\} \bmod 4$ and is homeomorphic to $\sphere^{1} \sqcup \sphere^{1}$ when $n \equiv 1 \bmod 4$.

\subsection{Homotopy Type of \texorpdfstring{$S_{gen}$}{Sgen}}
\label{Sec:SgenHomology}

In this subsection, we study the base $S_{gen}$ of the fibration $\sphere^{n-3} \rightarrow X_{gen} \rightarrow S_{gen}$. By Lemma \ref{Lem:StructureLemma} and Lemma \ref{lem:number_circles}, we know $S_{gen} =\sphere^{n-1} \setminus \bigsqcup_{i=1}^{r(n)}C_i$, where $C_i \cong \sphere^1$ is a circle. We compute the homology of $S_{gen}$ and conclude its homotopy type.

To begin, we use \emph{Alexander duality} \cite[Corollary 3.45]{Hat01} to compute the reduced cohomology of $S_{gen}$. Namely, if $K \subset \sphere^{n-1}$ is a compact, proper, locally contractible subset, then the cohomology of the complement of $K$ and the cohomology of $K$ are related by: 
\[ \widetilde{H}^i(\sphere^{n-1} \setminus K, \Z) = \widetilde{H}^{n-2-i}(K, \Z). \]
Presently, $K = \bigsqcup_{i=1}^{r(n)}\sphere^1$. 
We find the cohomology groups of $S_{gen}=\sphere^{n-1}\setminus K$ are given by 
\[ \widetilde{H}^i(S_{gen},\Z) = \begin{cases} 
\Z^{r(n)} & i= n-3\\
\Z^{r(n)-1} & i=n-2\\
0  & \text{else}.\\
\end{cases}
\]
In particular, all reduced cohomology groups vanish except two in consecutive dimensions. The universal coefficient theorem shows that $H^i(S_{gen},\Z)  \cong H_i(S_{gen},\Z)$. 
By examining minimal cell structures, one can show that $S_{gen}$ not only has the same homology as  
\[ Y \coloneq \bigg(\bigvee_{i=1}^{r(n)}\sphere^{n-3} \bigg) \vee \bigg(\bigvee_{i=1}^{r(n)-1} \sphere^{n-2} \bigg), \] 
but actually the two spaces are homotopy equivalent  \cite[Example 4C.2]{Hat01}.  

\subsubsection{Homological Generators for \texorpdfstring{$S_{gen}$}{Sgen}}\label{Sec:SgenHomologyGenerators}
Now, let $n \geq 5$, so that the connected components of $S_{sing}$ are a priori forbidden to link in $\sphere^{n-1}$. We apply Mayer-Vietoris to find generators for the homology of $S_{gen}$. We shall keep the notation as above: $K = \bigsqcup_{i=1}^{r} C_i$ is a disjoint union of $r=r(n)$ circles in $\sphere^{n-1}$, called $C_i$. Let us make the following choices: 
\begin{itemize}
    \item $A = S_{gen} = \sphere^{n-1}\setminus K$ 
    \item $B$ is a regular neighborhood of $K$, homotopy equivalent to $K$. 
    \item Hence, $A \cap B$ is a union of punctured regular neighborhoods of $C_i$. In particular, $A \cap B \cong \bigsqcup_{i=1}^k (C_i \times S_i)$, where $S_i \cong \sphere^{n-3}$ is found geometrically by exponentiating a normal $(n-3)$-sphere in the normal space to $\T_{x_i}C_i$ for some $x_i \in C_i$. 
\end{itemize}
Here, we use homology with $\Z$ coefficients. Mayer-Vietoris gives us the exact sequence
\[ \cdots \longrightarrow H_{i+1}(\sphere^{n-1}) \longrightarrow H_{i}(A \cap B) \longrightarrow H_i(A) \oplus H_i(B) \longrightarrow H_i(\sphere^{n-1}) \longrightarrow \cdots .\]
For the index $i= n-3$, we conclude the middle map 
$H_{i}(A \cap B) \stackrel{(j_*,k_*)}{\longrightarrow} H_i(A) \oplus H_i(B)$ is an isomorphism, where $j: A \cap B \rightarrow A $ and $k: A \cap B \rightarrow B$ are the inclusion maps. We conclude that the linking spheres 
$[S_i] \in H_{n-3}(S_{gen},\Z) \cong \Z^r$ form generators. 
In degree $i=n-2$, if we follow Mayer-Vietoris as above, this time we see 
\[ 0 \longrightarrow H_{n-1}(\sphere^{n-1}) \longrightarrow H_{n-2}(A \cap B) \longrightarrow H_{n-2}(A) \oplus H_{n-2}(B) \longrightarrow H_{n-2}(\sphere^{n-1})=0\]
Using $H_{n-1}(\sphere^{n-1})\cong\Z$, $H_{n-2}(A\cap B) \cong\Z^r$, $H_{n-2}(B) = 0$, we learn that $A \cap B \cong \bigsqcup_{i=1}^r (\sphere^1\times \sphere^{n-3})$ has 
inclusion $k: A \cap B \rightarrow S_{gen}=A$ such that $k_*:H_{n-2}(A \cap B) \rightarrow H_{n-2}(A)$ surjects with kernel $\Z$. Therefore, $H_{n-2}(S_{gen})\cong \Z^{r(n)-1}$ is generated (not freely) by the spherical cylinders $P_i=\sphere^{1}\times \sphere^{n-3}$ given by the boundary of a tubular neighborhood $N_{i}$ of $C_{i}$ inside $\sphere^{n-1}$. 

For later, we describe $r(n)-1$ smoothly embedded spheres $Q_{i} \subset S_{gen}$ that freely generate $H_{n-2}(S_{gen})$. To this end, define $Q_{i}$ as a smoothly embedded $(n-2)$-sphere in $\sphere^{n-1}$ that separates $C_{i}$ from the singular circles $(C_{j})_{j \neq i}$. 
If we choose the spherical cylinders $P_{i} =\partial N_i$ to be sufficiently close to $C_{i}$ by making the tubular neighborhood small enough, we can also arrange $Q_i$ to separate $P_{i}$ from $(P_{j})_{j \neq i}$. Up to changing orientations, each $Q_i$ is homologous to $P_{i}$ since $Q_{i} \cup P_{i}$ is the boundary of the $(n-1)$-dimensional submanifold of $S_{gen}$ obtained by removing the tubular neighborhood $N_{i}$ from the disk $D_{i} \subset \sphere^{n-1}$ bounded by $Q_{i}$ and containing $C_{i}$. 
Note that only $r(n)-1$ of these spheres are needed to generate $H_{n-2}(S_{gen})$.

\subsection{Stable Triviality}
We now exhibit the lifted base of pencil $X=\hat{\mathcal{B}}(\mathcal{P})$ as a regular level set in $\R^{2n}$. This re-proves in concrete fashion the smoothness of $X$ and verifies that its tangent bundle is stably trivial. 

\begin{lemma}[Stable Triviality]\label{Lem:StablyTrivial}
Suppose $n \geq 4$ is an integer. 
Let $\mathcal{P}$ be the $\Delta$-regular pencil $ \T_Q\Ha^2$. Then the manifold $X=\hat{\mathcal{B}}(\mathcal{P})$ is stably trivial. In fact, $\T X \oplus \varepsilon^5_\R \cong \varepsilon^{2n}_{\R}$.
\end{lemma}

\begin{proof}
We construct $X$ as a regular level set in $\R^{2n}$. 
Define $F:\R^{2n}\rightarrow \R^5$ by 
\[ F(x,y) = \big(\langle x,x\rangle -1, \langle y,y\rangle-1, \langle x,y\rangle, \langle \psi_1(x), y\rangle, \langle \psi_2(x),y\rangle \big).\]
Observe that $X = F^{-1}(0)$.

Write $F=(F_i)_{i=1}^5$. We will verify that $(\nabla F_i)_{i=1}^5$ are linearly independent, which implies that $(\nabla F_i)_{i=1}^5$ comprise a global non-vanishing frame of the normal bundle $\N X$, realizing a splitting $\varepsilon^{2n}_{\R} \cong\T\R^{2n}|_{X} \cong \T X \oplus \varepsilon^5_{\R}$.

Now, fix a point $(x,y) \in F^{-1}(0)$. Observe that the gradients $\nabla F_i$, written as elements of $\R^{2n} = \R^n\oplus \R^n$, are as follows: 
\begin{itemize}[noitemsep]
    \item $\nabla F_1(x,y) = 2( x, 0),$
    \item $\nabla F_2(x,y)= 2(0, {y})$
    \item $\nabla F_3(x,y)= ({y}, {x})$
    \item $\nabla F_4(x,y)= (\psi_1({y}), \psi_1({x})),$
    \item $\nabla F_5(x,y)= (\psi_2({y}), \psi_2({x})) $.
\end{itemize}
We will show that $(\nabla F_i(x,y))_{i=1}^5$ are linearly independent. 

Let us write $w_i \coloneqq \nabla F_i(x,y)$ and $w_i = (w_i^1, w_i^2)\in \R^{n}\oplus \R^{n}$. 
Suppose that $\sum_{i=1}^5 c_i w_i = 0$, for some scalars $(c_i)_{i=1}^5 \in \R^5$. Observe that 
$0=\sum_{i=1}^5 \langle w_i^1, x\rangle = 2c_1$
and symmetrically $0=\sum_{i=1}^5 \langle w_i^2, y\rangle = 2c_2$. Hence, $c_1=c_2=0$. 

The remaining equation asserts that for $\lambda = -c_3$, we have 
\[ \begin{cases} 
(c_4\psi_1+c_5\psi_2)(y) = \lambda y \\
(c_4\psi_1+c_5\psi_2)(x) = \lambda x \\
\end{cases}.
\]
This equation declares that the endomorphism $c_4\psi_1+c_5\psi_2 \in \mathcal{P}$ has a pair $(x,y)$ of orthogonal eigenvectors with the same eigenvalue $\lambda$. Proposition \ref{Prop:DeltaRegularity} then says $c_4\psi_1+c_5\psi_2 = 0$, which implies $c_4=c_5=0$ and hence $c_3=0$. The proof is then complete. 
\end{proof}

\subsection{Euler Class of \texorpdfstring{$\mathcal{R}^\bot \rightarrow S_{gen}$}{Rbot--> Sgen}}

Recall from the Structure Lemma that $X_{gen} \rightarrow S_{gen}$ is identified with the $\sphere^{n-4}$-bundle $\sphere(\mathcal{R}^\bot)|_{S_{gen}}\rightarrow S_{gen}$. 
We now study the rank $(n-3)$-vector bundle $\mathcal{R}^\bot|_{S_{gen}} \rightarrow S_{gen}$ over the generic locus $S_{gen}$. In particular, we determine the Euler class $e(X_{gen}) \in H^{n-3}(S_{gen})$ by considering the restriction of the bundle $\mathcal{R}^\bot$ to certain sub-spheres $S_i$ of codimension two in $S_{gen}$ that generate $H_{n-3}(S_{gen})$, as explained in $\S$\ref{Sec:SgenHomologyGenerators}. 

\begin{theorem}[Euler Class of $X_{gen}\rightarrow S_{gen}$]\label{thm:EulerClassXgen}
Let $n\geq 5$ be an integer and $X = \hat{\mathcal{B}}(\mathcal{P})$ for $\mathcal{P} =\T_Q\Ha^2_{pr}$ as in Lemma \ref{Lem:StructureLemma}. 
For each singular circle $C_i$ of $S_{sing}$, let $S_i \subset S_{gen}$ be a smooth $(n-3)$-sphere with linking number $1$ with $C_{i}$. Then
\begin{enumerate}[label=(\roman*)]
    \item The vector bundle $\mathcal{R}^\bot|_{S_i}\rightarrow S_i$ is isomorphic to the tangent bundle $\T S_i$. 
    \item The Euler class $e(X_{gen}) \in H^{n-3}(S_{gen})$ of $X_{gen}$ is as follows:
    \begin{itemize}
        \item When $n$ is even, $e(X_{gen})=0$. 
        \item When $n$ is odd, $e(X_{gen})=(\pm 2, \pm 2, \dots, \pm 2) $, in the basis $([S_i])_{i=1}^{r(n)}$. 
    \end{itemize}
\end{enumerate} 
\end{theorem}
\begin{proof} (i) 
Let $x_{0} \in C_{i}$. A sphere $S_{i} \subset S_{gen}$ that is a generator of $H_{n-3}(S_{gen},\Z)\cong \Z^{r(n)}$ can be obtained as follows. Let 
$\exp:\T_{x_0}\sphere^{n-1} \rightarrow \sphere^{n-1}$ be the Riemannian exponential map. 
Form the orthogonal splitting 
$\T_{x_0}\sphere^{n-1} = \T_{x_0}C_i \oplus \N_{x_0}C_i$ and let 
\[ Y_i = \{ y \in \N_{x_0}C_i :  ||y|| = \epsilon \} \]
be a sphere of small radius $\epsilon > 0$ centered at $0$ in the normal space. 
Then define $S_i \coloneq \exp( Y_i)$. By construction, $S_i$ has linking number 1 with $C_i$. 

Let us determine a tangential direction spanning $\T_{x_{0}}C_i$. Recall $S_{sing} \cong \hat{\mathbb{E}}(\mathbb{P}(\mathcal{P}))$ from Lemma \ref{Lem:StructureLemma}. 
Now, let $x(t)$ be a smooth parameterization of an arc of $C_{i}$ around $x_0$ such that $x(0)=x_{0}$, and such that 
$\pr_1 \circ x(t) = [\psi+tJ\psi] \in \mathbb{P}(\mathcal{P})$, where $\pr_1: S_{sing} \rightarrow \mathbb{P}(\mathcal{P})$ is the natural projection and $J$ is a complex structure on $\mathcal{P}$. 
That is, $x(t)$ is a unit eigenvector of $\psi(t):=\psi+tJ\psi$ relative to some eigenvalue $\lambda(t)$, meaning
\begin{align}\label{MovingEigenvector}
    \psi(t)x(t)= \lambda(t)x(t). 
\end{align}
Note that here we made a non-canonical choice of a representative $\psi \in \sphere(\mathcal{P})$ such that $x_{0}$ is an eigenvector for $\psi$. For the purposes of this proof, it is enough to notice that a smooth choice can be made in a small arc around $x_{0}$ and a different choice only changes the sign of $\psi(t)$ and $\lambda(t)$.

Differentiating \eqref{MovingEigenvector} at $t= 0$, we get 
\[ (\psi-\lambda_0\Id) \dot{x}=\dot{\lambda}x_0-(J\psi) x_0.\]
Now, define $L_0 \coloneqq \psi-\lambda_0\Id$
so that the tangent vector $\dot{x}$ satisfies the linear equation 
\begin{equation}\label{eq:linear}
L_0\dot{x}=\dot{\lambda}x_{0}-(J\psi) x_{0}.
\end{equation}
By definition $x_0$ is an eigenvector of $\psi$ with respect to eigenvalue $\lambda_0$. By Lemma \ref{lem:eigenvectors_psi2}, the eigenvalue $\lambda_0$ of $\psi$ has algebraic multiplicity one, meaning $\ker(L_0) = \R \{x_0\}$.
The matrix $L_0$ is symmetric, so the spectral theorem implies equation \eqref{eq:linear} has a solution if and only if the right-hand side is orthogonal to $x_{0}$.
Hence, $\dot{\lambda}=\langle x_{0}, (J\psi) x_{0} \rangle$. 

Substituting this relation in equation \eqref{eq:linear} we obtain 
\[
    L_0\dot{x}=\langle x_{0}, (J\psi) x_{0} \rangle x_{0}-(J\psi) x_{0}=-\pi_{x_0^\bot}((J\psi) x_{0}),
\]
where $\pi_{x_0^\bot}:\R^n \rightarrow x_0^\bot$ denotes the orthogonal projection. 
Thus, $\dot{x} \in x_0^\bot=\T_{x_0}\sphere^{n-1}$ is the unique solution to equation \eqref{eq:linear} in the orthogonal complement of $x_{0}$. \medskip 

We now consider the restriction of the fibration $X_{gen}\rightarrow S_{gen}$ to the sphere $S_i$. Recall that $\epsilon >0$ is currently fixed. With respect to this choice, every point $u \in S_i$ obtains the form $u= \frac{x_0+ f(\epsilon)w}{||x_0+f(\epsilon)w||}$ for some $w \in \N_{x_0}C_i$ of unit 
norm, where $f(\epsilon) >0$ is independent of $u$. In this way, the map 
\begin{align*}
    \sphere(\N_{x_0}C_i) &\longrightarrow S_i, \qquad 
    w \longmapsto \frac{x_0+ f(\epsilon)w}{||x_0+f(\epsilon)w||}
\end{align*} 
is a homeomorphism. 
For later, we note that as we change our choice of $\epsilon$, this constant $f(\epsilon)$ defines a continuous function $f:\R_{>0} \rightarrow \R_{>0}$ such that $\lim_{\epsilon \rightarrow 0^+}f(\epsilon)=0$. 

Recall the rank $(n-3)$ vector bundle $\mathcal{R}^\bot|_{S_{gen}} \rightarrow S_{gen}$ with fiber at $u \in S_{gen}$ given by the vectors in  $\R^n$ orthogonal to $u$, $\psi u$ and $(J\psi) u$. For notational simplicity, we write $E:=\mathcal{R}^\bot$ for this proof. Hence, the fiber of $X_{gen}$ over $u \in S_{gen}$ is $\sphere(E|_u)$. 
Writing $u=x_{0}+f(\epsilon)w$, observe that 
\begin{align*}
    \Span (u, (J\psi) u, \psi u)=\Span(u, (J\psi) u, L_0u) = \Span(u, (J\psi) u, L_0w).
\end{align*}
Hence, we find 
\[ E_u = \Span(u, (J\psi) u, L_0w)^\bot.\]

Following the first step of the proof, for each $\epsilon \in (0,1]$, we can consider $S_i(\epsilon)=\exp(Y_i(\epsilon))$. Under the obvious homothety, these spheres $S_i(\epsilon)$ are each canonically identified with $\sphere(\N_{x_0}C_i) \cong \sphere^{n-3}$. We now write $W =\N_{x_0}C_i$ for simplicity. 
In this way, by restriction of $E \rightarrow S_{gen}$, we can think of the vector bundles $E|_{S_i(\epsilon)} \rightarrow S_i(\epsilon)$ as defining a smooth family $E^{\epsilon}$ of vector bundles over $\sphere^{n-3}\cong \sphere(W)$. 

The crux of the proof is to examine the result as $\epsilon \rightarrow 0$. We claim in the limit, the bundles $E^{\epsilon} \rightarrow \sphere(W)$ converge to the rank $n-3$ vector bundle $E^0 \rightarrow \sphere(W)$ with fiber at $w$ given by
\[
    E^{0}|_{w}:=\Span(x_{0}, (J\psi)x_{0}, L_0w)^{\perp}=\Span(x_{0}, L_0\dot{x}, L_0w)^{\perp}  \ .
\]
We first explain why $E^0$ has fibers of dimension $n-3$. First, since $x_0$ is a $\psi$-eigenvector, Proposition \ref{Prop:DeltaRegularity} on $\Delta$-regularity implies $(J\psi)x_0$ and $x_0$ are linearly independent.  
Next, we claim $L_0w$ is not in the plane $\Span (x_{0},(J\psi)x_{0})$. Suppose otherwise, so that there are constants $\alpha, \beta \in \mathbb{R}$ such that 
\[
    L_0w=\alpha x_{0}-\beta(J\psi)x_{0}.
\]
This equation has a solution for $(\alpha,\beta)\neq (0,0)$ if and only if the right-hand side is orthogonal to the kernel of $L_0 $, which is spanned by $x_{0}$. Hence $\alpha=\beta \langle (J\psi) x_{0}, x_{0} \rangle$ and, comparing the equation above with equation \eqref{eq:linear}, we would have that $w=\beta\dot{x}$, contradicting the definition of $W$. Hence, $E^{0} \rightarrow \sphere(W)$ defines a rank $(n-3)$ vector bundle. 

To see the isomorphism $E^0 \cong E^{\epsilon}$ of vector bundles over $\sphere(W)$, we shall apply the following lemma, which essentially just uses that here we have a 1-parameter family of vector bundles. 

\begin{lemma}[Convergence of Orthogonal Complements]\label{Lem:OrthogonalBundleConvergence}
Let $M$ be a smooth manifold and $f:M\times [0,1]\rightarrow (\sphere^{n-1})^k$ be a smooth map such that for all $t \in [0,1]$ and all $x \in M$, the map $f_t:M\rightarrow (\sphere^{n-1})^k$ has the property that $f_t(x)$ is a tuple of linearly independent elements. 

Form the vector subbundle $E_{f_t} \rightarrow M$ of the trivial bundle $\varepsilon^n_{\R} \rightarrow M$ with fiber 
$E_{f_t}|_x=\Span_{\R} f_t(x)$ as well as the bundle splitting 
$\varepsilon^n_{\R}=E_{f_t} \oplus E_{f_t}^\bot$. Then $E_{f_0}^\bot \cong E_{f_{t}}^\bot$ for all $t \in [0,1]$. 
\end{lemma}

\emph{Proof of Lemma \ref{Lem:OrthogonalBundleConvergence}}:
Note that $f$ induces a smooth map
$\Gr_k(f): M \times [0,1] \rightarrow \Gr_k(\R^n)$ 
by $\Gr_k(f_t) = \Span (f_t)$. 
Let $\bot: \Gr_{k}(\R^{n})\rightarrow \Gr_{n-k}(\R^n)$ be the orthogonal complement map. We then obtain a smooth map 
$\Gr_{n-k}(f): M\times [0,1] \rightarrow \Gr_{n-k}(\R^n)$ by $\Gr_{n-k}(f)\coloneqq \bot \circ \Gr_k(f)$. Note that $E_{f_t}^{\bot} \cong \Gr_{n-k}(f)^*(\mathscr{T}^{n-k})$, where $\mathscr{T}^{n-k} \rightarrow \Gr_{n-k}(\R^n)$ is the tautological bundle. That is, $\mathscr{T}^{n-k}$ has fiber at $P \in \Gr_{n-k}(\R^n)$ given by $\mathscr{T}^{n-k}|_{P}=P$. 

Finally, we may conclude. We see that $\Gr_{n-k}(f)$ exhibits a smooth homotopy between $\Gr_{n-k}(f_0)$ and $\Gr_{n-k}(f_{t})$ for any $t \in [0,1]$ and hence 
the pullback bundles $\Gr_{n-k}(f_0)^*(\mathscr{T}^{n-k})$ and $\Gr_{n-k}(f_t)^*(\mathscr{T}^{n-k})$ are smoothly isomorphic as vector bundles 
$\square_{Lemma}$. \medskip 

We can now apply Lemma \ref{Lem:OrthogonalBundleConvergence} in the proof of Theorem \ref{thm:EulerClassXgen}. 

Let $\eta > 0$ be a small positive number. We define a map 
 $\phi_{\epsilon}: \sphere(W) \times [0, 1] \rightarrow (\sphere^{n-1})^3$, where $\phi_{\epsilon}(w)$ is the triple $(u+f(\epsilon)w, (J\psi)(u+f(\epsilon)w), L_0w)$ re-normalized to unit length. The map $\phi_{\epsilon}$ is smooth. By definition, 
$E^{\epsilon} = E|_{S_i(\epsilon)}= E_{\phi_{\epsilon}}^{\bot}$ in the notation of
Lemma \ref{Lem:OrthogonalBundleConvergence}. We then conclude that $E^0 \cong E^{\epsilon}$ for all $\epsilon \in [0,1]$.  

Finally, we show that the limiting vector bundle $E^0 \rightarrow \sphere(W) \cong \sphere^{n-3}$ is isomorphic to the tangent bundle $\T \sphere^{n-3}$. 
In fact, we claim the linear map induced by $L_0$ on the fibers of $E^{0}$ provides an isomorphism between $E^{0}$ and $\T \sphere(W)$. Recall that $\ker(L_0)=\R\{x_0\}$ and $\ker(L_0)^\bot = \mathrm{image}(L_0)$. Hence, for $ w \in \sphere(W)$, 
\begin{align*}
   L_0(E^{0}|_{w}) &= \{ L_{0}y \  | \ y\in \R^n, \langle y, x_{0} \rangle = \langle y, L_0\dot{x}\rangle= \langle y, L_0w \rangle =0 \} \\
    &= \{ L_0y \ | \ y \in \R^n, \langle y, x_{0} \rangle = \langle L_0y, \dot{x}\rangle= \langle L_{0}y, w \rangle =0 \} \\
    &= \{ z \in \ker(L_0)^{\perp} \ | \ \langle z, \dot{x}\rangle= \langle z, w \rangle =0 \} \\ 
    &= \{ z \in \mathbb{R}^{n} \ | \ \ \langle z, x_{0} \rangle = \langle z, \dot{x}\rangle= \langle z, w \rangle=0 \} = \T_{w} \sphere(W) \,.
\end{align*}
We conclude that $\T\sphere(W) \cong E^0 \cong E^{\epsilon}$ as rank $n-3$ vector bundles over $\sphere(W)$, proving (i). 
\medskip 

(ii) The statement about the Euler class follows from naturality \cite{MS74}. Let $j_i: S_i \hookrightarrow S_{gen}$ be the inclusion map inducing the generators $[S_i] \in H_{n-3}(S_{gen},\Z)$. We determined in (i) that $e_i:=e(X_{gen}|_{S_i} \rightarrow S_i) =e(\T\sphere^{n-3}\rightarrow \sphere^{n-3})= \pm 2$, depending on orientation. Now, the total Euler class $\mathbf{e}(\mathcal{R}^\bot \rightarrow S_{gen}) \in H^{n-3}(S_{gen}) \cong \Z^r$ is related to $e_i$ by naturality via
$e_i = j_i^*(\mathbf{e})$. Hence, (i) implies $\mathbf{e}(X_{gen}) = (\pm 2, \pm 2,\dots, \pm2)$ in the basis $([S_i])_{i=1}^{r(n)}$. \end{proof}

\begin{corollary}[Even Case: Product Homology]\label{Cor:TrivialEulerClass} Let $n\geq 6$ be an even integer. Then $H^*(X_{gen}, \Z) \cong H^*(S_{gen} \times \sphere^{n-3},\Z)$ as graded $\Z$-modules. 
\end{corollary}

\begin{proof} 
The result follows from the Gysin sequence and the vanishing of the Euler class. 
\end{proof}

\subsection{High Connectivity}

Let $n\geq 4$. In this section, we show that the base of pencil $X = \hat{\mathcal{B}}(\mathcal{P})$ is highly connected. 

To prove the connectedness of $X$, we shall work with a homotopy equivalent space. 
Recall that for a Fuchsian-Hitchin representation $\rho_0: \pi_1S \rightarrow \PSL(n,\R)$, we have cocompact domain of discontinuity $\Omega = \Omega_{\rho_0} \subset \mathcal{F}_{1,n-1}$ given by \eqref{Omega_Thick}. Pulling back $\Omega$ by the $\Z_2\times\Z_2$-covering map $V_2(\R^n) \rightarrow \mathcal{F}_{1,n-1}$, we obtain an open domain $\hat{\Omega}$ in the Stiefel manifold $V_2(\R^n)$.  
In Corollary \ref{Cor:NearestPointProjUpstairs}, we saw that $\hat{\Omega}$ fibers over $\Ha^2$ with fiber $\hat{\mathfrak{F}}$ diffeomorphic to $X$. Hence, we have a homotopy equivalence $\hat{\Omega}\simeq X$. Thus, it suffices to prove that $\hat{\Omega}$ is highly connected. 

Recall from $\S$\ref{Sec:TitsDomain} that $\hat{\Omega}$ is the complement of the thickening $\hat{K} \subset V_2(\R^n)$ of the limit set $\sphere^1\cong\Lambda \subset \Flag(\R^n)$ of $\rho_0$.
To conclude the high connectivity $\hat{\Omega}$, we will need an upper bound on the cohomological dimension of the thickening $\hat{K}$, which the following proposition provides. We state the result in $V_2(\R^n)$, but the analogous result also holds in $\mathcal{F}_{1,n-1}$. 

\begin{proposition}[Thickening is CW]\label{Prop:ThickeningCW}
Choose any full flag $F^{\bullet}\in \Flag(\R^n)$. Then 
\begin{enumerate}[label=(\roman*)]
    \item The thickening $\hat{K}_{F^{\bullet}}$ admits the structure of an $(n-2)$-dimensional CW complex. 
    \item Let $\xi:\sphere^1 \rightarrow \Flag(\R^n)$ be a continuous, transverse map and define $\Lambda = \im(\xi)$. 
    Then $\hat{K}_{\Lambda}$ admits the structure of an $(n-1)$-dimensional CW complex.
\end{enumerate}
\end{proposition}

\begin{proof}
(i) It is instructive to break $\hat{K}_{F^{\bullet}}$ into pieces. Write $F^{\bullet} =\big[F^1 \subset F^2 \subset \cdots \subset F^{n-1}]$. Then we may define the thickening of an $i$-plane $F^i$ as follows: 
\[ \hat{K}_{F^i} = \{ (u,v) \in V_2(\R^n) \mid \langle u\rangle \subseteq F^i \subseteq \langle v\rangle^\bot \}.\]
By construction, the thickening $\hat{K}_{F^{\bullet}}$ of the full flag $F^{\bullet}$ is the union of the thickenings of its component subspaces: 
\[ \hat{K}_{F^{\bullet}} = \bigcup_{i=1}^{n-1} \hat{K}_{F^{i}}.\]
The topology of each piece $\hat{K}_{F^i}$ is easily identifiable: $\hat{K}_{F^i} \cong \sphere^{i-1}\times \sphere^{n-i-1}$. However, we can describe $\hat{K}_{F^i}$ more geometrically relative to a choice of background Euclidean metric $g$ on $\R^n$, which we fix once-and-for-all for the following discussion. 
In particular, $\hat{K}_{F^i}$ is canonically identified with $\sphere(F^i) \times \sphere( (F^i)^{\bot_g})$. 
Using these identifications, we make a small observation. Note that for any indices $1\leq i <j\leq n-2$, we have 
\begin{align}\label{NestedThickenings}
    \hat{K}_{F^i} \cap \hat{K}_{F^{j+1}} \subseteq \hat{K}_{F^j} \cap \hat{K}_{F^{j+1}}. 
\end{align}
Moreover, one easily sees that 
\begin{align}\label{ThickeningIntersection}
     \hat{K}_{F^i} \cap \hat{K}_{F^{i+1}}\cong \sphere^{i-1} \times \sphere^{n-i-2}. 
\end{align}
With this observation regarding intersections in hand, we can inductively place a CW structure on $\hat{K}_{F^{\bullet}}$. Indeed, $\hat{K}_{F^1} \cong \sphere^0 \times \sphere^{n-2}$ has its usual product CW structure. We then build a CW structure on $\hat{K}({j+1}) = \bigcup_{i=1}^{j+1} \hat{K}_{F^{i}}$ from that of $\hat{K}(j)=\bigcup_{i=1}^j\hat{K}_{F^i}$ as follows. We first place the usual product CW structure on $\hat{K}_{F^{j+1}}$. By \eqref{NestedThickenings}, \eqref{ThickeningIntersection}, $\hat{K}(j)$ and $\hat{K}_{F^{j+1}}$ intersect exactly along $\sphere(F^j)\times \sphere((F^j)^\bot)$. 
Take a refinement of the CW structures on $\hat{K}(j)$ and $\hat{K}_{F^{j+1}}$ such that the  intersection becomes a subcomplex of each space. Upon gluing, we obtain a CW structure on $\hat{K}(j+1)$. At the end of this inductive process, we obtain an $(n-2)$-dimensional CW structure on $\hat{K}_{F^{\bullet}} = \hat{K}(n-1)$. 
\medskip 

(ii) Recall that full flags $F_1^{\bullet}, F_2^{\bullet}$ are \emph{transverse} when 
for any indices $1\leq i,j\leq n$, the intersection $F_1^i \cap F_2^j$ has the minimum possible dimension, namely $\max \{0, i+j-n\}$. Using transversality, one shows directly that $F_1^{\bullet} \pitchfork F_2^{\bullet}$ implies $\hat{K}_{F_1^{\bullet}} \cap \hat{K}_{F_2^{\bullet}} = \emptyset$. 

Now, let $\xi: \sphere^1 \rightarrow \Flag(\R^n)$ be a continuous, transverse map. We obtain a fiber bundle $E\rightarrow \sphere^1$ with fiber $E_{x} = \hat{K}_{\xi(x)}$. That is to say, $E = \{ (x, \bm{u}) \in \sphere^1 \times V_2(\R^n) \mid \bm{u} \in \hat{K}_{\xi(x)} \}$. There is a continuous injective map $\iota: E \rightarrow V_2(\R^n)$ by $\iota(x,\bm{u}) = \bm{u}$. Hence, $\iota$ can be viewed as a homeomorphism between $E$ and $ \hat{K}_{\Lambda}$. Since the fiber $\hat{K} = \hat{K}_{F^{\bullet}}$ and the base $\sphere^1$ of $E$ each admit the structure of CW complex, the total space also admits a CW structure of dimension $n-1 = \dim(\sphere^1)+ \dim(\hat{K})$. 
\end{proof}

We shall obtain the connectedness of $\hat{\Omega}$ from the following proposition. 

\begin{proposition}[Homology of Complement]\label{Prop:HomologyComplement}
Suppose that $M$ is a closed $n$-dimensional smooth manifold and $A$ is a closed subset with the homotopy type of an $(n-j)$-dimensional CW complex. Then for $k\leq j-2$, 
\[ H_{k}(M \setminus A) \cong H_{k}(M). \]
\end{proposition}

\begin{proof}
We shall use Poincar\'e-Alexander-Lefschetz duality \cite[Chapter VI, Theorem 8.3]{Bre93}. Given our assumptions on $M$, this theorem furnishes us the following isomorphism of (singular) homology and cohomology groups\footnote{To be more precise, the theorem uses \v{C}ech cohomology for $A$. However, since $A$ is homotopy equivalent to a CW complex, singular cohomology with integer coefficients and \v{C}ech cohomology agree.}, using the coefficient ring $\Z$: 

\[ H_{k}(M, M\setminus A) \cong H^{n-k}(A).\]
Now, by the long-exact sequence of the pair $(M, M\setminus A)$ in homology, we have
\[
    H^{n-k-1}(A)\cong H_{k+1}(M,M\setminus A) \rightarrow H_{k}(M\setminus A) \rightarrow H_{k}(M) \rightarrow H_{k}(M, M\setminus A)\cong H^{n-k}(A).
\]
Thus, $H_k(M) \cong H_k(M\setminus A)$ for $k \leq j-2$. 
\end{proof}

\begin{corollary}\label{cor:homology-vanishes}
Suppose that $n \geq 5$. Then $X = \hat{\mathcal{B}}(\mathcal{P})$ has $\tilde{H}_k(X, \Z)=0$ for $0 \leq k \leq n-4$. 
\end{corollary}

\begin{proof}
Since $X$ is homotopy equivalent to $\hat{\Omega}$, we have $H_{*}(\hat{\Omega}) \cong H_*(X)$. 
Proposition \ref{Prop:HomologyComplement} applied with $M= V_2(\R^n)$ and $A = \hat{K}_{\Lambda}$ yields $H_{k}(\hat{\Omega}) \cong H_k(M)$ for $k \leq n-4$. Since the Stiefel manifold $V_2(\R^n)$ is $(n-3)$-connected, we conclude that $\tilde{H}_k(\hat{\Omega})=0$ for $0 \leq k \leq n-4$. 
\end{proof}

We now verify that $X$ is simply connected to enable an application of Hurewicz's theorem. 
\begin{lemma}[Simple Connectivity]\label{lem:simplyconnected}
Let $n \geq 5$. Then the manifold $X$ is simply connected. 
\end{lemma}

\begin{proof}
We again denote $p: X \rightarrow \sphere^{n-1}$ as the almost fibration. 

Label the singular circles $S_{sing} = \bigsqcup_{i=1}^{r(n)} C_i$. It will be convenient to define 
\[X_l \coloneq p^{-1}\big(S_{gen} \sqcup \bigsqcup_{i=1}^l C_i\big),\] 
for any integer $0\leq l\leq r(n)$. That is, $X_l$ is obtained by adding to $X_{gen}$ the total space over the first $l$ singular circles, hence $X= X_{r(n)}$. We denote by $A_i$ a small regular neighborhood of $p^{-1}(C_i)$ in $X$ satisfying $A_i \cap X_{sing} = A_i$. 
Recall the fibration $\sphere^{n-4} \rightarrow X_{gen}\rightarrow S_{gen}$.

Now, when $n \geq 6$, the space $X_{gen}$ is simply connected, as witnessed by the long exact sequence of homotopy groups. In this case, we notice that for any index $1\leq i \leq r(n)$, the inclusion map $k: X_{gen} \cap A_i\rightarrow A_i$ induces an isomorphism 
$k_*: \pi_1(X_{gen}\cap A_i)\rightarrow \pi_1(A_i)$ of the cyclic groups
$\pi_1(X_{gen} \cap A_i) \cong \Z \cong \pi_1(A_i)$. 

In the case $n=5$, things are similar, but $X_{gen}$ is no longer simply connected. By Theorem \ref{thm:EulerClassXgen}, in this case, the Euler class of $e(X_{gen}) \in H^2(S_{gen},\Z)$ is $\mathbf{2}=(2,2,\dots, 2)$ in an appropriate basis for homology. One then finds $\pi_1(X_{gen})\cong \Z_2$. On the other hand, for any index $i$, we have $\pi_1(X_{gen} \cap A_i) \cong \Z\oplus \Z_2$. 
Here, the induced homomorphisms $j_*:\pi_1(X_{gen}\cap A_i) \rightarrow \pi_1(X_{gen})$ and $k_*:\pi_1(X_{gen}\cap A_i) \rightarrow \pi_1(A_i)$ are just the projections onto each factor of the domain. 

Let $l \geq 1$ and $n \geq 5$ be arbitrary. Consider the open cover $X_l = X_{l-1} \cup A_l$, with connected intersection $X_{l-1} \cap A_l = X_{gen} \cap A_l$, a punctured regular neighborhood of $p^{-1}(C_l)$. 
An inductive application of Seifert-Van Kampen easily shows $X_1, X_2,\dots,X_{r(n)}=X$ are all simply connected. 
\end{proof}

\begin{corollary}\label{cor:highly-connected}
Let $n\geq 5$. Then $X = \hat{\mathcal{B}}(\mathcal{P})$ is $(n-4)$-connected. 
\end{corollary}

\begin{proof}
When $n=5$, the claim holds by Lemma \ref{lem:simplyconnected}. Otherwise, $n\geq 6$, and the claim follows by Hurewicz's theorem, Lemma \ref{lem:simplyconnected}, and Corollary \ref{cor:homology-vanishes}.
\end{proof}

\subsection{Computation of Homology}\label{Sec:HomologyX}

Let $X $ be the total space of our almost-fibration $p: X \rightarrow \sphere^{n-1}$. We wish to apply Mayer-Vietoris to compute the homology of $X$. We shall do so using the open covering $X = A \cup B$, where $A$ and $B$ are as follows:
\begin{itemize}
    \item $A$ is a regular neighborhood of $X_{sing} =p^{-1}(S_{sing})$ in $X$. 
    \item $B = X_{gen}$, an $\sphere^{n-4}$-bundle over $S_{gen}$.
\end{itemize}
Note that $A \cap B$ is homotopy equivalent to an $\sphere^{n-4}$-bundle over $ \bigsqcup_{i=1}^{r(n)} (\sphere^1\times \sphere^{n-3})$. 
We handle the general case $n\geq 6$ in this section and treat $n=5$ separately in $\S$\ref{Sec:n=5}. \medskip 

We can use the Gysin sequences of $X_{gen}\rightarrow S_{gen}$ and $X_{gen}|_{S_i}\rightarrow S_i$ to compute the homology of $B$ and of $A \cap B$. 
Here, we state only the relevant homology groups.

\begin{lemma}[Homology of Intersection]\label{Lem:HomologyIntersection} 
Let $n \geq 6$. For indices $2 \leq k\leq n-2$, the homology groups of $A \cap B$ are as follows: 
\begin{enumerate}[label=(\roman*)]
    \item When $n$ is odd, \begin{align}\label{HomologyOfIntersection}
    H_k(A \cap B, \Z) = \begin{cases}
        0 &  2\leq k \leq n-5,\\
        \Z_{2}^{r(n)} & k \in \{n-4, n-3\},\\
        0 & k=n-2.\\
    \end{cases}
    \end{align}
    \item When $n$ is even, the only changes are $H_k(A \cap B,\Z) = \Z^{2r(n)}$ for $k = n-3$ and $H_{k}(A\cap B, \Z)\cong \Z^{r(n)}$ for $k \in \{n-2,n-4\}$.
\end{enumerate}
\end{lemma}

\begin{proof} 
We compute via cohomology, then translate back. Recall that for any topological space $Y$ with the homotopy type of a finite CW complex, the groups $H_k(Y,\Z)$ and $H^k(Y,\Z)$ have the same rank and their torsion subgroups $T_k,T^k$ are related by $T_k =T^{k+1}$. \medskip  

(ii) The claim follows from Theorem \ref{thm:EulerClassXgen}, by the same reasoning as in Corollary \ref{Cor:TrivialEulerClass}. \medskip 

(i) The Gysin sequence shows $H^k(A \cap B) \cong H^k(\bigsqcup_{i=1}^{r(n)}\sphere^1\times 
\sphere^{n-3})=0$ for indices $0\leq k\leq n-5$. 
To compute the remaining cohomology groups, it suffices to compute on each connected component, then take $r(n)$ direct sums. Thus, we set $U_i \subset \sphere^{n-1}$ to be a punctured regular neighborhood of $C_i \subset S_{sing}$ and $E_i = p^{-1}(U_i)$ to be the $\sphere^{n-4}$-bundle over $U_i$.
Note that $U_i$ is homotopy equivalent to $\sphere^1\times \sphere^{n-3}$. 
By Theorem \ref{thm:EulerClassXgen}, the Gysin sequence furnishes the following exact sequence: 
\begin{center}
\begin{tikzcd}[scale=2, row sep=2.5em]
      &
    0 \arrow{r} &
    H^{n-4}(E_i) 
    \arrow[dll, rounded corners=8pt, to path={ 
        -- ([xshift=3ex]\tikztostart.east) 
        |- ([yshift=2ex]\tikztotarget.north) 
        -- (\tikztotarget) 
    }] & \\
    |[label=below:{\substack{\cong \\ \\ \mathbb{Z}}}]| H^0(U_i) \arrow{r}{\cdot 2} &
    |[label=below:{\substack{\cong \\ \\ \mathbb{Z}}}]| H^{n-3}(U_i) \arrow{r}{p^*} &
    H^{n-3}(E_i) 
   \arrow[dll, rounded corners=8pt, to path={ 
        -- ([xshift=3ex]\tikztostart.east) 
       |- ([yshift=2ex]\tikztotarget.north) 
        -- (\tikztotarget) 
    }] \\
    |[label=below:{\substack{\cong \\ \\ \mathbb{Z}}}]| H^1(U_i) \arrow{r}{\cdot 2} &
    |[label=below:{\substack{\cong \\ \\ \mathbb{Z}}}]| H^{n-2}(U_i) \arrow{r}{p^*} &
    H^{n-2}(E_i) 
    \arrow[dll, rounded corners=8pt, to path={ 
        -- ([xshift=3ex]\tikztostart.east) 
        |- ([yshift=2ex]\tikztotarget.north) 
        -- (\tikztotarget) 
    }] & \\
    H^2(U_i) =0.
\end{tikzcd}
\end{center}
A very straightforward diagram chase shows $H^{n-4}(E_i)=0$ and $H^{n-2}(E_i)\cong \Z_2 \cong H^{n-3}(E_i)$. 
A similar calculation with the Gysin exact sequence verifies $H^{n-1}(E_i)=0$. 
\end{proof}

Next, we compute the relevant homology of the generic locus $X_{gen}$. 
\begin{lemma}[Homology of Generic Locus]\label{Lem:CohomologyGenericLocus}
Let $X_{gen}$ be the generic locus of $X = \hat{\mathcal{B}}(\mathcal{P})$. For indices $1 \leq k \leq n-2$, the homology of $X_{gen}$ is as follows: 
\begin{enumerate}[label=(\roman*)]
    \item For $n\geq 7$ odd,
\begin{align}\label{CohomologyGenericLocus}
    H_k(X_{gen},\Z) = \begin{cases}
        0 &  1\leq k \leq n-5,\\
        \Z_2 & k= n-4,\\
        \Z^{r(n)-1} & k \in \{ n-3,n-2\}, \\
    \end{cases}
\end{align}
\item For $n \geq 6$ even, the only changes are  $H_{n-3}(X_{gen},\Z)\cong \Z^{r(n)}$ and $H_{n-4}(X_{gen},\Z)\cong \Z$.
\end{enumerate}
\end{lemma}

\begin{proof}
We shall compute in cohomology with $\Z$-coefficients as in Lemma \ref{Lem:HomologyIntersection}. Recall from $\S$\ref{Sec:SgenHomology} that
$S_{gen}$ is homotopy equivalent to $\big(\bigvee_{i=1}^{r(n)}\sphere^{n-3}\big) \vee \big(\bigvee_{i=1}^{r(n)-1}\sphere^{n-2} \big)$.

Let $n \geq 6$ be any integer. For the indices $1\leq k\leq n-5$, the Gysin sequence shows $H^{k}(X_{gen})\cong H^k(S_{gen})=0$. Similarly, using the Gysin sequence one easily finds that 
$H^{n-2}(X_{gen})\cong H^{n-2}(S_{gen})\cong \Z^{r(n)-1}$. 

Next, we consider the middle cohomology groups for indices $k \in \{n-4,n-3\}$. The Gysin sequence provides the following exact sequence:

\[ 0 {\longrightarrow} H^{n-4}(X_{gen}) \stackrel{\partial}{\longrightarrow} H^0(S_{gen}) \stackrel{\cup e}{\longrightarrow} H^{n-3}(S_{gen}) \stackrel{p^*}{\longrightarrow} H^{n-3}(X_{gen}) \longrightarrow 0.\]
Finally, we conclude the desired result in cases. \medskip 

(i) If $n$ is odd, we see immediately that $H^{n-4}(X_{gen})= 0$ since the displayed map $\cup e$ is injective. Now, Theorem \ref{thm:EulerClassXgen}, shows $e(X_{gen}) =(2,2,\dots, 2) \in H^{n-3}(S_{gen}, \Z)\cong \Z^{r(n)}$, for appropriate choice of generators. Exactness then gives 
\[ H^{n-3}(X_{gen})\cong \coker(\cup e)  \cong \Z^{r(n)-1}\oplus \Z_{2} .\] 

(ii) If $n$ is even, then the Euler class $e(X_{gen})$ vanishes by Theorem \ref{thm:EulerClassXgen}. Thus, in this case, $H^{n-4}(X_{gen})\cong \Z$ and $H^{n-3}(X_{gen})\cong \Z^{r(n)}$. 
\end{proof}

Before we compute the homology $H_*(X)$, we shall introduce another useful lemma. A crucial step in all cases, even and odd, will be to understand how certain generators of $H_{n-2}(A \cap B)$ include in $H_{n-2}(A)$. The next technical lemma will aid in this endeavor.

Let us introduce some notation for the lemma. Denote $(A\cap B)_i$ as the connected component of $A \cap B$ intersecting a neighborhood of $C_i$. We obtain spherical cylinders $P_i \cong \sphere^1 \times \sphere^{n-4}$ in $(A \cap B)_i$, 
by letting $\gamma_i$ be a homotopically nontrivial circle in a punctured regular neighborhood of the singular circle $C_i$ and 
defining $P_i:= p^{-1}(C_i)$. Note that $P_i$ is a product by orientability of $p: X_{gen} \rightarrow S_{gen}$. 
We shall trade out $P_i$ for a partner $P_i'$ in $A$ that is homologically equivalent and instead contained in $X_{sing}$.

\begin{lemma}[Replacing Spherical Cylinders]\label{Lem:Replace}
Let $n \geq 5$. Let $C_i$ be any singular circle. Then $P_i$ is homologous in $A$ to a trivial $\sphere^{n-4}$-fiber sub-bundle $P_i'$ of $\sphere^{n-3}\rightarrow p^{-1}(C_i) \rightarrow C_i$. 
\end{lemma}
\begin{proof}
To produce the submanifold $P_i'$, we consider the behavior of the bundle $X_{gen} \rightarrow S_{gen}$ as points in the base approach the singular locus. 
Let $u_0 \in C_i$ be a point on the singular circle $C_i$. By definition, there is a triple $(\lambda,c_{1},c_{2}) \in \R^3$ such that $\lambda \neq0$, $\bm{c} \neq \bm{0}$, and 
\[ -\lambda u_{0}+(c_1\psi_{1}+c_{2}\psi_{2})u_{0}=0.\]
Let $\eta \in \R^n$ be a vector which is not in the kernel of $\lambda\Id+c_{1}\psi_{1}+c_{2}\psi_{2}$. 
For any vector $v \in \sphere^{n-1}$, we shall denote $W_v:= \Span \{ v, \psi_1(v), \psi_2(v)\}$. 
Now, suppose we approach $u_{0}$ along the path $u(t)=\frac{u_{0}+t\eta}{||u_0+t\eta||}$. Observe that the $3$-dimensional spaces 
$W_{u(t)}$ converges as $t\to 0$ to $W_{u(0)} \oplus L_{\eta}$, where $L_{\eta}$ is the line spanned by $-\lambda\eta+c_{1}\psi_{1}\eta+c_{2}\psi_{2}\eta$. Now, the fiber of $X_{gen} \rightarrow S_{gen}$ at $u(t)$ is $ \sphere\big( W_{u(t)}^\bot\big)$, the unit sphere in the orthogonal complement of $W_{u(t)}$ in $\R^n$. 
The above reasoning implies $W_{u(t)}^\bot$ converges to $(W_{u(0)} \oplus L_{\eta})^\bot$ in $\Gr_{n-3}(\R^n)$ as $t\rightarrow 0$. 
Interpreting this convergence appropriately, we can produce $P_i'$. Indeed, let $\eta$ now be a unit normal vector field to $C_i$ in $\sphere^{n-1}$. For $\varepsilon >0$ sufficiently small, we construct $\gamma_i$ as the image of the map $C_i\rightarrow \sphere^{n-1}$ given by $x \mapsto \exp_{x}(\varepsilon \eta(x))$ and define $P_i := p^{-1}(\gamma_i)$. By the geometric convergence described in the previous paragraph, the submanifold $P_i$ is homologous in $A$ to the submanifold $P_i'$, built as a sphere bundle $\sphere^{n-4} \rightarrow P_i' \rightarrow C_i$, where $P_i'$ has fiber 
\[P_i'|_{x} = \sphere((W_x \oplus L_{\eta(x)})^\bot) .\] 

To finish the proof, we must show that $P_i'$ is a trivial fiber bundle over $C_i \cong \sphere^1$. Recall the splitting $\T\sphere^{n-1} =\mathcal{R} \oplus \mathcal{R}^\bot$ from Lemma \ref{Lem:StructureLemma}, where 
$\mathcal{R}|_x = \pi_{x^\bot}(\mathcal{P}(x))$. Recall that $\mathcal{R}|_{C_i}$ and $\mathcal{R}^\bot|_{C_i}$ are vector bundles of ranks $1$ and $n-2$, respectively, over the circle $C_i$. 
In this notation, $P_i'$ is the sphere bundle of the vector bundle $(L_{\eta}^\bot \cap \mathcal{R}^\bot)|_{C_i}$. 
Thus, the proof is complete if the vector bundle $L_{\eta}^\bot \cap \mathcal{R}^\bot \rightarrow C_i$ is orientable. To this end, we shall introduce a few auxiliary bundles. 
Inspired by the proof of Lemma \ref{lem:number_circles}, we consider the line bundle $\mathbb{E}(\mathcal{P}) \rightarrow C_i$ with fiber at $x$ given by $\mathbb{E}(\mathcal{P})|_x = \{ \psi \in \mathcal{P} \mid \exists \lambda \in \R , \psi(x) = \lambda x\}$. 
We may define a vector bundle morphism $\phi:\mathbb{E}(\mathcal{P}) \rightarrow L_{\eta}$ by $(x,\psi)\mapsto (x, -\lambda\eta + \psi(x))$. As a non-vanishing map, $\phi$ is an isomorphism. 
Now, we consider the trivial bundle $\underline{\mathcal{P}} \rightarrow C_i$ given by $\underline{\mathcal{P}} = C_i\times \mathcal{P}$. Equip $\underline{\mathcal{P}}$ with an arbitrary Euclidean metric and split  
$\underline{\mathcal{P}} = \mathbb{E}(\mathcal{P})\oplus \mathbb{E}(\mathcal{P})^\bot$. Recall the isomorphism $\mathbb{E}(\mathcal{P})^\bot \cong \mathcal{R}|_{C_i}$
from the proof of Corollary \ref{cor:orientability_singular}. Now, since $\underline{\mathcal{P}}$ is trivial, the line bundle $\mathcal{R}|_{C_i} \rightarrow C_i$ is orientable if and only if $L_{\eta}\rightarrow C_i$ is orientable. Using this equivalence, we can now verify triviality of the fiber bundle $P_i' \rightarrow C_i$.\medskip 

\textbf{Case 1: $\mathcal{R}^\bot \rightarrow C_i$ is orientable.} 
In this case, $\mathcal{R}|_{C_i} \rightarrow C_i$ is orientable. 
Hence, $L_{\eta} \rightarrow C_i$ is orientable and thus its orthogonal complement in $\mathcal{R}^\bot$ is orientable. \medskip 

\textbf{Case 2: $\mathcal{R}^\bot \rightarrow C_i$ is non-orientable}. In this case, $\mathcal{R}|_{C_i} \rightarrow C_i$ is non-orientable and thus $L_{\eta} \rightarrow C_i$ is non-orientable. Thus, its orthogonal complement in $L_{\eta}^\bot \subset \mathcal{R}^\bot$ is orientable. \medskip 

In all cases, $\sphere^{n-4} \rightarrow P_i' \rightarrow C_i$ is orientable and thus trivial.
\end{proof}

We can now understand the consequences for the inclusion of $P_i$ in homology. 
\begin{corollary}[Spherical Cylinders Under Inclusion]\label{cor:spherical_inclusion}
Let $n \geq 5$. Then for $C_i \subset S_{sing}$ a singular circle, 
\begin{enumerate}[label=(\roman*)]
    \item If the fiber bundle $p^{-1}(C_i)\rightarrow C_i$ is orientable, then $i_*: H_{n-3}(A \cap B,\Z) \rightarrow H_{n-3}(A,\Z)$ satisfies $i_{*}([P_{i}])=0$. 
    \item If the fiber bundle $p^{-1}(C_i) \rightarrow C_i$ is non-orientable, then $i_*:H_{n-3}(A \cap B, \Z_2) \rightarrow H_{n-3}(A,\Z_2)$ satisfies $i_{*}([P_{i}])\neq 0$. 
\end{enumerate}
\end{corollary}
\begin{proof} (i) In the orientable case, Lemma \ref{Lem:Replace} immediately implies $i_{*}([P_{i}])=[P_i']= 0$. 

(ii) Suppose that $p^{-1}(C_i)\rightarrow C_i$ is non-orientable. 
Let us denote $X_{i}:= p^{-1}({C_i})$. Since $P_i'\subset X_i$, it suffices to prove that $[P_i']\neq0$ in $H_{n-3}(X_i, \Z_2)$. 
By Lemma \ref{Lem:Replace}, $P_{i}'$ is an embedded codimension one submanifold of $X_i$ with normal bundle isomorphic to the pullback bundle of $L_{\eta}$ under the projection $P_i' \rightarrow C_i$. In particular, the normal bundle of $P_i'$ in $X_i$ is non-orientable and hence non-trivial. 

By Poincar\'e duality and the universal coefficient theorem, we have
\begin{align}\label{eq:PD}
     H_{n-3}(X_i, \Z_{2})\cong H^{1}(X_i,\Z_{2}) \cong\Hom(\pi_{1}(X_i), \Z_{2}).
\end{align}
The class $[P_{i}'] \in H_{n-3}(X_i, \Z_{2})$ corresponds to the homomorphism $\phi_{P_{i}'}:\pi_{1}(X_i)\rightarrow \Z_{2}$ taking the intersection number modulo 2:
\[
   \phi_{P_{i}'}(\, [\gamma]\,)= |\gamma_T \cap P_{i}'| \bmod 2,
\]
where $\gamma_T \in [\gamma]$ is any representative intersecting $P_i'$ transversely. 
Now, let us denote $\mathcal{L}$ for the normal line bundle of $P_i'$ in $X_i$. By non-orientability of $\mathcal{L}$, there is a closed loop $\gamma:[0,1]\rightarrow P_i'$ with orientation-reversing holonomy. 
Let $s$ be a unit section of $\gamma^*\mathcal{L}$. Then define $\hat{\gamma}: [0,1]\rightarrow P_i'$ by $\hat{\gamma}(t) = \exp_{\gamma(t)}( \varepsilon s(t))$. For $\varepsilon >0$ sufficiently small, note that $\hat{\gamma}$ does not intersect $P$. However, we can then construct $\alpha =\hat{\gamma} * \beta$ for a path $\beta$ in a local coordinate neighborhood of $\gamma(0)=\gamma(1)$ from $\hat{\gamma}(0)$ to $\hat{\gamma}(1)$ that crosses $P_i'$ exactly once. Hence, $\phi_{P_i'}([\alpha])\neq0$. By \eqref{eq:PD}, this means $i_*([P_i])\neq0$. 
\end{proof}

We now have enough information to compute the homology of $X$ for $n$ even. 

\begin{theorem}[Homology in Even Case]\label{thm:hom_even}
Let $n \geq 6$ be even. Then the $(n-4)$-connected $(2n-5)$-manifold $X$ has torsion-free integer homology and Betti numbers $b_{n-2}=b_{n-3}=n-1$.
\end{theorem}
\begin{proof}

We compute only with $\Z$ coefficients in this proof. 
By Corollary \ref{cor:highly-connected}, $X$ is $(n-4)$-connected. Thus, $H_k(X)$ is nonzero only for $k \in \{0,n-3,n-2,2n-5\}$. Note that, by Poincar\'e duality and the universal coefficient theorem, the group $H_{n-2}(X,\Z)$ is a free $\Z$-module with the same rank as $H_{n-3}(X,\Z)$.  We write $H_{n-3}(X) = \Z^{b_{n-3}}\oplus T_{n-3}$. 

We shall compute $H_{n-3}(X)$ using the Mayer-Vietoris exact sequence: 
\begin{align*}
    \cdots \rightarrow  H_{s+1}(X) \xrightarrow{\partial_{s+1}} H_{s}(A\cap B) \xrightarrow[\Phi_{s}]{(i_{*},j_{*})} H_{s}(A) \oplus H_{s}(B) \xrightarrow{k_{*}-l_{*}} H_{s}(X) \xrightarrow{\partial_{s}} H_{s-1}(A\cap B) \rightarrow \cdots \ 
\end{align*}
We can extract the following short exact sequence:
\begin{align}\label{eq:HomologyEvenMiddle}
    0 \longrightarrow \coker(\Phi_{n-3}) \longrightarrow H_{n-3}(X) \xrightarrow{\partial_{n-3}} \im(\partial_{n-3}) \longrightarrow 0.
\end{align}  
By exactness, $\im(\partial_{n-3})=\ker(\Phi_{n-4})$. Note that $H_{n-4}(X)=0$ implies $\Phi_{n-4}$ surjects. By Lemmas \ref{Lem:HomologyIntersection}(ii) and \ref{Lem:CohomologyGenericLocus}(ii), we then conclude that $\ker(\Phi_{n-4})\cong \Z^{r(n)-1}$. 
As a consequence, the exact sequence \eqref{eq:HomologyEvenMiddle} must split. 

We need to understand the map
\[
    \Phi_{n-3}:H_{n-3}(A\cap B) \rightarrow H_{n-3}(A) \oplus H_{n-3}(B) \ .
\]
Recall $A \cap B $ is homotopy equivalent to an $\sphere^{n-4}$-bundle over $\bigsqcup_{k=1}^{r(n)}(\sphere^1\times \sphere^{n-3})$ with trivial Euler class. Thus, generators of $H_{n-3}(A\cap B)$ come in two flavors: 
\begin{itemize}[noitemsep]
    \item $r(n)$ spheres $\hat{S_i} \cong \sphere^{n-3}$, constructed as sections of $X_{gen} \rightarrow S_{gen}$ across the spheres $S_i \subset S_{gen}$ generating $H_{n-3}(S_{gen})$,
    \item $r(n)$ spherical cylinders $P_{i}\cong \sphere^{1}\times \sphere^{n-4}$ obtained by $P_i := p^{-1}(\gamma_i)$, where $\gamma_i \subset S_{gen}$ is a circle parallel to the singular circle $C_i$.
\end{itemize} 

By Corollary \ref{cor:spherical_inclusion}, we know that $i_*([P_i])=0$. It is 
also clear that $P_i$ is nullhomologous in $B=S_{gen}$ since the $\sphere^1$-factor of $P_{i}$ corresponds to a loop $\gamma_i $ in the simply connected space $S_{gen}$.
Altogether, we find $P_{i} \in \ker(\Phi_{n-3})$.
Now, $H_{n-3}(B)$ is generated by $[\hat{S}_i]$ for $1\leq i\leq r(n)$. 
Thus, $\coker(\Phi_{n-3})\cong H_{n-3}(A) \cong \Z^{r(n)}$. By \eqref{eq:HomologyEvenMiddle}, we conclude that $H_{n-3}(X,\Z) \cong \Z^{2r(n)-1}$. Lemma \ref{lem:number_circles} provides the relation $2r(n)-1=n-1$, which completes the proof.
\end{proof}

To compute the homology of $X$ in the odd case, we need one more piece of information: the homology of the singular locus $X_{sing}$. In the case $n \equiv 3 \bmod 4$, $X_{sing}$ is simply the disjoint union of $r(n)$ copies of $\sphere^{1}\times \sphere^{n-3}$, so its homology is easily computed, just as in the case of $n$ even. However, when $n \equiv 1 \bmod 4$, we need to be more careful because the fibration $X_{sing} \rightarrow S_{sing}$ is non-trivial, since $S_{sing}^{no}\neq \emptyset$. For the latter, we have the following lemma. 

\begin{lemma}[Homology of Singular Locus, Exceptional Case]\label{lem:homology_singular}
Let $n\geq 5$ be congruent to $1 \bmod 4$. The integer homology of $X_{sing}$ is as follows: 
\[ H_{k}(X_{sing},\Z) = \begin{cases} \Z^{r(n)} & k\in \{0,1\},\\ 
                            0 & 2\leq k\leq n-4,\\
                         \Z^{r(n)-2} \oplus \Z_2^2 & k = n-3, \\
                            \Z^{r(n)-2} & k = n-2 .\\
\end{cases}
\]
\end{lemma}
\begin{proof} Recall that $X_{sing}\cong(S_{sing}^{o}\times \sphere^{n-3}) \sqcup (S_{sing}^{no} \tilde{\times} ~ \sphere^{n-3})$, with the factor $\sphere^{n-3}$ coming from the fiber $\sphere(\mathcal{R}^\bot) \rightarrow C_i$, and $S_{sing}^{o}\cong\bigsqcup_{i=1}^{r(n)-2}\sphere^{1}$ and $S_{sing}^{no}\cong\sphere^{1} \sqcup \sphere^{1}$. Here, $\tilde{\times}$ denotes the non-trivial twisted product. 
It is thus sufficient to prove that 
\[ 
    H_{k}(\sphere^{1} \tilde{\times} ~ \sphere^{n-3},\Z) = \begin{cases} \Z & k \in \{0,1\},\\ 
                            0 & 2\leq k\leq n-4,\\
                         \Z_2 & k = n-3 ,\\
                            0 & k = n-2. \\
\end{cases}
\]
Let $E=\sphere^{1} \tilde{\times} ~ \sphere^{n-3}$. It is clear that $H_{0}(E)=\Z$ by connectedness and $H_{n-2}(E)=0$ because it is a non-orientable $(n-2)$-manifold. To compute the remaining homology, we realize $E$ as the mapping torus of an orientation-reversing homeomorphism $f:\sphere^{n-3} \rightarrow \sphere^{n-3}$.

\emph{Wang's exact sequence} for this fibration over $\sphere^{1}$ is (see \cite[Example 2.48]{Hat01})
\[
    \cdots \longrightarrow H_k(\sphere^{n-3}) \xrightarrow{\,\id-f_*\,} H_k(\sphere^{n-3}) \longrightarrow H_k(E)
    \longrightarrow H_{k-1}(\sphere^{n-3}) \xrightarrow{\,\id-f_*\,} H_{k-1}(\sphere^{n-3}) \rightarrow \cdots .
\]
This immediately yields $H_{k}(E)=0$ for $2 \leq k \leq n-4$. 
Since $f_*=\id$ on $H_0(\sphere^{n-3})$, it follows that $H_1(E) \cong \Z$. As $f$ reverses orientation, $f_*=-\id$ on $H_{n-3}(\sphere^{n-3})$ and 
\begin{align*}
H_{n-3}(E)
& \cong \coker\bigl(\id-f_*: H_{n-3}(\sphere^{n-3}) \rightarrow H_{n-3}(\sphere^{n-3})\bigr) \cong \mathbb Z_2.
\end{align*}
\end{proof}

We now compute the homology of $X$ in the odd case.

\begin{theorem}[Homology in Odd Case]\label{thm:hom_odd}
Let $n \geq 7$ be odd. The $(n-4)$-connected manifold $X$ has torsion-free integer homology and Betti numbers $b_{n-2}=b_{n-3}=2n-1$. 
\end{theorem}

\begin{proof}
Throughout the proof we again use $\Z$ coefficients for homology, unless otherwise specified. We shall write $H_{k}(X,\Z)=\Z^{b_k}\oplus T_k$, with $T_k$ purely torsion. 

The space $X$ is a closed orientable $(2n-5)$-manifold. Note that orientability follows from Lemma \ref{Lem:StablyTrivial}. 
Hence, $H_{2n-5}(X)=\Z$. By Corollary \ref{cor:highly-connected}, $X$ is $(n-4)$-connected. All that remains is to compute $H_{k}(X)$ for $k \in \{n-2,n-3\}$. Again, by the universal coefficient theorem and Poincar\'e Duality, $T_{n-2}=T_{n-4}=0$. Thus, we need only compute $H_{n-3}(X)$. 

We now implement the strategy outlined, applying Mayer-Vietoris with $A$ a regular neighborhood of $X_{sing}$ and $B = X_{gen}$. Note that the homology of $A$ depends on the congruence class of $n \bmod 4$. Consequently, we treat these two cases separately. \medskip

\textbf{Case 1: $\bm{n \equiv 3 \,\textbf{mod}\, 4}.$} In this case, $A$ deformation retracts onto $X_{sing} \cong \bigsqcup_{i=1}^{r(n)}(\sphere^{1} \times \sphere^{n-3})$ by Corollary \ref{cor:orientability_singular}. We know $H_{*}(B)$ and $H_{*}(A \cap B)$ by Lemmas \ref{Lem:HomologyIntersection} and \ref{Lem:CohomologyGenericLocus}. Mayer-Vietoris then provides the following exact sequence:

\begin{center}
\begin{tikzcd}[scale=1.5, row sep=2.5em]
    |[label=below:{\substack{\cong \\ \\ (\Z_{2})^{r(n)}}}]| H_{n-3}(A \cap B) \arrow{r}{(i_*,j_*)} &
    |[label=below:{\substack{\cong \\ \\  \Z^{r(n)} \oplus \Z^{r(n)-1}}}]| H_{n-3}(A)\oplus H_{n-3}(B)\arrow{r}{k_*-l_*} &
    H_{n-3}(X) 
    \arrow[dll, rounded corners=8pt, to path={ 
        -- ([xshift=3ex]\tikztostart.east) 
        |- ([yshift=2ex]\tikztotarget.north) 
        -- (\tikztotarget) 
    }] \\
    |[label=below:{\substack{\cong \\ \\ (\Z_{2})^{r(n)}}}]| H_{n-4}(A \cap B) \arrow{r}{(i_*,j_*)} &
    |[label=below:{\substack{\cong \\ \\ 0 \oplus \Z_{2}}}]| H_{n-4}(A) \oplus H_{n-4}(B) \arrow{r}{k_*-l_*} &
    H_{n-4}(X)=0 
\end{tikzcd}
\end{center}

The above exact sequence immediately implies $b_{n-3}= 2r(n)-1$. Hence, $b_{n-2}=2r(n)-1$ by Poincar\'e duality. 
The crux of the proof is to show the torsion $T_{n-3}$ vanishes, where Poincar\'e duality and the universal coefficient theorem provide no assistance. Here, we examine the relevant short exact sequence:  
\[ 0 \longrightarrow \Z^{2r(n)-1} \longrightarrow H_{n-3}(X,\Z) \longrightarrow \im(\del_{n-3})\longrightarrow 0. \]
In particular, $\im(\del_{n-3})\cong \Z_{2}^{r(n)-1}$ by exactness. Let us write 
\begin{align}\label{MiddleHomologySplitting}
    H_{n-3}(X, \Z) \cong \Z^{2r(n)-1} \oplus T_{n-3}.
\end{align}
Note that  $\partial_{n-3}$ is injective on 
$T_{n-3}$. Hence, $T_{n-3} \cong  (\Z_{2})^{s}$ for $0 \leq s \leq r(n)-1$. We will show that $s=0$ to conclude that $T_{n-3}=0$. To this end, we compute the homology with $\Z_{2}$-coefficients and show that the rank of $H_{n-3}(X,\Z_{2})$ is $2r(n)-1$.
First, we recall how the rank of $H_{n-3}(X,\Z_{2})$ is related to the rank and the $2$-torsion of $H_{n-3}(X,\Z)$ via the universal coefficient theorem. Since $H_{n-4}(X)=0$, we have
\begin{equation}\label{eq:def_s}
    H_{n-3}(X,\Z_{2}) \cong H_{n-3}(X,\Z)\otimes \Z_{2} \cong \Z_{2}^{2r(n)-1+s} .
\end{equation}
On the other hand, we can independently compute $H_{n-3}(X,\Z_{2})$ using Mayer-Vietoris with $\Z_{2}$-coefficients. 
Coefficients in a field force all short exact sequences to split. Hence, 
\[
    H_{n-3}(X,\Z_{2})\cong\coker(\Phi_{n-3})\oplus \ker(\Phi_{n-4}) ,
\]
where $\Phi_i:H_{i}(A\cap B, \Z_2) \rightarrow H_i(A,\Z_2) \oplus H_i(B,\Z_2)$ is the homomorphism in the Mayer-Vietoris sequence with $\Z_2$-coefficients.
Note that $H_{n-3}(A,\Z_{2}) \cong \Z_{2}^{r(n)}$.

Now, for an orientable sphere bundle $\sphere^r \rightarrow E\stackrel{p}{\rightarrow} Y$, with Euler class $e(E) \in 2H^{r+1}(Y,\Z)$, 
the Gysin sequence implies that (non-canonically) $H_*(E,\Z_2) \cong H_*(Y \times \sphere^r,\Z_2)$. We shall use that there exists an isomorphism 
\[\phi=(\alpha,\beta):H_*(E,\Z_2) \rightarrow H_*(Y,\Z_2) \oplus H_{*-r} (Y,\Z_2) \cong H_*(Y \times \sphere^r, \Z_2)\]
such that $\alpha=p_*$. With such an isomorphism fixed, we have a canonical and geometrically well-behaved lift of classes in $H_*(Y,\Z_2)$ to $H_*(E,\Z_2)$. We can push this discussion a bit further. Given an orientation-preserving morphism $f:E_1 \rightarrow E_2$ of orientable $\sphere^r$-bundles with even Euler classes lifting a map $\overline{f}: Y_1\rightarrow Y_2$, we obtain the commutative diagram: 
\[
\begin{tikzcd}
H_*(Y_1 \times \sphere^{r},\Z_2) \arrow[r, "\phi_1^{-1}"] 
& 
H_{*}(E_1, \Z_2) \arrow[r, "f_*"]\arrow[d,"(p_1)_*"] 
& 
H_*(E_2,\Z_2) \arrow[r, "\phi_2"] \arrow[d, "(p_2)_*"]
& 
H_*(Y_2 \times \sphere^r, \Z_2) \\
H_*(Y_1,\Z_2)\arrow[u, "i_*"]\arrow[r, "\id"] & 
H_*(Y_1,\Z_2)\arrow[r, "\overline{f}_*"] 
& H_*(Y_2,\Z_2) \arrow[r, "\id"] 
& H_*(Y_2,\Z_2) \arrow[u, "i_*"]
\end{tikzcd}
\]
We now apply this reasoning to $E_1= A \cap B$ and $E_2 = B$ or $E_2 =A$.

In particular, $H_{n-3}(A\cap B, \Z_{2}) \cong \Z_{2}^{r(n)}\oplus \Z_{2}^{r(n)}$, where the first factor corresponds to the linking spheres $S_{i} \subset p(A \cap B)$ normal to $C_{i}$ and the second factor is generated by the spherical cylinders $P_{i}= p^{-1}(\gamma_i)\cong \sphere^1 \times \sphere^{n-4}$, where $\gamma_{i} \subset \sphere^{n-1}$ is a loop parallel to $C_i$ in a punctured neighborhood of $C_{i}$. Applying the same observation from above, $H_{n-3}(B,\Z_{2})\cong \Z_{2}^{r(n)}$ has generators corresponding to the linking spheres $S_{i}$. With this in mind, we can understand the cokernel of the map
$\Phi_{n-3}: \Z_{2}^{r(n)}\oplus \Z_{2}^{r(n)}\rightarrow \Z_{2}^{r(n)} \oplus \Z_{2}^{r(n)} $. Let us write
$\Phi_{n-3}=(\Phi_{n-3}^A, \Phi_{n-3}^B)$. By the correspondence with classes $[S_i]$, we conclude $\Phi_{n-3}^B$ surjects and $\Phi_{n-3}^A([S_i])=0$. 
By Corollary \ref{cor:spherical_inclusion}, the spherical cylinders $P_{i}$ are null-homologous both in $A$ and $B$. 
Hence, $\coker(\Phi_{n-3}) \cong \Z_{2}^{r(n)}$. 

Computing $\ker\big(\Phi_{n-4})$ is similar but easier. The Gysin sequence with $\Z_2$-coefficients implies that $H_{n-4}(A\cap B, \Z_{2})\cong \Z_{2}^{r(n)}$, generated by the $\sphere^{n-4}$-fibers of each component and $H_{n-4}(B, \Z_{2})\cong \Z_{2}$ also generated by a $\sphere^{n-4}$-fiber. One directly finds $H_{n-4}(A, \Z_{2})=0$ using $A \simeq X_{sing}$.
It follows that $\ker(\Phi_{n-4})=\Z_{2}^{r(n)-1}$.
We conclude that $H_{n-3}(X,\Z_{2}) \cong \Z_{2}^{2r(n)-1}$, which means $s=0$ by equation \eqref{eq:def_s}. 
Now, using $2r(n)-1=2n-1$ from Lemma \ref{lem:number_circles}, this case is complete. 
\medskip

\textbf{Case 2: $\bm{n\equiv 1}$ mod $4$.} The strategy in this case is similar to the previous case, however, a crucial difference appears now due to the fact that $S_{sing}^{no}\neq \emptyset$ by Corollary \ref{cor:orientability_singular}. 

We now know the homology of $A$ by Lemma \ref{lem:homology_singular}, whereas $H_{*}(B)$ and $H_{*}(A \cap B)$ remain unchanged from the previous case. The Mayer-Vietoris long exact sequence thus becomes

 \begin{center}
\begin{tikzcd}[scale=1.5, row sep=2.5em]
    |[label=below:{\substack{\cong \\ \\ (\Z_{2})^{r(n)}}}]| H_{n-3}(A \cap B) \arrow{r}{(i_*,j_*)} &
    |[label=below:{\substack{\cong \\ \\  \Z_{2}^{2} \oplus \Z^{r(n)-2} \oplus \Z^{r(n)-1}}}]| H_{n-3}(A)\oplus H_{n-3}(B)\arrow{r}{k_*-l_*} &
    H_{n-3}(X) 
    \arrow[dll, rounded corners=8pt, to path={ 
        -- ([xshift=3ex]\tikztostart.east) 
        |- ([yshift=2ex]\tikztotarget.north) 
        -- (\tikztotarget) 
    }] \\
    |[label=below:{\substack{\cong \\ \\ (\Z_{2})^{r(n)}}}]| H_{n-4}(A \cap B) \arrow{r}{(i_*,j_*)} &
    |[label=below:{\substack{\cong \\ \\ 0 \oplus \Z_{2}}}]| H_{n-4}(A) \oplus H_{n-4}(B) \arrow{r}{k_*-l_*} &
    H_{n-4}(X)=0 .
\end{tikzcd}
\end{center}

Using the fact that the alternate sum of ranks in an exact sequence must be zero, we obtain $b_{n-2} = b_{n-3}= 2r(n)-3=2n-1$ with Lemma \ref{lem:number_circles}. In particular, $b_{n-2}$ is the same as in Case 1. Again, the difficult part is the computation of the torsion subgroup $T_{n-3}$.

Now, $H_{n-3}(A)$ is no longer torsion-free. However, we can reason similarly. Let $H \leq \Z_2^2$ be the image of the torsion subgroup $\Z_2^2$ under $k_*-l_*$. By rank considerations, $k_*-l_*$ maps infinite order elements to infinite order elements in $H_{n-3}$. 
Thus, we can extract from Mayer-Vietoris the fact that the connecting homomorphism $\partial_{n-3}$ injects $T_{n-3}/H $ into $H_{n-4}(A \cap B)$. 
Hence, $T_{n-3} \cong \Z_2^s \oplus \Z_4^t$ for some integers $s,t \geq 0$. 
As before, we will prove that $s=t=0$ by showing that $H_{n-3}(X,\Z_{2})$ has rank $2n-1$. Now, let us once again write 
\[
    H_{n-3}(X,\Z_{2})=\coker(\Phi_{n-3})\oplus \ker(\Phi_{n-4}),
\]
where $\Phi_i:H_{i}(A\cap B, \Z_2) \rightarrow H_i(A,\Z_2) \oplus H_i(B,\Z_2)$ is the homomorphism in Mayer-Vietoris. 

Since $A\cap B$ and $B$ are unchanged from the previous case, we use the same generators $([S_i],[P_i])_{i=1}^{r(n)}$ for $H_{n-3}(A\cap B, \Z_2)$ and $([S_i])_{i=1}^{r(n)}$ for $H_{n-3}(B, \Z_2)$. 
The Gysin sequence shows the $\Z_2$-homology of $A$ remains the same as in Case 1, despite the change in homotopy type. 
Thus, once again $\Phi_{n-3}$ obtains the form $\Phi_{n-3}: \Z_{2}^{r(n)}\oplus \Z_{2}^{r(n)} \rightarrow \Z_{2}^{r(n)} \oplus \Z_{2}^{r(n)} $. 
Let us write $\Phi_{n-3}= (\Phi_{n-3}^A, \Phi_{n-3}^B)$ again. 

By the same reasoning as in the previous case, $\Phi_{n-3}^B$ surjects. 
The key difference is the behavior of the generators $[P_i]$. In particular, Corollary \ref{cor:spherical_inclusion} implies that 
$\Phi_{n-3}^A$ has image $\Z_2^{2}$ in this case, corresponding to the two spherical cylinders $P_i$ built from $C_i \subset S_{sing}^{no}$. 
Thus, $\coker(\Phi_{n-3})\cong \Z_2^{r(n)-2}$. 

The analysis of $\ker\big(\Phi_{n-4}: H_{n-4}(A\cap B, \Z_2) \rightarrow H_{n-4}(A, \Z_{2}) \oplus H_{n-4}(B, \Z_{2})\big)$ remains unchanged from the previous case, so $\ker(\Phi_{n-4})=\Z_{2}^{r(n)-1}$.
Combining everything, we conclude that $H_{n-3}(X,\Z_{2})\cong \Z_{2}^{2r(n)-3}=\Z_{2}^{2n-1}$. Thus, $s=t=0$ and $T_{n-3}=0$. 
\end{proof}

\section{Classifying the Fiber}\label{Sec:FiberTopology}

Recall that $X = \hat{\mathcal{B}}(\mathcal{P})$ is highly connected closed odd-dimensional manifold. In this section, we use Wall's work on surgery theory to determine the homeomorphism type of $X$, and the diffeomorphism type up to connected sum with an exotic sphere, for $n \geq 6$. 

\subsection{Wall's Classification}\label{Sec:WallClassification}

We now recall Wall's classification of highly connected odd-dimensional manifolds \cite{Wal67}. 

Wall considers \emph{almost-closed} manifolds, namely smooth compact manifolds $M$ with boundary $\partial M$ a homotopy sphere. He achieves an honest classification of almost-closed $(m-1)$-connected $(2m+1)$-manifolds: up to diffeomorphism. Now, we are presently interested instead in closed manifolds under these hypotheses, for which Wall obtains only a classification up to connected sum with homotopy $n$-spheres. We shall call two closed $n$-manifolds $M_1, M_2$ \emph{almost diffeomorphic} when there exists a homotopy $n$-sphere $\Sigma$ such that $M_1 \# \Sigma$ is diffeomorphic to $M_2$. We recall that by deep work of Kervaire-Milnor \cite{KM63}, the set $\Theta_n$ of $h$-cobordism classes of closed smooth manifolds homotopy equivalent to $\sphere^n$ is a \emph{finite group}, for $n\geq 5$, under the operation of connected sum. 
By Smale's $h$-cobordism theorem, for $n \geq 5$ we can interpret $\Theta_n$ as the group of smooth structures on the topological manifold $\sphere^n$ \cite{Sma62a}.

Now, let $m \geq 3$. We loosely recall Wall's invariants for the classification, endeavoring to use his original notation here. We will state a more precise version of his theorem under the simplified conditions at hand later in this subsection. 

Denote by $G$ the abelian group $G := H_{m}(M,\Z)$, which by Poincar\'e duality determines the whole homology $H_{*}(M, \Z)$ due to the connectivity hypothesis. 
Also, denote by $G^*$ the torsion subgroup of $H_{m}(M,\Z)$. Then \cite[Theorem 7]{Wal67} asserts that almost diffeomorphism classes of closed $(m-1)$-connected $(2m+1)$-manifolds $M$ are in one-to-one correspondence with (isomorphism classes of) the following system of invariants: 
\begin{itemize}
    \item \textbf{Homological Invariants}
        \begin{itemize}[label=$\circ$]
            \item A finitely generated abelian group $G = H_{m}(M, \Z)$. 
            \item A bilinear form $b:G^*\times G^* \rightarrow \Q/\Z$ that is $(-1)^{m+1}$-symmetric. 
            \item If $m$ is odd, a quadratic form $q: G^* \rightarrow \Q/2\Z$ with associated bilinear form $2b$. 
        \end{itemize}
    \item \textbf{Tangential Invariants}

    \begin{itemize}[label=$\circ$]
        \item A homomorphism $\hat{\alpha}: H_{m}(M,\Z) \rightarrow \pi_{m-1}(\mathbf{SO})$, equivalently $\hat{\alpha} \in H^m(M, \pi_{m-1}(\mathbf{SO}))$. 
        \item When $m \not\in \{0,1\} \bmod 8$, a cohomology class $\hat{\beta} \in H^{m+1}(M, \pi_{m}(\mathbf{SO}))$. 
        \item When $m $ is even and $m\notin \{4,8\}$, a cohomology class $\hat{\phi} \in H^{m+1}(M, \Z_2)$.
        \item \emph{Exceptional Tangential Invariants}  -- defined only if $\alpha(G)\neq 0$ mod 2.
    \end{itemize}
    \end{itemize}
\begin{remark}
The exceptional invariants $\omega \in \Z_2$ and $\omega(f) \in \Z_8$, along with the conditions under which they are defined, are rather technical. Both invariants shall not appear in our case of interest. Thus, we believe it best to relieve the reader of any concern regarding them.\footnote{Should the intrepid reader nonetheless desire these invariants, they can be found across \cite[Lemmas 28, 29]{Wal67} and \cite[Lemma 23]{Wal65}, with a careful eye also on the `exceptional case.'} 
\end{remark}
    
Here, two systems of invariants are isomorphic if there is an isomorphism $\theta:G_1\rightarrow G_2$ respecting the remaining invariants. 
Moreover, the system of invariants is \emph{additive} over connected sums. 
This classification is the culmination of a series of papers \cite{Wal63a, Wal63b, Wal63c, Wal65, Wal66, Wal67}. 

In this section, we recall sufficient details regarding the relevant invariants to allow their study in the present case of interest, namely for $M = \hat{\mathcal{B}}(\mathcal{P})$. 

\begin{remark}[No torsion invariants]
Since $\hat{\mathcal{B}}(\mathcal{P})$ has torsion-free homology over $\Z$, the invariants $b$ and $q$ vanish. Consequently, we do not discuss these invariants any further.
\end{remark}

We now describe the tangential invariants. The invariant $\hat{\alpha}: G \rightarrow \pi_{m-1}(\mathbf{SO})$ is the stable version of an invariant $\alpha: \pi_m(M) \rightarrow \pi_{m-1}(\SO(m))$ that we now define. 
\begin{definition}[The invariant $\alpha$]
Let $M$ be a closed $(m-1)$-connected $(2m+1)$-manifold. 
Each homotopy class $x \in \pi_m(M)$ can be represented by a smooth embedding $\iota_x: \sphere^m \rightarrow M$, unique up to smooth isotopy (see Theorem \ref{thm:HomotopicToEmbedding}). 

Define $\alpha: \pi_m(M) \rightarrow \pi_{m-1}(\SO(m+1))$ by associating to $x$ the clutching function of the normal bundle $[\N{\iota_x}]$. 
\end{definition}

By \cite[Theorem 1]{Wal61}, the map $\alpha$ is a homomorphism. 
Let $\mathbf{S}_*:\pi_{m-1}(\SO(m+1))\rightarrow \pi_{m-1}(\mathbf{SO})$ be the induced map of the stabilization homomorphism $\mathbf{S}: \SO(m+1)\rightarrow \mathbf{SO}$. 
Using that $\pi_{m-1}(\SO(m+1))$ is in the stable range, one easily checks $ \alpha = 0$ exactly when $\mathbf{S}_* \circ \alpha=0$. See Proposition \ref{Prop:StablyTrivaltoTrivial} for the details. 
Now, in the present case of $M=\hat{\mathcal{B}}(\mathcal{P})$, we have $\mathbf{S}_* \circ \alpha =0$ due to Lemma \ref{Lem:StablyTrivial}. This vanishing simplifies the discussion enormously: many of Wall's invariants also vanish as a consequence. 
In particular, the invariants $\omega$ and $\omega(f)$ are defined only if $\alpha \neq 0$ (see \cite[page 277]{Wal67}, namely the remark on the `exceptional case'). Thus, these invariants are not defined for $\hat{\mathcal{B}}(\mathcal{P})$, justifying our earlier remark.

\subsubsection{Defining $S\beta$} 

What remains is for us to introduce the tangential invariants $( \hat{\beta},\hat{\phi})$, and their purpose, along with the relation between $\alpha$ and $\hat{\alpha}$. We explain the former point only in the case that $\alpha=0$. Furthermore, we show the pair $\hat{\beta},\hat{\phi}$ are equivalent data to another auxiliary invariant $S\beta$. Along the way, we describe the relation between $\alpha$ and $\hat{\alpha}$.

We would like to introduce the invariant $\beta$ from which $\hat{\beta}$ and $S\beta$ are built. 
To do so, we first recall a transversality result of Haefliger. 

\begin{theorem}[Homotopic to Embedding {\cite{Hae62}}]\label{thm:HomotopicToEmbedding}
Let $M^m$ be a smooth $m$-manifold 
and $f: \sphere^{s}\rightarrow M^{m}$ be continuous. Then $f$ is homotopic to a smooth embedding if $2m\geq 3s+3$ and $M$ is $(2s-m+1)$-connected. Moreover, if $2m>3s+3$ and $M$ is $(2s-m+2)$-connected, then any two such embeddings are smoothly isotopic.
\end{theorem}

The invariant $\beta$ is defined just as $\alpha$, but in a dimension higher. 

\begin{definition}[The invariant $\beta$]\label{Defn:Beta}
Let $M^{2m+1}$ be an $(m-1)$-connected smooth manifold. 
Define a function $\beta: \pi_{m+1}(M) \rightarrow \pi_m(\SO(m))$ as follows. Given $[f] \in \pi_{m+1}(M)$, select a smoothly embedded representative $f^{s}$ with normal bundle $\N f^s\rightarrow \sphere^{m+1}$, whose isomorphism type is independent of $f^s$ by Theorem \ref{thm:HomotopicToEmbedding}. Taking the clutching function then gives a well-defined map $\beta([f]) =[\N f^s]$. 
\end{definition}

We can now define $S\beta$ by altering $\beta$ as follows. In particular, the domain of $S\beta$ is different. 

\begin{definition}[The invariant $S\beta$]\label{Defn:SBeta}
Let $S: \SO(m) \rightarrow \SO(m+1)$ denote the reducible inclusion. The invariant $S\beta$ is the following homomorphism 
\[ S\beta: H_{m+1}(M) \rightarrow S_*(\pi_{m}(\SO(m)))\; \; \]
\[ S \beta := S_* \circ \beta \circ h_{m+1}^{-1},\]
where $h_{m+1}: \pi_{m+1}(M) \rightarrow H_{m+1}(M)$ is the Hurewicz map, which is surjective.
\end{definition}

We shall now explain why $S\beta$ is well-defined. We first state a proposition (obtained from \cite{Wal67}), which we shall use here and later. 
\begin{proposition}\label{Prop:HurewiczUpOne}
Let $M^{2m+1}$ be an $(m-1)$-connected manifold for $m\geq 2$. Then 
\[ \pi_{m+1}(M) \cong H_{m+1}(M,\Z) \oplus H_{m}(M,\Z)\otimes \Z_2.\]
\end{proposition}

\begin{proof}
Consider the short exact sequence:
\[ 0 \longrightarrow H_{m}(M, \Z) \xrightarrow{\xi\circ h_{m}^{-1}} \pi_{m+1}(M) \xrightarrow{h_{m+1}} H_{m+1}(M,\Z) \longrightarrow 0,  \]
where $\xi \in \pi_{m+1}(\sphere^{m})\cong \Z_2$ is a generator and $h_i:\pi_i(M)\rightarrow H_i(M,\Z)$ is the Hurewicz homomorphism. 
Note that $h_m$ is an isomorphism and $h_{m+1}$ is surjective by Hurewicz's theorem. Exactness follows since any suspended element $f \circ \xi \in \pi_{m+1}(M)$, for $f \in \pi_{m}(M)$, is in the kernel of $h_m$. 
By Poincar\'e duality and the universal coefficient theorem, $H_{m+1}(M,\Z)$ is torsion-free and thus the above sequence splits. 
\end{proof}

Now, we verify the definition of $S\beta$. 

\begin{proposition}
The map $S\beta$ is well-defined and is a homomorphism. 
\end{proposition}

While this result is entirely due to Wall, it requires chasing down some definitions across multiple papers, and we believe it is useful to include for readability. 
\begin{proof}

We briefly discuss the homotopy group $\pi_{m+1}(M)$. 
Since $M$ is $(m-1)$-connected, the Hurewicz homomorphism $\pi_{m+1}(M) \rightarrow H_{m+1}(M,\Z)$ is surjective. By Proposition \ref{Prop:HurewiczUpOne},
\begin{align}\label{HurewiczUpOne}
 \pi_{m+1}(M) \cong \Z^{b_{m}} \oplus \Z_2^{b_m}. 
\end{align}
That is, each $(m+1)$-cell contributes a copy of $\Z$, as expected, but we also obtain a copy of $\Z_2$ from each $m$-cell due to the stable homotopy group $\Z_2 \cong \pi_{m+1}(\sphere^m)$. 

We recall one more auxiliary object $\mu$ and its relation to $\beta$ from \cite{Wal63a}. The invariant $\mu$ is of the form $\mu: \pi_{m+1}(M) \times \pi_{m+1}(M) \rightarrow \pi_{m+1}(\sphere^m)$. We shall not even need the definition of this map, but only its type and the fact that it is related to $\beta$ through the following identity by \cite[Theorem 1]{Wal63a}
\begin{align}\label{Wall:AlmostHom}
     \beta(x+y)= \beta(x) + \beta(y)  + \partial \mu(x,y), 
\end{align}
where $\partial: \pi_{m+1}(\sphere^{m}) \rightarrow \pi_m(\SO(m))$ is the connecting homomorphism in the long exact sequence of homotopy groups of the fiber bundle $\SO(m) \rightarrow \SO(m+1) \stackrel{p}{\rightarrow} \sphere^m$: 
\[
    \cdots \longrightarrow \pi_{m+1}(\SO(m))\stackrel{S_{*}}{\longrightarrow} \pi_{m+1}(\SO(m+1)) \stackrel{p_*}{\longrightarrow} \pi_{m+1}(\sphere^{m}) \stackrel{\partial}{\longrightarrow} \pi_{m}(\SO(m)) \longrightarrow \cdots 
\]

In particular, applying $S_{*}$ to \eqref{Wall:AlmostHom}, we obtain that $S_* \circ \beta: \pi_{m+1}(M) \rightarrow S_*(\pi_m(\SO(m)))$ is a homomorphism, using that $S_* \circ \partial =0$ by exactness. 

Let us fix a generator $\xi \in \pi_{m+1}(\sphere^m)$. Then \eqref{HurewiczUpOne} says that the torsion subgroup $T$ of $\pi_{m+1}(M)$ is exactly $\pi_{m}(M)\circ \xi$, meaning every torsion element factors through $\xi$.

Now, if we take $x \in H_{m+1}(M)$, then its preimage under the Hurewicz homomorphism obtains the form $h_{m+1}^{-1}(x) = f+ T$ for some element $f \in \pi_{m+1}(M)$. Thus, any element $\hat{f} \in h_{m+1}^{-1}(x)$ in the preimage obtains the form $\hat{f}=f+g \circ \xi$ for some $g \in \pi_m(M)$. To see that $S\beta$ is well-defined, we must check that $S_* \circ \beta \circ h_{m+1}^{-1}$ kills the torsion subgroup $T$. 
To this end, we observe the following:
\[ \beta(g\circ \xi) = F(\alpha(g), \xi) = 0,\]
where we use that the map $F$ defined in \cite[Lemma 2]{Wal63a} is linear in the first factor since $\xi$ comes from suspension \cite[Lemma 5]{Wal63a}. 
Thus, $T \subset \ker(S\beta)$ and hence
\[ S_* \circ \beta \circ h_{m+1}^{-1}(x) = (S_{*}\circ \beta)(f) \]
is well-defined.
Moreover, $S_* \circ \beta \circ h_{m+1}^{-1}$ is a homomorphism since $S_*\circ \beta$ is a homomorphism. 
\end{proof}

Finally, we can define $\hat{\beta}$ as the stabilization of $S\beta$. 
\begin{definition}[The stable invariants $\hat{\alpha}, \hat{\beta}$]\label{Defn:StableInvariants}
Let $\mathbf{S}_*: \pi_m(\SO(m+1))\rightarrow \pi_m(\mathbf{SO})$ denote the induced map of the stabilization homomorphism. Then define $\hat{\beta} := \mathbf{S}_* \circ S\beta$. 
Analogously, $\hat{\alpha} := \mathbf{S}_*\circ \alpha \circ h_m^{-1}$, where here $\mathbf{S}_*: \pi_{m-1}(\SO(m+1))\rightarrow \pi_{m-1}(\mathbf{SO})$.  
\end{definition}
Note that by Hurewicz, $h_m:\pi_m(M)\rightarrow H_m(M,\Z)$ is an isomorphism and thus $\hat{\alpha}$ is immediately well-defined and a homomorphism. 
Also, by the universal coefficient theorem, we can view $\hat{\beta} \in H^{m+1}(M, \pi_m(\mathbf{SO}))$, as the originally advertised type of $\hat{\beta}$. 

\begin{table}[ht]
    \centering
    \begin{tabular}{|c|c|c|c|c|c|c|c|c|}
        \hline 
        $m \bmod 8$ & 0& 1&2&3&4&5&6&7 \\ \hline
        $\pi_{m}(\mathbf{SO})$ & $\Z_2$ &$\Z_2$ & 0& $\Z$ & 0 & 0&0&$\Z$  \\ \hline 
        $\pi_{m}(\SO(m))$ & $(\Z_2)^3$ &$\Z_2\oplus \Z_2$ & $\Z_4$ & $\Z$ & $\Z_2\oplus \Z_2$ & $\Z_2$&$\Z_4$&$\Z$  \\ \hline 
        $\pi_{m}(\SO(m+1))$ & $(\Z_2)^2$ & $\Z \oplus \Z_2$ & $\Z_2$ & $\Z\oplus \Z$ & $\Z_2$ & $\Z$&$\Z_2$&$\Z\oplus \Z$  \\ \hline
    \end{tabular} 
    \medskip
    \caption{Row 1: Stable homotopy groups $\pi_m(\mathbf{SO})$, depending only on $m \bmod8 $ by Bott periodicity. Rows 2 and 3 show unstable homotopy groups $\pi_m(\SO(m))$ and $\pi_{m}(\SO(m+1))$ for $m\geq 8$, from \cite{Ker60}. Note that $\hat{\beta}$ takes values in the first row, $\beta$ in the second, and $S\beta$ in the third. }
    \label{Table:StableHomotopyGroups}
\end{table}

\begin{table}[ht]
    \centering
    \begin{tabular}{|c|c|c|c|c|c|c|c|c|}
        \hline 
        $m$ &3&4&5&6&7 \\ \hline
        $\pi_{m}(\SO(m))$ & $\Z$ & $\Z_2\oplus \Z_2$ & $\Z_2$ & 0& $\Z$  \\ \hline 
        $\pi_{m}(\SO(m+1))$ & $\Z \oplus \Z$& $\Z_2$ & $\Z$& $0$ & $\Z\oplus \Z$ \\ \hline
    \end{tabular} 
    \medskip
    \caption{Low-dimensional homotopy groups \cite{Ker60}.}
    \label{Table:HomotopyGroupsLowDimensions}
\end{table}

Next, we explain that $\hat{\phi}$, the remaining invariant, can simply be ignored. To this end, the following result summarizes the discrepancy between $\beta, \hat{\beta}, S\beta$ across all the cases, and defines $\hat{\phi}$ (cf. \cite[Proposition 4]{Wal65}). In particular, $S\beta$ is equivalent to the pair $(\hat{\beta},\hat{\phi})$. 

\begin{proposition}[Comparison of $\beta$'s] \label{Prop:CompareBetas}
Let $m \geq 4$ be a positive integer. Then 
\begin{enumerate}[label=(\roman*)]
    \item If $m\neq 6$ is even, the map $\mathbf{S}_*: \pi_m(\SO(m+1)) \rightarrow \pi_m(\mathbf{SO})$ has kernel $\Z_2$. Thus, there is a 
    homomorphism $\phi: \pi_m(\SO(m+1)) \rightarrow \Z_2$ such that 
    \[ (\mathbf{S}_*, \phi): \pi_m(\SO(m+1)) \rightarrow  \pi_m(\mathbf{SO}) \oplus \Z_2 \] is an isomorphism. 
    Define $\hat{\phi} = \phi\circ S\beta$. 
    Then $S_*:\pi_m(\SO(m))\rightarrow \pi_m(\SO(m+1))$ surjects and consequently $(\hat{\beta}, \hat{\phi})$ is equivalent to $S\beta$,  
    \item If $m$ is odd or $m=6$,  
    the stabilization map $\mathbf{S}_*:S_*\pi_m(\SO(m+1)) \rightarrow \pi_m(\mathbf{SO})$ is an isomorphism. 
    Hence, $\hat{\beta}$ is equivalent to $S\beta$. 
\end{enumerate}
\end{proposition}

\begin{proof} Let $m \geq 4$. The proof is merely a case-by-case analysis   
using Table \ref{Table:StableHomotopyGroups}.

(i) For $m \equiv \{0,2,4,6\}$ and $m\neq 6$, the first part of the result immediately follows by the fact that the stabilization map $\mathbf{S}_*:\pi_m(\SO(m+1)) \rightarrow \pi_m(\mathbf{SO})$ surjects.
Now, by the table, $\pi_m(\SO(m))$ and $\pi_{m+1}(\SO(m+1))$ are each  finite groups. Then considering the long exact sequence of homotopy groups of the fiber bundle $\SO(m)\rightarrow \SO(m+1)\rightarrow \sphere^m$, we see: 
\[ \pi_m(\SO(m))\xrightarrow{S_*} \pi_m(\SO(m+1))\xrightarrow{p_*} \pi_m(\sphere^m) \cong \Z, \]
which implies $S_*$ surjects by exactness. Then $S_*\pi_m(\SO(m))=\pi_{m}(\SO(m+1))$ and hence the pair $(\hat{\beta},\hat{\phi})$ are equivalent data to the invariant $S\beta$. 

(ii) Now, let $m \geq 7$ be odd. The result similarly follows from the fact that the homomorphisms  $\pi_m(\SO(m))\rightarrow \pi_m(\mathbf{SO})$ and $S_*\pi_m(\SO(m))\rightarrow \pi_m(\mathbf{SO})$ surject.
\end{proof}

We now summarize the situation with the following theorem, clarifying exactly the two invariants needed for Wall's classification in the present case.
\begin{theorem}[Simplified Classification {\cite{Wal67, Wil72}\footnote{Wall's classification excludes the cases $m=3$ and $m=7$. However, Wilkens extended Wall's results in these two cases \cite{Wil72}, and his extensions apply in our setting since $H_{m}(X,\Z)$ is torsion-free.}}]\label{thm:WallSimplified}
Let $m\geq3$ be an integer and $M$ be a closed $(m-1)$-connected $(2m+1)$-manifold such that $\alpha=0$. Assume that $M$ has torsion-free homology. Then $M$ is classified up to almost diffeomorphism by the following:
  \begin{itemize}
    \item The Betti number $b_{m} = \rank(H_{m}(M, \Z))$, 
    \item The homomorphism $S\beta: H_{m+1}(M,\Z) \rightarrow \pi_m(\SO(m+1))$. 
\end{itemize}  
\end{theorem}

\begin{proof}
Let $M_1, M_2$ be two closed manifolds with the same invariants. 
Let us consider $M_{i}'=M_{i} \setminus B_i^{2n-5}$, where $B_i^{2n-5}$ is a smooth open ball contained entirely in the domain of a local smooth chart of $M_{i}$. Then the manifolds with boundary $M_i'$ are almost closed and have the same system of invariants. Applying \cite[Theorem 7]{Wal67}, we conclude $M_1'$ and $M_2'$ are diffeomorphic. 
Gluing back the original disks from $M_1'$ and $M_2'$ to their respective boundary spheres, we obtain a diffeomorphism between $M_1 \#\Sigma$ and $M_2$, for some homotopy sphere $\Sigma \in \Theta_{2m+1}$. 
\end{proof}

Note that we can only conclude that $M_1$ and $M_2$ share the same almost diffeomorphism type because different ways of filling in boundary spheres with disks can produce non-diffeomorphic manifolds.  

\begin{remark}[Invariants of Homotopy Spheres]
Let $\Sigma \in \Theta_{2m+1}$ be a homotopy sphere. Then $b_m=0$, $\T\Sigma$ is stably trivial, and $S\beta=0$. That is, $\Sigma$ corresponds to the trivial system of invariants and hence are invisible to Wall's technology. Since Wall's invariants are additive over connected sums, for any $(m-1)$-connected closed manifold $M^{2m+1}$, both $M$ and $M \# \Sigma$ have the same system of invariants. 
\end{remark}

We now provide an example of Theorem \ref{thm:WallSimplified} that distinguishes two familiar spaces.

\begin{example}[Distinguishing Sphere Bundles]
Let $m \geq 4$ be even, $m \neq 6$. Consider the highly connected closed manifolds $M_1:=V_2(\R^{m+2})\cong \T^1\sphere^{m+1}$ and $M_0:=\sphere^{m+1}\times \sphere^m$. Each manifold $M_i$ is parallelizable \cite{Sut64} (hence $\alpha=0$) and satisfies $H_m(M_i,\Z)\cong \Z$. In fact, these two smooth manifolds are distinguished by Theorem \ref{thm:WallSimplified}, and provide perhaps the simplest such application.

We now briefly explain why $S\beta_{M_0}=0$ and $S\beta_{M_1}\neq 0$, so that $S\beta$ forbids the two to be almost diffeomorphic.
Note that $\Sigma_0 := \sphere^{m+1}\times \{*\}$ generates $H_{m+1}(M_0)$. One easily finds that the normal bundle $\N \Sigma_0\cong \varepsilon_\R^{m}$ is trivial. Thus, $S\beta_{M_0}([\Sigma_0])=0$, meaning $S\beta_{M_0}=0$. Now, a smoothly embedded sphere $\Sigma_1$ generating $H_{m+1}(M_1)$ is given by 
\[ \Sigma_1 = \{(x,J(x)) \in M_1 \mid x \in \sphere^{m+1} \},\]
where $J$ is any orthogonal complex structure on $\R^{m+2}$. In other words, $\Sigma_1$ is the graph of a section of the unit tangent bundle.
Note the canonical identification \[ \N_{x} \Sigma_1 \simeq \{y \in \R^{m+2} \mid y\in \Span \{x, J(x)\}^\bot \}.\]
Thus, 
$\N \Sigma_1 \oplus \varepsilon_1^{\R} \cong \T \sphere^{m+1}$. We conclude 
$S\beta_{M_1}([\Sigma_1])= [\T \sphere^{m+1}]\neq 0$, meaning $S\beta_{M_1}\neq0$.
\end{example}

\subsection{Simplified Invariants for \texorpdfstring{$\hat{\mathcal{B}}(\mathcal{P})$}{B(P)}}

In this section, we record the general information we know about Wall's invariants for $X=\hat{\mathcal{B}}(\mathcal{P})$ in all cases. Here, we denote the ambient dimension by $n$, so that $X \subset V_2(\R^n)$ and  $\dim(X)=2n-5$. Hence, $m=n-3$ here. 

\subsubsection{Vanishing of Stable Tangential Invariants}

We now show that $\hat{\alpha} = \hat{\beta} =0$ for $X = \hat{\mathcal{B}}(\mathcal{P})$. 

The following basic result says  $\hat{\alpha}=0$ implies that $\alpha=0$ in our situation. 
\begin{proposition}\label{Prop:StablyTrivaltoTrivial}
Let $m \geq 2$. 
If $f: \sphere^{m}\rightarrow M^{2m+1}$ is an immersion and $M$ is stably parallelizable, then $\N f$ is trivial. 
\end{proposition}

\begin{proof}
Let $M$ and $f: \sphere^{m}\rightarrow M$ be as in the hypotheses. 
Note that $M$ is orientable since $\T M$ is stably trivial. 
The normal bundle $\N f \rightarrow \sphere^m$ is of rank $m+1$ and hence classified by its clutching functions $[\N f] \in \pi_{m-1}(\SO(m+1))$. Now, $\N f$ is stably trivial since $\T M$ is. 
Hence, for a sufficiently large positive integer $j$, the stabilization $(S_{j})_*: \pi_{m-1}(\SO(m+1)) \rightarrow \pi_{m-1}(\SO(m+j))$ pushes $\N f$ forward to the identity element. 
However, for $j \geq 1$, the map $(S_{j})_*$ is injective, as seen by inductively considering the long-exact sequence of homotopy groups of the fiber bundles $\SO(k-1) \rightarrow \SO(k)\rightarrow \sphere^{k-1}$. Thus, $\N f$ itself is trivial. 
\end{proof}

\begin{corollary}
The invariant $\alpha \in \Hom\big(\pi_m(X), \pi_{m-1}(\SO(m+1)) \big)$ is trivial for $X = \hat{\mathcal{B}}(\mathcal{P})$. 
\end{corollary}

By similar reasoning, we easily obtain the following result, whose proof we omit.  

\begin{proposition}\label{Prop:StableNormal}
Suppose that $M^{2m+1}$ is a closed manifold with stably trivial tangent bundle. If $f:\sphere^j \rightarrow M$ is a smooth immersion, the normal bundle $\N f$ is stably trivial.
\end{proposition}

We learn that the stable tangential invariants $\hat{\alpha}, \hat{\beta}$ vanish in our case. 

\begin{corollary}[Trivial Stable Tangential Invariants]\label{Cor:StableTangentialInvariants}
The invariants $\hat{\alpha} \in H^{m}(X, \pi_{m-1}(\mathbf{SO}))$ and $\hat{\beta} \in H^{m+1}(X, \pi_{m}(\mathbf{SO}))$ are both trivial for $X = \hat{\mathcal{B}}(\mathcal{P})$. 
\end{corollary}

\begin{proof}
Recall that $\T X$ is stably trivial by Lemma \ref{Lem:StablyTrivial}. 
The result then follows from Propositions \ref{Prop:StablyTrivaltoTrivial} and \ref{Prop:StableNormal}. 
\end{proof}

\subsection{Classification, \texorpdfstring{$n$}{n} even}\label{Sec:TopologyEven}

In this case, $m= n-3$ is odd and the techniques of Wall allow us to classify $X$ easily. 

\begin{theorem}[Fiber Classification, Even Case]\label{thm:GlobalTopologyEven}
Let $n\geq 6$ be even and $\mathcal{P} = \T_Q\Ha^2_{\pr}$. Then the base of pencil $X=\hat{\mathcal{B}}(\mathcal{P}) \subset V_2(\R^n)$ is almost diffeomorphic to $\#_{i=1}^{n-1} (\sphere^{n-3}\times \sphere^{n-2})$.
\end{theorem}

\begin{proof}
Note that by Proposition \ref{Prop:CompareBetas}(ii) and Corollary \ref{Cor:StableTangentialInvariants}, the invariant $S\beta$ vanishes for $X$. The result then follows from Theorem \ref{thm:WallSimplified} and Theorem \ref{thm:hom_even}. 
\end{proof}

By topological invariance, we conclude our first main result. 
\begin{corollary}
Let $n \geq 6$ be even, $\rho:\pi_1S \rightarrow \PSL(n,\R)$ be a Hitchin representation, and $\hat{\Omega}_{\rho} \subset V_2(\R^n)$ be the universal cover of the domain $\Omega_{\rho} \subset \mathcal{F}_{1,n-1}$. There is a smooth fiber bundle projection $\hat{\Omega}_{\rho} \rightarrow \Ha^2$ with fiber almost diffeomorphic to 
$\#_{i=1}^{n-1} (\sphere^{n-3}\times \sphere^{n-2})$. 
\end{corollary}

\subsection{Classification, \texorpdfstring{$n$}{n} odd}\label{Sec:TopologyOdd}

We now classify $X = \hat{\mathcal{B}}(\mathcal{P})$ up to connected sum with homotopy spheres in the case that $n\geq 7$ is odd. Unlike in the even case, the hardest work is yet to come here: we must compute the invariant $S\beta$. We show that $S\beta=0$ and thereby deduce that $X$ is once more almost diffeomorphic to a connected sum of products of spheres.

\subsubsection{Generators for $H_{n-2}(X)$.}\label{Subsec:MiddleHomologyGenerators}

We now describe a set of redundant generators for $H_{n-2}(X)$. We will see dichotomy for these homology classes emerge: those induced by inclusion of $X_{gen}$ and those that are not. For the latter classes, the computation of $S\beta$ shall be considerably more difficult.

We again use the open covering $X=A \cup B$, where $A$ a regular neighborhood of $X_{sing}$ and $B=X_{gen}$. 
Recall that $A$ is homotopy equivalent to $X_{sing} \cong (S_{sing}^{o} \times \sphere^{n-3}) \sqcup (S_{sing}^{no} \tilde{\times} ~ \sphere^{n-3})$, and $S_{sing}=S_{sing}^o\sqcup S_{sing}^{no}$ partitions the singular circles according to the orientability or non-orientability of the restriction of $\mathcal{R}^\bot$ to the given circle. 
The space $A \cap B$ is homotopy equivalent to an orientable $\sphere^{n-4}$-bundle over the base $V=\bigsqcup_{i=1}^{r(n)} (\sphere^1\times \sphere^{n-3})$ with Euler class $(\pm2,\pm2,\cdots, \pm2) \in H^{n-3}(V,\Z)$ by Theorem \ref{thm:EulerClassXgen}.

Now, Mayer-Vietoris provides the following short-exact sequence: 
\begin{align}\label{MV-441}
    0=H_{n-2}(A\cap B) \longrightarrow H_{n-2}(A) \oplus H_{n-2}(B) \xrightarrow{\Psi_{n-2}} H_{n-2}(X) \xrightarrow{\partial_{n-2}}  \ker(\Phi_{n-3}) \longrightarrow 0 
\end{align}
Our first task will be to understand the image of $\Psi_{n-2}$. 
To this end, we introduce a definition that we shall frequently reference. 
\begin{definition}[Spherical cylinders]
Let $P_i \subset A \cap B$ be the embedded spherical cylinder $P_i\cong \sphere^1\times \sphere^{n-4}$ given by taking a circle $\gamma_i \subset p(A \cap B)$ homotopic to $C_i$ and defining $P_i := p^{-1}(\gamma_i)$. 
\end{definition}

Now, let us partition the index set $I = \{1,2,\dots, r(n)\}$ of the singular circles as $I=I^{o}\sqcup I^{no}$, where $i \in I^o$ exactly when $C_i \subset S_{sing}^o$. 

The following result describes geometrically the cokernel of $\Psi_{n-2}$. 

\begin{lemma}[Classifying $\mathrm{coker}(\Psi_{n-2})$]\label{Lem:ImagePsi}
The image of $\partial_{n-2}$ is freely generated by the classes $\{[P_i] \}_{i \in I^o}$. 
\end{lemma}

Thus, a nonzero element $x \in H_{n-2}(X)$ satisfies $x \notin \mathrm{image}(\Psi_{n-2})$ if and only if $\partial_{n-2}(x)$ lies in the span of $\{[P_i]\}_{i \in I^o}$.

\begin{proof}
Let us analyze in more detail the group $\ker(\Phi_{n-3})$, which take two different forms according to the congruence class of $n$ mod $4$.
Let $(A\cap B)_i$ denote the connected component of $A\cap B$ intersecting a neighborhood of $C_i$. 
Using Mayer-Vietoris again, one finds that the spherical cylinder $[P_i]$ generates $H_{n-3}( (A \cap B)_i,\Z)\cong \Z_2$. As $[P_i]$ has order two in $H_{n-3}(A\cap B, \Z)$, orientations will be of no concern. \medskip

\textbf{Case 1: $\bm{n \equiv 3 \bmod 4}$}. Here, $S_{sing}^{no} = \emptyset$, hence $\Phi_{n-3}: H_{n-3}(A\cap B) \rightarrow H_{n-3}(A)$ is the trivial map, because $H_{n-3}(A\cap B)$ is pure torsion whereas $H_{n-3}(A) \oplus H_{n-3}(B)$ is a free $\Z$-module. Therefore, $\partial_{n-2}$ surjects onto $H_{n-3}(A\cap B)\cong \Z_{2}^{r(n)}$, which is generated by all spherical cylinders $\{ \, [P_i]\, \}_{i=1}^{r(n)}$. The claim follows since $I = I^o$ when $n \equiv 3 \bmod 4$. \medskip

\textbf{Case 2: $\bm{n \equiv 1 \bmod 4}$}. In this case, $S_{sing}^{no}\neq \emptyset$, and $\Phi_{n-3}$ fits into the exact sequence:

\begin{center}
    \begin{tikzcd}[scale=1.5, row sep=2.5em]
    |[label=below:{\substack{\cong \\ \\ \Z_{2}^{r(n)}}}]| H_{n-3}(A \cap B) \arrow{r}{\Phi_{n-3}} &
    |[label=below:{\substack{\cong \\ \\  \Z_{2}^{2} \oplus \Z^{r(n)-2} \oplus \Z^{r(n)-1}}}]| H_{n-3}(A)\oplus H_{n-3}(B)\arrow{r}{\Psi_{n-3}} &
   |[label=below:{\substack{\cong \\ \\ \Z^{2r(n)-3}}}]| H_{n-3}(X) \arrow{r}& 0
   \end{tikzcd}
\end{center} 
Here, we apply Lemma \ref{Lem:CohomologyGenericLocus}, Lemma \ref{lem:homology_singular}, and Theorem \ref{thm:hom_odd} to evaluate the homology groups. 
 In particular, since $H_{n-3}(X)$ is free, $\ker(\Phi_{n-3}) \cong \Z_{2}^{r(n)-2}$. 
A spherical cylinder $P_i$ satisfies $[P_i] \notin \ker(\Phi_{n-3})$ exactly when $i \in I^{no}$ by Corollary \ref{cor:spherical_inclusion}. 
 Therefore, the image of $\del_{n-2}$ is the subgroup of $H_{n-3}(A\cap B)$ generated by all spherical cylinders $P_{i}$ with $i \in I^{o}$.
\end{proof}

We will use Lemma \ref{Lem:ImagePsi}  to describe two explicit families of homology classes that generate $H_{n-2}(X)$. To achieve such generators, we now describe piecewise smoothly embedded $(n-2)$-spheres $Q_i^{\pm}$
satisfying $\partial_{n-2}([Q_i^{\pm}])=[P_i]$. 

To define the spheres $Q_i^{\pm}$, we shall need the following preliminary definition. 
\begin{definition}[Hemisphere Bundles]\label{Defn:HemisphereBundles}
Let $P_i' \cong \sphere^1\times \sphere^{n-4}$ be the trivial fiber sub-bundle of $p^{-1}(C_i)$ from Lemma \ref{Lem:Replace}. Denote 
$Q_i^+(A), Q_i^-(A) \cong \sphere^1 \times \D^{n-3}$ as the hemisphere bundles over $C_i$ with common boundary $P_i'$, such that $p^{-1}(C_i) =Q_i^+(A) \cup Q_i^-(A)$. 
\end{definition}

Now, the spheres of interest are built as follows. 

\begin{lemma}[Spherical Generators of $\coker(\Psi_{n-2})$]\label{Lem:DefineQi}
Let $D^2_i \subset \sphere^{n-1}$ be a smoothly embedded disk with $\partial D^2_i=C_i$ and interior $B^2_i \subset S_{gen}$. 
There is a topological $(n-2)$-sphere $Q_i^{\pm} \subset X$ satisfying the following:
\begin{enumerate}[label=(\roman*)]
    \item $Q_i^{\pm} \subset p^{-1}(D^2_i)$, so that $p$ defines a continuous map $p: Q_i^{\pm} \rightarrow D^2_i$. 
    \item $p^{-1}(C_i) = Q_i^{\pm}(A)$. 
    \item $\partial_{n-2}([Q_i^{\pm}])= [P_i]$. 
\end{enumerate}
\end{lemma}

While the lemma is not a definition per se, it does uniquely constrain the construction. 

\begin{proof}
  We now describe $Q_i^+$, as $Q_i^-$ is analogous. 
  The idea of the construction is to build $Q_i$ in two pieces, one from each of $X_{gen}$ and $X_{sing}$. Let $\mathbb{D}^k$ denote a closed $k$-disk and observe the following: 
\begin{align}\label{sphereSplitting}
    \sphere^{n-2}=\partial\mathbb{D}^{n-1} = \partial(\mathbb{D}^{2} \times \mathbb{D}^{n-3})= (\sphere^1\times \mathbb{D}^{n-3}) \cup (\mathbb{D}^2\times \sphere^{n-4}),
\end{align} 
which is just a variant of the familiar decomposition of $\sphere^3$ as a union of two solid tori. 

We shall build $Q_i^+$ with a decomposition $Q_i^+ = Q_i^+(A) \cup \overline{Q}_i(B)$ as in \eqref{sphereSplitting}, where  
\begin{itemize}
    \item $Q_i^+(A) \subset X_{sing}$ and $Q_i^+(A) \cong \sphere^1\times \mathbb{D}^{n-3}$, 
    \item $\overline{Q}_i(B) \subset X_{gen}$ and $\overline{Q}_i(B) \cong \mathbb{D}^2 \times \sphere^{n-4}$. 
\end{itemize}
Here, $Q_i^+(A)$ is from Definition \ref{Defn:HemisphereBundles}, we set $Q_i(B) := p^{-1}(B^2_i)$ and then define $\overline{Q}_i(B)$ to be the closure of $Q_i(B)$ in $Q_i := Q_i(A) \sqcup Q_i(B)$. By definition, we automatically have $Q_i \subset p^{-1}(D^2_i)$ and $p^{-1}(C_i)=Q_i^+(A)$. 
What remains is a justification that $Q_i$ is indeed a topological $(n-2)$-sphere, and a verification of (iii). 

Following the proof of Lemma \ref{Lem:Replace}, one sees that $\partial \overline{Q}_i(B)=P_i'$. Hence, by the splitting \eqref{sphereSplitting}, we conclude that $Q_i^+(A)$ and $Q_i(B)$ glue together to produce a topological $(n-2)$-sphere. 

We now check (iii). We can decompose $Q_i^+$ alternatively to make this evident. Up to adjusting our initial disk $D^2_i$, we can suppose that the circle $\gamma_i \subset \sphere^{n-1}$ such that $P_i = p^{-1}(\gamma_i)$ satisfies $\gamma_i \subset B^2_i$. 
Let us split $D^2_i = \mathbb{A}^2_i \cup \overline{D}^2_i$, where the annulus $\mathbb{A}^2_i$ has boundary $C_i \sqcup \gamma_i$, and $\overline{D}^2_i $ is a sub-disk with boundary $\gamma_i$. 
We can then instead write $Q_i^+ = Q_i'\cup Q_i''$, where 
$Q_i' = p^{-1}(\mathbb{A}^2_i)$ and $Q_i'' = p^{-1}(\overline{D}^2_i)$. 
Observe the following: $Q_i' \subset A$, $Q_i'' \subset B$ and these submanifolds have common boundary $P_i$, with opposite orientation.
By definition of the boundary map $\partial$ in Mayer-Vietoris \cite[page 150]{Hat01}, we conclude $\partial_{n-2}([Q_i^+]) = [P_i]$. 
\end{proof}

We now describe complementary generators of $H_{n-2}(X)$, which come from the inclusion $X_{gen} \hookrightarrow X$. Recall the fibration $\sphere^{n-4} \rightarrow X_{gen}\rightarrow S_{gen}$. In $\S$\ref{Sec:SgenHomologyGenerators}, we described separating $(n-2)$-spheres $\Sigma_i$ that generate the group $H_{n-2}(S_{gen})$. The Gysin sequence implies that the fibration $p: X_{gen} \rightarrow S_{gen}$ has the property that $p_*:H_{n-2}(X_{gen}) \rightarrow H_{n-2}(S_{gen})$ is an isomorphism. 
Thus, the $(n-2)$-spheres $\Sigma_i$ induce homology classes $p_*^{-1}([\Sigma_i]) \in H_{n-2}(X_{gen})$. Although this definition is rather indirect, we shall later prove in $\S$\ref{Subsec:SeparatingSpheres} that $p:X_{gen}\rightarrow S_{gen}$ admits a section $s\in \Omega^0(\Sigma_i, X_{gen}|_{\Sigma_i})$, and hence the class
$p_*^{-1}([\Sigma_i])$ is just $s_*([\Sigma_i])$.

We now verify that the homology classes defined above generate $H_{n-2}(X)$.

\begin{lemma}[Dichotomy of Generators]\label{Lem:Dichotomy}
$H_{n-2}(X)$ is generated by the homology classes $\{[Q_{i}^{\pm}]\}_{i \in I^{o}}$ along with $\{\Psi_{n-2}(p_*^{-1}([\Sigma_i]))\}_{i \in I}$. 
\end{lemma}

\begin{proof} From the short exact sequence
\[ 
0 \longrightarrow H_{n-2}(A) \oplus H_{n-2}(B) \xrightarrow{\Psi_{n-2}} H_{n-2}(X) \xrightarrow{\partial_{n-2}}  \ker(\Phi_{n-3})\longrightarrow 0, 
\]
we know that $H_{n-2}(X)$ is generated by the image of $\Psi_{n-2}$ together with some homology classes $\{x_{i}\}$ such that $\partial_{n-2}(x_{i})$ generate $\ker(\Phi_{n-3})$. By Lemma \ref{Lem:DefineQi} the family $\{[Q_{i}^{\pm}]\}_{i}$ has this property.
Also by construction, the classes $p_*^{-1}([\Sigma_i])$, for $i \in I$, generate $H_{n-2}(B)$.  
Thus, the lemma follows if $\Psi_{n-2}(H_{n-2}(A))$ is spanned by $\{[Q_{i}^{\pm}]\}_{i}$. The remainder of the proof verifies this is indeed the case. 

Since $A$ is homotopy equivalent to $(S_{sing}^{o} \times \sphere^{n-3}) \sqcup (S_{sing}^{no} \tilde{\times} ~ \sphere^{n-3})$, by Lemma \ref{lem:homology_singular}, the homology group $H_{n-2}(A)$ is generated by $[p^{-1}(C_i)]$ with $C_{i} \subset S_{sing}^{o}$. Let us denote $Y_i:=p^{-1}(C_i) \cong \sphere^1\times \sphere^{n-3}$. 
We still denote by $[Y_i]$ the corresponding classes in $H_{n-2}(X)$, which does not cause ambiguity because $\Psi_{n-2}$ is injective. We claim that $[Q_{i}^{+}]+[Q_{i}^{-}]=[Y_i]$, which will complete the proof. 
To see this, we represent the class $[Q_{i}^{+}]+[Q_{i}^{-}]$ by the cycle 
\[Q_{i}^{+}(A) \cup Q_{i}^+(B) \cup Q_{i}^{-}(A) \cup Q_{i}^-(B),\]
where we distinguish the orientation on $Q_i(B) \cong \mathbb{B}^2 \times \sphere^{n-4}$. 
It is clear that $Y_i = Q_i^+(A) \cup Q_i^-(A)$. The other two terms $Q_i^{+}(B)$ and $Q_i^-(B)$ cancel out as singular chains, yielding the desired equality $[Y_i]=[Q_{i}^{+}]+[Q_{i}^{-}]$.
\end{proof}

\subsubsection{Computing $S\beta$ on $Q_{i}$}

We now prove the invariant $S\beta$ vanishes 
on the homology classes $[Q_i^{\pm}]$. The following proof is quite involved, so we provide a summary here.

\begin{itemize}
    \item \textbf{Step 0}: We orthogonally split the tangent bundle $\T X = V\oplus H$ into vertical and horizontal parts using the projection map $p: X \rightarrow \sphere^{n-1}$ from the Structure Lemma. We note how this splitting interacts with $\T Q_i(A)$ and $\T Q_i(B)$.
    \item \textbf{Step 1}: We trivialize the normal bundle of $Q_i$ restricted to $Q_i(B)$ geometrically. 
    \item \textbf{Step 2}: We trivialize the normal bundle of $Q_i$ restricted to $Q_i(A)$ geometrically. 
    \item \textbf{Step 3}: We carefully smooth the corners of the $C^0$-embedded submanifold $Q_i$ to become a $C^{\infty}$-embedded submanifold $Q_i^{s}$, where $s$ stands for ``smooth''. In particular, the gluing applies a flow to $Q_i(A)$ and $Q_i(B)$, which near their common boundary pushes out, and then appends a new cylinder $\mathcal{C}\cong (\sphere^1\times \sphere^{n-4})\times (0,1)$ over the boundary in the middle. See Figure \ref{Fig:Smoothing}. This new cylindrical piece plays a key role in Step 4.
    \item \textbf{Step 4}: We produce $n-4$ linearly independent sections of the once-stabilized normal bundle $E=\N Q_i^{s}\oplus \varepsilon^1_{\R}$ by taking $(n-4)$-frames $\mathbf{V}$ and $\mathbf{W}$ for the normal bundle restricted to $Q_i(A)$ and $Q_i(B)$, respectively, and then using the new cylinder $\mathcal{C}$ to glue these local sections together to global sections.
    \item \textbf{Step 5}: We finally prove that $E$ is trivial. Since we know $E$ is stably trivial, to prove it is actually trivial, we need only show $E$ is not isomorphic to $\T \sphere^{n-2}$. By Adams' celebrated work \cite{Ada60} on vector fields on spheres, to forbid $E$ from being $\T\sphere^{n-2}$, the existence of an $(n-4)$-frame on $E$ is sufficient.
\end{itemize}
We now prove the theorem. 

\begin{theorem}\label{thm:nightmare_spheres} 
Let $n \geq 7$ be odd. Then $S\beta([Q_i])=0$. That is, the homology class $[Q_i]$ is represented by a smoothly embedded sphere $Q_i^{s}$ in $X$ such that the once stabilized normal bundle $\N Q_i^{s}\oplus \varepsilon^1_{\R}$ is trivial. 
\end{theorem}

\begin{proof}
\textbf{Step 0: Setup}. Note that there is a natural Riemannian metric on $X$ inherited from $\R^{2n}$. Now, let us introduce some notation. Given $(x,y) \in X$, write $V_{(x,y)} \coloneqq \ker(dp_{(x,y)})$, where $p: X \rightarrow \sphere^{n-1}$ is the almost-fibration from Lemma \ref{Lem:StructureLemma}. We note the following: 
\begin{itemize}
    \item $\dim V_{(x,y)} = n-4$ if and only if $x \in S_{gen}$, 
    \item $\dim V_{(x,y)} = n-3$ if and only if $x \in S_{sing}$. 
\end{itemize}
Indeed, $\dim V_{(x,y)}=\dim\mathcal{R}^\bot_{|x}-1$.
Hence, $dp_{(x,y)}: \T_{(x,y)}X \rightarrow \T_{x}\sphere^{n-1}$ surjects exactly when $(x,y) \in X_{gen}$ and otherwise $\rank( dp_{(x,y)}) = n-2$ when $(x,y) \in X_{sing}$. 

We now define a complementary horizontal subspace via the Riemannian structure on $X$:
\[ H_{(x,y)} \coloneq [ V_{(x,y)}^\bot \subset \T_{(x,y)} X]. \]  
Thus, $\dim H_{(x,y)} =n-1$ exactly when $x \in S_{gen}$ and 
$\dim H_{(x,y)} = n-2$ exactly when $x \in S_{sing}$. \medskip 

\textbf{Step 1: Consider $Q_i(B)$}. Take $(x,y) \in Q_i(B)$. 
Recall that $Q_i(B) = X|_{B^2_i} \cong \mathbb{B}^2\times \sphere^{n-4}$, where $B^2_i\subset S_{gen}$ is the interior of a smooth $2$-disk $D^2_i=B^2_i \sqcup C_i $. 
Using the splitting $\T X = HX \oplus VX$, observe that for any $(x,y) \in Q_i(B)$, its tangent space splits as 
\[ \T_{(x,y)}Q_i(B) = \overline{H}_{(x,y)} \oplus V_{(x,y)},\]
for some two-dimensional subspace $\overline{H}_{(x,y)} \subset H_{(x,y)}$. Hence, 
\[ \big( \N _{(x,y)}Q_i \subset \T_{(x,y)} X \big) = \big(\, \overline{H}_{(x,y)}^\bot \subset H_{(x,y)} \big) .\] 
On the other hand, as noted in \emph{Step 0}, there is a canonical identification between $x^\bot$ and $H_{(x,y)}$ via $dp_{(x,y)}$ in the case that $x \in S_{gen}$. Now, note that the projection $p:V_2(\R^n) \rightarrow \sphere^{n-1}$ is a Riemannian submersion. The restriction $p|_{Q_i(B)}: Q_i(B) \rightarrow B^2_i$ is also a Riemannian submersion. This means $dp_{(x,y)}$ isometrically maps $\overline{H}_{(x,y)}$ to $\T_xB^2_i$ and hence 
\[ dp_{(x,y)}(\N_{(x,y)}Q_i(B)) = \big(\N_xB^2_i \subset \T_x\sphere^{n-1} \big) .\]
The above pointwise identifications describe a canonical bundle isomorphism: 
\[ p^*( \N B^2_i\subset \T\sphere^{n-1})\cong \N Q_i(B).\] In particular, this isomorphism provides a preferred ``geometric'' trivialization of $\N Q_i(B)$, since $\N B^2_i \rightarrow B^2_i$ is a trivial vector bundle.  

We denote $\overline{Q}_i(B) \cong \mathbb{D}^2\times \sphere^{n-4}$ as the closure of $Q_i(B)$ in $Q_i$. 
Now, $\N {\overline{Q}_{i}(B)}$ is also a trivial bundle. Indeed, this is an immediate consequence of the fact that the inclusion map 
$\iota: \mathbb{B}^2 \times \sphere^{n-4} \hookrightarrow \mathbb{D}^2 \times \sphere^{n-4}$ is a homotopy equivalence. 
In particular, we can fix a global frame $\mathbf{W}=(W_j)_{j=1}^{n-3}$ of $\N \overline{Q}_{i}(B)$ such that each section $W_j = W_j(x,y)$ depends only on $x \in \mathbb{D}^2$ and not on $y \in \mathbb{S}^{n-4}$.  

Recall that $P_i'=\partial Q_i(B)$. Since $P_i'$ has trivial normal bundle in $\overline{Q}_i(B)$,
\begin{align}\label{BoundaryNormal}
    \N P_{i}' \cong \N Q_{i}|_{P_{i}'}\oplus \varepsilon_{\R}^{1}.
\end{align}
Thus, $\N P_i'$ is also trivial. Moreover, via the trivialization of $\N\overline{Q}_{i}(B)$, we can find a global frame that restricts to $P_{i}=\partial \overline{Q}_{i}(B)$ as an $(n-3)$-frame of $\N P_{i}$. 

Now, fix any $(x,y) \in P_i'$ and decompose $\N P_i'$ into horizontal and vertical parts under the splitting $\T X= HX \oplus VX$. For later, note that the line $\T_{(x,y)}\overline{Q}_{i}(B) \cap \N_{(x,y)}P_i'$, i.e., the normal line of $P_i'$ in $\overline{Q}_i(B)$, is \emph{not} vertical. \medskip 

\textbf{Step 2: Consider $Q_i(A)$.} Recall that topologically $Q_i(A) \cong \sphere^1 \times \mathbb{D}^{n-3}$. Now, take a point $(x,y) \in Q_i(A)$. In this case, the vertical subspace $V_{(x,y)}$ has dimension $n-3$ and $\N Q_{i}(A)|_{(x,y)}$ is a codimension one subspace $H_{(x,y)}'$ of $H_{(x,y)}$. Note that for $(x,y) \in P_i' =\partial Q_i(A)$, the line $\T_{(x,y)}Q_{i}(A) \cap \N_{(x,y)}P_i'$ \emph{is} vertical. 
This shows that $Q_i$ as currently defined is not smooth. 

As in the previous step, we claim that $\N {Q_{i}(A)}\rightarrow Q_i(A)$ is a trivial vector bundle. This follows immediately from the fact that $\N Q_i(A)$ is an orientable vector bundle over $\sphere^{1} \times \D^{n-3}$. That is, topologically, triviality is controlled by the restriction of the vector bundle to a copy of $\sphere^1 \times \{*\}$. Now, just as in the previous step, a trivialization of $\N Q_{i}(A)$ restricts to an $(n-3)$-frame of $\N P_{i}$ by \eqref{BoundaryNormal}. \medskip

\textbf{Step 3: Smoothing corners.} This is the most technical and delicate part because the sphere $Q_{i}$ we defined by gluing $Q_{i}(A)$ and $\overline{Q}_{i}(B)$ is not smoothly immersed: at their common boundary $P_{i}'$, the topological manifold $Q_i$ has a corner singularity. 
By this we mean that there is a homeomorphism $\phi: P_i' \times [-\epsilon, \epsilon] \rightarrow  X$ such that $\phi|_{P_i' \times [-\epsilon, 0]}$ is a smooth embedding into $Q_i(A)$ and $\phi|_{P_i' \times [0, \epsilon]}$ is a smooth embedding into $\overline{Q}_i(B)$, but $\phi=\phi(x,t)$ is not differentiable with respect to $t$ at points of the form $(x,0)$. 
We shall smooth out the corner, replacing $Q_i$ by a smoothly embedded sphere $Q_{i}^{s}$ such that $[Q_i]=[Q_i^{s}]$ in $H_{n-2}(X)$.  

Let us start with smoothing the corner. 
While we believe the present situation to not be a novel case of smoothing, we handle it directly to fix the notation going forward. 
The crucial feature we shall use is that at every point of the common intersection $P_i'$, the normal vectors of $P_i'$ in $Q_i(A)$ and $\overline{Q}_i(B)$ make a nonzero angle. 

Let $\nu(A)$ and $\nu(B)$ be unit sections of $\T Q_{i}(A)|_{P_i'}$ and $\T \overline{Q}_{i}(B)|_{P_i'}$, each normal to $P_{i}'$ and pointing towards the exterior of $Q_{i}(A)$ and $\overline{Q}_{i}(B)$, respectively. 
We first select a smooth interpolation $(\nu_t)_{t \in [-\epsilon, \epsilon]}$ of non-vanishing normal vector fields $\nu_{t} \in \Omega^0(P_i', \N P_i')$ such that $\nu_{-\epsilon}=\nu(B)$ and $\nu_{\epsilon}=\nu(A)$. Recall from Steps 1 and 2 that $\nu(A)$ and $\nu(B)$ are globally linearly independent on $P_i'$. Thus, we can arrange the smooth interpolation to lie entirely in the plane spanned by $\nu(A)$ and $\nu(B)$, so that $\nu_{t}$ is never tangent to $Q_{i}(A)$ and to $\overline{Q}_{i}(B)$ for $t<\epsilon$ and $t>-\epsilon$, respectively. Now, ${\nu}_t$ can be expressed as a linear combination, which we write as 
\[ \nu_t=c_A^t\nu(A)+c_B^t\nu(B). \]

Let us take $Y\cong P_i' \times \R^{n-2}$ to be a regular neighborhood of $P_i'$ in $X$. 
Now, we define extensions $\hat{\nu}(A), \hat{\nu}(B) \in \Omega^0(Y,\T X)$, which can be any smooth sections satisfying the following mild prescriptions: 
\begin{itemize}
    \item $\hat{\nu}(A), \hat{\nu}(B)$ are linearly independent on $Y$, 
    \item $\hat{\nu}(A)|_{Q_i(A) \cap Y}$ is tangent to $Q_i(A)$, 
    \item $\hat{\nu}(B)|_{Q_i(B)\cap Y}$ is tangent to $Q_i(B)$.
\end{itemize}
By simultaneously `straightening' $P_i'$ in $Q_i(A)$ and $Q_i(B)$, we can produce a convenient Riemannian metric on $Y$. We will use this neighborhood to locally deform $Q_i$ only on $Q_i\cap Y$. 

\begin{lemma} There exists a Riemannian metric on $Y$ such that $P_{i}'$ is totally geodesic in both $Q_{i}(A)$ and $Q_{i}(B)$.
\end{lemma}
\begin{proof} Choose normal vector fields $\nu_{1}, \dots, \nu_{n-4} \in \Omega^{0}(P_{i}',\N P_{i}')$ that complete $\nu(A)$ and $\nu(B)$ to a frame of $\N P_{i}' \cong P_i'\times \R^{n-2}$. Take extensions $\hat{\nu}_{i} \in \Omega^{0}(Y,\T X)$ such that $(\hat{\nu}(A),\hat{\nu}(B),\hat{\nu}_1, \dots, \hat{\nu}_{n-4})$ are also linearly independent. After possibly shrinking $Y$, we can define 
a diffeomorphism $\Phi:Y \rightarrow P_{i} \times (-\delta_A, \delta_{A}) \times(-\delta_B,\delta_{B}) \times (-\delta,\delta)^{n-4}$ by flowing along these vector fields. Moreover, this map $\Phi$ will satisfy (recall: $\nu(A),\nu(B)$ are outward normals) 
\[
    \Phi(Q_{i}(A) \cap Y)=P_{i}\times (-\delta_{A}, 0] \times \{0\} \times \{0\},\qquad
    \Phi(Q_{i}(B) \cap Y)=P_{i}\times \{0\} \times (-\delta_B,0] \times \{0\}.
\]
Then the pull-back metric
\[
    g_{Y}:=\Phi^{*}(g_{P_{i}}\oplus da^2\oplus db^2 \oplus g_{\R^{n-4}}).
\]
satisfies the properties of the statement.
\end{proof}

We shall now build the smoothed sphere $Q_i^s$. 
Let $\chi:[-\epsilon, \epsilon] \rightarrow [0,1]$ be a smooth symmetric bump function such that
$\chi(t)=1$ for $|t|\leq \frac{\epsilon}{4}$ and $\chi(t)=0$ for $|t| \geq \frac{3\epsilon}{4}$. Let $s_{0}>0$ be a sufficiently small fixed real number. The desired sphere $Q_{i}^{s}$ is obtained by gluing four pieces:

\begin{itemize}
    \item $Q_{i}\setminus U$ remains untouched, where $U = \Phi^{-1}(P_i' \times (-\epsilon,\epsilon)^2\times \{0\})$,  
    \item $Q_{i}(A)\cap U$ is replaced by 
    \begin{align}\label{Q(A)Deformation}
        Q_{i}(A)^{\nu} \coloneq\ \{\Phi^{-1}(x,t, s_{0}\chi(t),0) \ | \ x\in P_{i}', \,t \in [-\epsilon, 0] \}
    \end{align}
    \item $Q_{i}(B)\cap U$ is replaced by 
    \begin{align}\label{Q(B)Deformation}
            Q_{i}(B)^{\nu} \coloneq\{ \Phi^{-1}\big(x,s_0\chi(t),-t,0 \big) \ | \ x\in P_{i}',\, t \in [0, \epsilon] \,\big\} 
    \end{align}
    \item $P_i'$ is replaced by a new cylinder diffeomorphic to $P_{i}'\times [-\epsilon, \epsilon]$ given by
    \begin{align*}
         \mathcal{C}&=\{ \exp_{x}^{Y}\big(s_{0}\nu_{t}(x)\big) \ |  \ x \in P_{i}', \  t\in [-\epsilon, \epsilon]\}\\
         & =\{\Phi^{-1}(x, s_{0}c_{A}^t(x), s_{0}c_{B}^t(x), 0) \mid  t\in [-\epsilon, \epsilon]\}.
    \end{align*}
\end{itemize}
See Figure \ref{Fig:Smoothing}, which illustrates the smoothing process.
The first three pieces glue together smoothly since $\chi(t) =0$ for $|t| \geq \frac{3\epsilon}{4}$. One also checks that $\mathcal{C}$ glues smoothly on both ends of the cylinder since $\chi(t)=1$ for $t \leq \frac{\epsilon}{4}$, 
$\nu_{-\epsilon}=\nu(B)$, and $\nu_{\epsilon}=\nu(A)$.
Thus, we have found that $Q_i^{s}$ is a smoothly immersed sphere. 

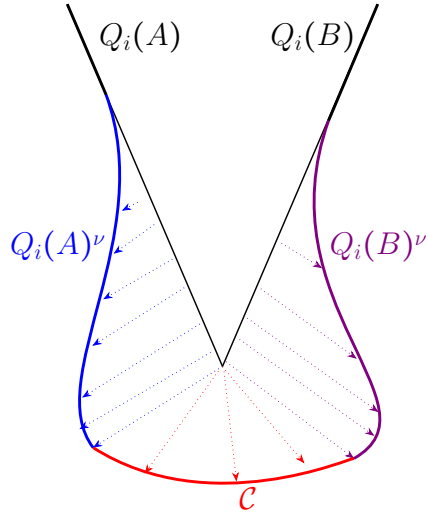
\begin{figure}[h]
\centering
\resizebox{0.4\textwidth}{!}{
  \begin{tikzpicture}[scale=1]
  \coordinate (A) at (2.00, 10.00);
  \coordinate (B) at (5.00, 3.00);
  \coordinate (P134) at (8.00, 10.00);
  \coordinate (C) at (2.50, 1.43);
  \coordinate (L) at (4.77, 3.27);
  \coordinate (M) at (2.24, 1.79);
  \coordinate (N) at (4.56, 3.82);
  \coordinate (O) at (2.29, 2.40);
  \coordinate (P) at (4.26, 4.52);
  \coordinate (Q) at (2.49, 3.40);
  \coordinate (R) at (4.03, 5.07);
  \coordinate (S) at (2.70, 4.30);
  \coordinate (T) at (1.82, 4.88);
  \coordinate (U) at (8.02, 4.88);
  \coordinate (E_1) at (5.50, 0.11);
  \coordinate (V) at (3.68, 5.74);
  \coordinate (W) at (2.90, 5.18);
  \coordinate (X) at (3.36, 6.17);
  \coordinate (Y) at (3.06, 6.01);
  \coordinate (Z) at (5.07, 3.09);
  \coordinate (AB) at (5.33, 3.52);
  \coordinate (AC) at (7.93, 1.66);
  \coordinate (AD) at (5.41, 3.98);
  \coordinate (AE) at (8.01, 2.12);
  \coordinate (AF) at (5.74, 4.50);
  \coordinate (AG) at (7.59, 3.15);
  \coordinate (AH) at (6.14, 5.38);
  \coordinate (AI) at (6.92, 4.90);
  \coordinate (AJ) at (5.03, 2.90);
  \coordinate (AK) at (5.27, 0.78);
  \coordinate (AL) at (5.00, 2.99);
  \coordinate (AM) at (6.58, 1.13);
  \coordinate (AN) at (4.92, 2.91);
  \coordinate (AO) at (3.50, 0.93);
  \coordinate (AQ) at (4.75, 0.00);
  \coordinate (AR) at (7.50, 1.25);
  \coordinate (AS) at (2.75, 8.25);
  \coordinate (AT) at (3.75, 5.50);
  \coordinate (AU) at (1.50, 2.75);
  \coordinate (AV) at (7.05, 7.75);
  \coordinate (AW) at (5.79, 4.22);
  \coordinate (AX) at (9.29, 2.22);
  \coordinate (AR_1) at (7.54, 1.22);
  \coordinate (H) at (6.75, 8.92);
  \coordinate (I) at (3.42, 8.92);

  \draw[thick] (A) -- (B);
  \draw[-Stealth, blue, dotted] (L) -- (M);
  \draw[-Stealth, blue, dotted] (B) -- (C);
  \draw[-Stealth, blue, dotted] (N) -- (O);
  \draw[-Stealth, blue, dotted] (P) -- (Q);
  \draw[-Stealth, blue, dotted] (R) -- (S);
  \node[above, blue, font=\LARGE] at (T) {$Q_i(A)^{\nu}
$};
  \node[above, violet, font=\LARGE] at (U) {$Q_i(B)^{\nu}$};
  \node[red, above, font=\LARGE] at (E_1) {$\mathcal{C}$};
  \draw[thick] (B) -- (P134);
  \draw[-Stealth, blue, dotted] (V) -- (W);
  \draw[-Stealth, blue, dotted] (X) -- (Y);
  \draw[-Stealth, violet, dotted] (Z) -- (AR_1);
  \draw[-Stealth, violet, dotted] (AB) -- (AC);
  \draw[-Stealth, violet, dotted] (AD) -- (AE);
  \draw[-Stealth, violet, dotted] (AF) -- (AG);
  \draw[-Stealth, violet, dotted] (AH) -- (AI);
  \draw[-Stealth, red, dotted] (AJ) -- (AK);
  \draw[-Stealth, red, dotted] (AL) -- (AM);
  \draw[-Stealth, red, dotted] (AN) -- (AO);
  \draw[red, ultra thick] (C) .. controls (4.75, 0.00) and (7.50, 1.25) .. (AR_1);
  \draw[blue, ultra thick] (AS) .. controls (3.75, 5.50) and (1.50, 2.75) .. (C);
  \draw[violet, ultra thick] (AV) .. controls (5.79, 4.22) and (9.29, 2.22) .. (AR_1);
  \node[above, font=\LARGE] at (H) {$Q_i(B)$};
  \node[above, font=\LARGE] at (I) {$Q_i(A)$};
  \draw (A) -- (AS);
  \draw[ultra thick] (P134) -- (AV);
  \draw[ultra thick] (A) -- (AS);
\end{tikzpicture}
}
\caption{Smoothing of $Q_{i}$. }
\label{Fig:Smoothing}
\end{figure}

\textbf{Step 3.1: Embeddedness}. While we know that the inclusion map $\iota_s: Q_i^{s} \hookrightarrow X$ is an immersion and that $\iota_s$ is homotopic to a ($C^{\infty}$-)embedding by Theorem \ref{thm:HomotopicToEmbedding}, we cannot be certain a priori that $\iota_s$ is \emph{regularly homotopic} to an embedding. For this reason, we cannot compute the invariant $S\beta$ on the homology class $[Q_i]$ by studying $\iota_s$ unless we are certain that it is a smooth embedding. We shall now verify that this is the case. 
\begin{lemma}
$Q_i^{s}$ is embedded for $s_0$ sufficiently small. 
\end{lemma}

\begin{proof} Observe that $Q_i(A)^{\nu}\cup \mathcal{C}$ and $Q_i(B)^{\nu} \cup \mathcal{C}$ are each individually embedded for $s_0$ and $\epsilon$ sufficiently small. Indeed, this follows by considering  regular neighborhoods of $Q_i(A)$ and $\overline{Q}_i(B)$ in $X$, respectively. To prove embeddedness, we need only show $Q_i(A)^{\nu} \cap Q_i(B)^{\nu} =\emptyset$ for $s_0$ small enough. We will prove that $Q_i(A)$ and $\overline{Q}_i(B)$ move apart at first order.

Take any distinct points $x \in Q_{i}(A) \cap Y$ and $y \in \overline{Q}_i(B) \cap Y$. 
From now on, we write points on $Y$ using the coordinates given by $\Phi$ so that $x$ and $y$ obtain the form 
\[
    x=\Phi^{-1}(p,r,0,0) ,\qquad y=\Phi^{-1}(q,0,u,0),
\]
where $r ,u\leq 0$. We may suppose $(r,u)\neq (0,0)$, since otherwise $x,y\in P_i'$. Replacing $s_0$ by $s$, under the deformations
defining $Q_{i}(A)^\nu$ and $Q_{i}(B)^\nu$, these points become
\[
x_s=\Phi^{-1}(p,r,s\chi(r),0),
\qquad
y_s=\Phi^{-1}(q,s\chi(u),u,0).
\]
Because the metric is a product, we have
\[
d_Y(x_s,y_s)^2
=
d_{P_{i}'}(p,q)^2
+(-r+s\chi(u))^2
+(-u+s\chi(r))^2.
\]
It follows that
\[
\frac12\frac{d}{ds}d_Y(x_s,y_s)^2
=
\chi(u)(-r+s\chi(u))+\chi(r)(-u+s\chi(r))\geq0.
\]
Since $(r,u) \neq (0,0)$, the inequality is strictly positive when $s > 0$. 
At $s=0$, we also have
\[
\frac{d}{ds}\bigg |_{s=0}d_Y(x_s,y_s)=\frac{1}{2d_{Y}(x,y)}\frac{d}{ds}\bigg|_{s=0}d_Y(x_s,y_s)^2
=  \frac{-\chi(u) r-\chi(r)u}{\sqrt{d_{P_{i}'}(p,q)^2+r^2+u^2}} >0.
\]
In any case, the function $D(s)=d_Y(x_s,y_s)$ is strictly increasing for any pair of distinct initial points $x,y$. Therefore, for $s$ sufficiently small, no point of $Q_{i}(A)^\nu$ can coincide with a point of $Q_{i}(B)^\nu$. 
\end{proof} 

We now resume the proof of Theorem \ref{thm:nightmare_spheres}.\medskip 

Note that by reversing the relevant flows, the inclusion map $\iota_s:Q^s_i \hookrightarrow X$ is homotopic to the inclusion $\iota:Q_s \hookrightarrow X$ and thus $[Q_i]=[Q_i^s]$ in $H_{n-2}(X)$. 
Therefore, by Definitions \ref{Defn:Beta} and \ref{Defn:SBeta}, we can use the smoothly embedded sphere $Q_{i}^{s}$ to compute $S\beta([Q_{i}])$. \medskip 

\textbf{Step 4: Assemble sections.} 
Just shy of a trivialization, we now produce an $(n-4)$-frame of $\N Q_i^s \oplus \varepsilon^1_{\R}$. 

In this step, we shall view the smoothed sphere $Q_i^s$ as coming in three pieces: canonically embedded copies of $Q_i(A)$ and $\overline{Q}_i(B)$, with a new cylinder $\mathcal{C} $ over $P_i'$ connecting them: 
\begin{itemize}
    \item $Q_i^s(A) := \big(Q_i(A)\setminus U) \cup Q_i(A)^{\nu}$, 
    \item $Q_i^s(B) := \big(Q_i(B)\setminus U) \cup Q_i(B)^{\nu}$, 
    \item $\mathcal{C}$, diffeomorphic to $P_i'\times [0,1]$. 
\end{itemize}

For clarity, we note the preferred diffeomorphism 
$f_A: Q_i(A) \rightarrow Q_i^s(A)$ is the identity on $Q_i(A) \setminus U$ and applies the flow described in \eqref{Q(A)Deformation} on $U$; the corresponding map $f_B: Q_i(B)\rightarrow Q_i^{s}(B)$ is defined analogously. 
Let us write $f_{\mathcal{C}}: P_i'\times [-\epsilon, \epsilon] \rightarrow \mathcal{C}$ for the preferred diffeomorphic parametrization of $\mathcal{C}$
given by $f_{\mathcal{C}}(x,t) = \exp_{x}^{Y}(s_0\nu_t(x))$. 
We then obtain diffeomorphisms $f_i^t: P_i' \rightarrow f_{\mathcal{C}}(P_i' \times \{t\})=:P_{i}^{'t}$ for any $t \in [-\epsilon, \epsilon]$. \medskip 

We now proceed to produce an $(n-4)$-frame $\mathbf{X} \in \Omega^0(Q_i^s, V_{n-4}(\N Q_i^s \oplus \varepsilon^1_{\R}))$. We shall only need $\mathbf{X}$ to be continuous. To do so, we will first specify the restriction of $\mathbf{X}$ to each of $Q_i^s(A)$ and $Q_i^s(B)$, and then interpolate across these two local sections via $\mathcal{C}$. 

To start, we describe the construction of $\mathbf{X}$ on $Q_i^s(A)$ and $Q_i^s(B)$. Recall that in Steps 1 and 2 we found global frames $\mathbf{V} = (V_{j})_{j=1}^{n-3}$ and $\mathbf{W} = (W_{j})_{j=1}^{n-3}$ of $\N Q_i(A)$ and $\N \overline{Q}_i(B)$.
Up to an application of Gram-Schmidt, we can suppose 
$\mathbf{W},\mathbf{V}$ are orthonormal. 
By using the canonical diffeomorphisms $f_A$ and $f_B$, we can push forward $\mathbf{V}$ and $\mathbf{W}$ to define local sections
$\mathbf{X}|_{Q_i^s(A)} = (f_A)_*(\mathbf{V})$ 
and $\mathbf{X}|_{Q_i^s(B)} = (f_B)_{*}(\mathbf{W})$. 
Since $f_A$ and $f_B$ are induced from local  diffeomorphisms of $X$, we see that $(f_A)_*$ and $(f_B)_*$ do push forward normal vectors to normal vectors. To see this, one can view the normal bundle $\N M$ of a submanifold $M \subset \hat{M}$ at a point $x \in M$ as $\T_{x}\hat{M}/T_{x}M$.

Next, we make some simple but important identifications.
Recall by Steps 1 and 2, we have preferred isomorphisms $\N P_i' \cong \N Q_i(A)|_{P_i'} \oplus \varepsilon^1_{\R}$ and $\N P_i' \cong \N \overline{Q}_i(B)|_{P_i'} \oplus \varepsilon^1_{\R}$.
Since the vector bundle $\N P_i' \rightarrow P_i'$ is trivial, we can fix a once-and-for-all trivialization 
$\N P_i' \cong \varepsilon^{n-2}_{\R}$. 
Thus, the restrictions of our original frames $\mathbf{V}|_{P_i'}$ and $\mathbf{W}|_{P_i'}$ can be viewed, for the moment, as of type $\mathbf{V}|_{P_i'}, \mathbf{W}|_{P_i'} \in \Omega^0(P_i', V_{n-4}(\R^{n-2}))$.

The one key property that shall guarantee our success is that we were able to select $\mathbf{V}= \mathbf{V}(x,y)$ and $\mathbf{W} = \mathbf{W}(x,y)$ to depend only on $x$. In particular, $\mathbf{V}|_{P_i'}$ and $\mathbf{W}|_{P_i'}$ depend only on $x \in \sphere^1$. 
From this perspective, we can view $\mathbf{V}|_{P_i'},\mathbf{W}|_{P_i'}$ as objects of type $\sphere^1 \rightarrow V_{n-4}(\R^{n-2})$. Now, the Stiefel manifold $V_j(\R^m)$ is $(m-j-1)$-connected. Thus, $\pi_1(V_{n-4}(\R^{n-2}))=0$. 
In particular, $\mathbf{V}|_{P_i'},\mathbf{W}|_{P_i'}: \sphere^1 \rightarrow V_{n-4}(\R^{n-2})$ are homotopic. 
Let us define $\mathbf{Y}_t: \sphere^1 \rightarrow V_{n-4}(\R^{n-2})$, for $-\epsilon \leq t \leq \epsilon$, as a continuous map such that $\mathbf{Y}_{-\epsilon} = \mathbf{V}$ and $\mathbf{Y}_{\epsilon} = \mathbf{W}$. 

Using the fixed identifications $f^t_i$, we can push forward $\mathbf{Y}_t \in \Omega^0(P_i', \N P_i')$ to the section $(f^t_i)_*(\mathbf{Y}_t)\in \Omega^0(P^{'t}_i,(\N Q_{i}^s \oplus \varepsilon_{\R}^{1})|_{P_{i}^{'t}})$. The same remarks apply as earlier: this is possible because $f^t_i$ is not just a diffeomorphism between $P_i'$ and $P_i^{'t}$, but the restriction of a local diffeomorphism of $X$. Since these $f_i^t$ vary smoothly, we then obtain $\mathbf{X}|_{\mathcal{C}}$ by letting $t$ vary between $-\epsilon$ and $\epsilon$. By construction, we see that 
$\mathbf{X}$ is a well-defined, continuous global section of $V_{n-4}(\N Q_i^s \oplus \varepsilon^1_{\R})$. Indeed, on the intersections $\mathcal{C} \cap Q_i^s(A)$ and $\mathcal{C} \cap Q_i^s(B)$, the two definitions of $\mathbf{X}$ agree. This completes \emph{Step 4}.
\medskip

\textbf{Step 5: Conclude.} We now prove that $\N Q_{i}^{s} \oplus \varepsilon_{\R}^{1}$ must be trivial, given that it has $n-4$ linearly independent sections. 
Note that the exceptional case $n=9$ is trivial by Table \ref{Table:HomotopyGroupsLowDimensions}. In the remaining cases that $n\geq 7$ and $n \neq9$, by stable triviality of $\N Q_i^s$, we need only exclude the possibility that $\N Q_{i}^{s} \oplus \varepsilon_{\R}^{1}$ is isomorphic to $\T\sphere^{n-2}$. 
Indeed, for such $n$ odd, by Tables \ref{Table:StableHomotopyGroups} and \ref{Table:HomotopyGroupsLowDimensions}, and the surjectivity of $\mathbf{S}_*: \pi_{n-3}(\SO(n-2)) \rightarrow \pi_{n-3}(\mathbf{SO})$, the homomorphism $\mathbf{S}_*$ has kernel $\Z_2$, which is generated by the clutching function of $\T\sphere^{n-2}$.  

We now explain why $\N Q_{i}^{s} \oplus \varepsilon_{\R}^{1}$ is not isomorphic to $\T\sphere^{n-2}$. 
In \cite{Ada60}, Adams provides a sharp upper bound 
for the maximum number $\rho(k)$ of linearly independent vector fields one can obtain on the sphere $\sphere^{k}$. To this end, write $k+1=2^{4a+b}(2\ell+1)$ for integers $0\leq b \leq 3$ and $a,\ell \geq 0$, and then $\rho(k)= 2^{b}+8a-1$. We now observe the following elementary inequality: 
\[ \rho(k)\leq 2\log_{2}(k+1)+2. \] 
In our setting, $k=n-2$. It is straightforward to verify that $n-4>2\log_{2}(n-1)+2$ for $n\geq 14$, so $\N Q_{i}^{s} \oplus \varepsilon_{\R}^{1}$ cannot be isomorphic to $\T\sphere^{n-2}$ for $n \geq 14$, since the former vector bundle admits more linearly independent sections than the latter. Only three cases remain, namely $n \in \{7,11,13\}$. One can easily check $n-4 > \rho(n-2)$ in each case, concluding the proof. 
\end{proof}

\subsubsection{Computing $S\beta$ for $\Psi_{n-2}(p_*^{-1}[\Sigma_{i}])$}\label{Subsec:SeparatingSpheres}
We now handle the other family of generators of $H_{n-2}(X)$ from Lemma \ref{Lem:Dichotomy}. Recall that $\Sigma_i \subset S_{gen}$ was originally defined as an $(n-2)$-sphere separating $C_i$ from the remaining exceptional circles. 
We shall now lift $\Sigma_i$ from $S_{gen}$ to $X_{gen}$.
We then prove that $S\beta$ vanishes on this class, once included into $H_{n-2}(X)$. 

To start, we have a general lemma about vector bundles on spheres. We shall apply it to $\mathcal{R}^\bot|_{\Sigma_i}\rightarrow \Sigma_i$ momentarily. In particular, if $\mathcal{R}^\bot|_{\Sigma_i}\oplus \varepsilon^1_{\R}$ is trivial, then $\mathcal{R}^\bot|_{\Sigma_i}$ has a section. 
\begin{lemma}\label{lem:section} Let $m \geq 4$ be even and $E \rightarrow \sphere^{m+1}$ be a vector bundle of rank $m$. 
If $E \oplus \varepsilon^1_{\R}$ is trivial, then $E$ admits a non-vanishing section. 
\end{lemma}

\begin{proof}
By hypothesis, we have a bundle isomorphism 
$E\oplus\varepsilon^1_{\mathbb R}
\cong \varepsilon^{m+1}_{\mathbb R}$.
Choose bundle metrics and an orthogonal trivialization
$\Phi:E\oplus\varepsilon^1_{\mathbb R}
\rightarrow \varepsilon^{m+1}_{\R}$. 
Let $s \in \Omega^0 \big(\sphere^{m+1}, \sphere(E\oplus \varepsilon^1_{\R}) \big)$ be a unit section of $\varepsilon^1_{\R}$. 
Denote by $\pr_2: \varepsilon_{\R}^{m+1} \rightarrow \R^{m+1}$ the projection map of $\varepsilon^{m+1}_{\R}$.  
Then $g:=\pr_2\circ \Phi(s)$ defines a map $g:\sphere^{m+1}\longrightarrow \sphere^m\subset\mathbb R^{m+1}$.
Since $\Phi$ is orthogonal, for every $x\in \sphere^{m+1}$ the image
of the fiber $E_x$ under $\Phi$ is the orthogonal complement of $g(x)$. Therefore, we obtain an isomorphism 
$E\cong g^*\T\sphere^m$. It remains to show that $g^*\T\sphere^m$ admits a nowhere-vanishing section. Note that this task is non-trivial as $\T\sphere^m$ admits no such section. 

We now use some algebraic machinery to produce the desired section. 
Consider the unit tangent bundle fibration
\[
\sphere^{m-1}\stackrel{i}{\longrightarrow} \T^1\sphere^m
\stackrel{p}{\longrightarrow} \sphere^m,
\qquad
p(u,v)=u.
\]
The relevant part of its long exact sequence of homotopy groups is
\begin{center}
\begin{tikzcd}[scale=1.5, row sep=2.5em]
 \pi_{m+1}(\sphere^{m-1}) \arrow{r}{i_*} &
     \pi_{m+1}(\T^1 \sphere^m) \arrow{r}{p_{*}} &
     |[label=below:{\substack{\cong \\ \\  \Z_2}}]|
    \pi_{m+1}(\sphere^{m})
    \arrow[dll, rounded corners=8pt, to path={ 
        -- ([xshift=3ex]\tikztostart.east) 
        |- ([yshift=2ex]\tikztotarget.north) 
        -- (\tikztotarget) 
    }] \\
    |[label=below:{\substack{\cong \\ \\ \Z_{2}}}]| \pi_{m}(\sphere^{m-1}) \arrow{r}{i_*} &
    |[label=below:{\substack{\cong \\ \\ \Z_{2}}}]| \pi_{m}(\T^1 \sphere^{m})  \arrow{r}{p_{*}} &
    |[label=below:{\substack{\cong \\ \\ \Z}}]|
    \pi_{m}(\sphere^{m})
\end{tikzcd}
\end{center}

Let us explain the above isomorphisms. Now, the first unstable homotopy group $\pi_{m}(\sphere^{m-1})$ is $\Z_2$ for $m\geq3$. Next, we claim that $\pi_m(\T^1\sphere^m) \cong \Z_2$. 
Indeed, since $\T^1\sphere^m \cong V_2(\R^{m+1})$ is $(m-2)$-connected, we can verify the claim by applying Proposition \ref{Prop:HurewiczUpOne},  
after employing the Gysin sequence to compute $H_{m-1}(\T^1\sphere^m)\cong \Z_2$ and $H_m(\T^1\sphere^m)=0$ when $m$ is even. We now advance with the above exact sequence and the specified isomorphisms. 

Observe that the final map $p_*$ is zero and hence $i_*:\pi_m(\sphere^{m-1})\rightarrow \pi_m(V_2(\R^{m+1}))$ is an isomorphism. Thus, $p_*:\pi_{m+1}(\T^1\sphere^m)\rightarrow \pi_{m+1}(\sphere^m)$ is surjective. 
Therefore, we may choose a lift
$\widetilde g:\sphere^{m+1}\longrightarrow \T^1\sphere^m$
such that $p\circ\widetilde g=g$. Writing $\widetilde g(x)=\big(g(x),v(x)\big)$, then $v(x)\in \T_{g(x)}\sphere^m$, so $v$ defines a non-vanishing
section of $g^*\T\sphere^m\cong E$. 
\end{proof}

We intentionally ignored regularity in the previous proof for the following reason. 
\begin{remark}
Let $M$ be a smooth manifold. Then a smooth vector bundle $E \rightarrow M$ admits a continuous non-vanishing section if and only if it admits a smooth non-vanishing section. 
\end{remark}

We now verify the hypothesis of the previous lemma for $\mathcal{R}^\bot|_{\Sigma_i} \rightarrow \Sigma_i$. 

\begin{lemma}\label{lem:trivial_Rperp} Let
$n\geq7$ be odd and $\Sigma_{i}$ be a separating sphere in $S_{gen}$. Then $\mathcal{R}^{\perp}_{gen}|_{\Sigma_i}\oplus \varepsilon^1_{\R}$ is a trivial vector bundle. 
\end{lemma}

\begin{proof} Let $D_{i}$ be the disk bounded by $\Sigma_{i}$ in $\sphere^{n-1}$ and containing $C_{i}$. We will show that the vector bundle $(\mathcal{R}^{\perp}_{gen}\oplus \varepsilon_{\R}^{1})_{|_{\Sigma_{i}}}$ extends to $D_{i}$. We emphasize: the claim is non-trivial because the original projection $\mathcal{R}^{\perp}|_{D_i} \rightarrow D_i$ is not an honest vector bundle as its fibers have non-constant dimension. 
Let $N(C_{i}) \cong C_{i} \times \mathbb{D}^{n-2}$ be the closure of a regular neighborhood of $C_{i}$ in $\sphere^{n-1}$. We imagine $N(C_i)$ to be obtained by the Riemannian exponential map of a small $\varepsilon$-band in the normal bundle of $C_i$. The vector bundle $\mathcal{R}^{\perp}_{gen} \oplus \varepsilon_{\R}^{1}$ is clearly defined on $D_{i} \setminus \mathrm{int}(N(C_{i}))$. The goal is thus to extend the vector bundle from $\partial N(C_{i})$ to all of $N(C_{i})$. Now, $\partial N(C_{i})$ is diffeomorphic to $C_{i} \times \sphere^{n-3}$, where $\{p\} \times \sphere^{n-3}$ is identified with the linking sphere $S_{i}(p)$, which is in turn identified with the sphere $\sphere_{\varepsilon}(\N_pC_i)=\{x \in \N_pC_i : ||x||=\varepsilon\}$.

By Theorem \ref{thm:EulerClassXgen}, the restriction of $\mathcal{R}^{\perp}_{gen}$ to $S_{i}(p)$ is isomorphic to the tangent bundle of $S_{i}(p)$ itself. 
We recall how the isomorphism was constructed in the proof, as this identification is crucial going forward. Each circle $C_{i}$ comes equipped with a covering map $\pi: C_{i} \rightarrow \mathbb{P}(\mathcal{P})$ that associates to a point $p \in C_{i}$ the unique projective class $[\psi] \in \mathbb{P}(\mathcal{P})$ such that $p$ is an eigenvector for $\psi$. The isomorphism $L:\mathcal{R}^{\perp}_{|_{S_i(p)}}\rightarrow \T S_{i}(p)$ is obtained by the fiberwise multiplication of an endomorphism $L=L(p,\psi) := \psi-\lambda \mathrm{Id} $, where $p$ is a $\lambda$-eigenvector for $\psi$. In particular, $L(p,\psi)$ depends on the choice of a representative $\psi \in [\psi]$. 

We now recall when a continuous choice of $\psi$ can be made along all of $C_i$. 
From the proof of Lemma \ref{lem:number_circles}, we know that for $n$ odd, the singular circles $C_{i}$ come in two flavors, specifically \emph{type (a)} and \emph{type (b)} within the lemma: 
\begin{itemize} 
    \item \emph{Type (a)}: the projection $\pi$ is a double cover, hence it lifts to a map $\tilde{\pi}: C_{i} \rightarrow \mathcal{P}$, meaning a consistent choice of representatives in the pencil exists along $C_i$; \medskip
    \item \emph{Type (b)}: the map $\pi$ is a bijection. Hence, we instead parameterize $C_{i}$ by a smooth map $\gamma=\gamma_{i}: [0,1] \rightarrow C_{i}$, which is injective on $(0,1)$, and such that $\gamma_i(0)=\gamma_i(1)$. We can lift $\pi \circ \gamma:[0,1]\rightarrow \mathbb{P}(\mathcal{P})$ to a map $\psi:[0,1]\rightarrow \mathcal{P}$, i.e., such that $\psi(t)\gamma(t) = \lambda(t)\gamma(t)$. By elementary covering theory, $\psi$ will necessarily satisfy $\psi(0)=-\psi(1)$.
\end{itemize}
We employ these maps $L(p,\psi)$ now, still keeping the point $p$ fixed. 
For a fixed slice $S_i(p)$, this allows us to trivialize its fibers as follows: 
\begin{align}\label{Kappa_Isomorphism}
    \kappa(p,\psi):(\mathcal{R}^{\perp} \oplus \varepsilon_{\R}^{1})_{|_{S_i(p)}} \rightarrow  \N_pC_i \qquad \big(x,(y,c)\big)\mapsto L(p,\psi)(y)+cx.
\end{align}
For every $x \in S_i(p)$, the map $\kappa(p,\psi)|_x$ is a linear isomorphism. 
In other words, with $\psi =\psi(p)$ fixed, we obtain a uniform identification of $(\mathcal{R}^\bot \oplus \varepsilon^{1}_{\R})|_x$, for every $x \in S_i(p)$, with the normal space to $C_i$ at the corresponding point $p$.

The map $\kappa =\kappa(p,\psi)$ varies smoothly with respect to $p$ and $\psi(p)$. Moreover, $\kappa$ is \emph{linear} in $\psi$, meaning $\kappa(p, -\psi(p))=-\kappa(p, \psi(p))$.
We now consider cases for the circle $C=C_i$. \medskip

\textbf{Case (i): $C$ is of type (a).} We shall produce a vector bundle isomorphism exhibiting triviality. We take for granted that a continuous choice $\psi=\psi(p)$ has been chosen. 
Let $\pr_1: \partial N(C_i)\rightarrow C_i$ denote the projection map. 
We write $(p,x) \in \partial N(C_i)$ for a point in the base and again $(y,c)$ for a point in the fiber. 
Then we obtain a vector bundle isomorphism 
\[ \kappa:(\mathcal{R}^\bot \oplus \varepsilon_{\R}^{1})_{|_{\partial N(C_i)}}\rightarrow \pr_1^*(\N C_i ) \qquad \big( (p,x), (y,c)\big)\mapsto \big( (p,x),\kappa(p,\psi)(x,y,c)\big).\]
That is, $\kappa|_{(p,x)}$ is given by $\kappa(p,\psi)|_x$ from \eqref{Kappa_Isomorphism}.
Since $C_i$ and $\sphere^{n-1}$ are orientable, the normal bundle $\N C_i \subset \T \sphere^{n-1}$ is trivial and hence so is the pullback bundle $\pr_1^*(\N C_i)$. Thus, $(\mathcal{R}^\bot\oplus \varepsilon^1_{\R})|_{\partial N(C_i)}$ is trivial as desired, meaning the bundle extends to all of $D_i$. \medskip

\textbf{Case (ii): $C$ is of type (b)}. Here, we employ a similar strategy as in \emph{Case (i)}, except that now a continuous choice of $\psi$ cannot be made, and $\psi(0) = -\psi(1)$. 
In this case, we can say that the restricted vector bundle $(\mathcal{R}_{gen}^{\perp} \oplus \varepsilon_{\R}^{1})|_{N(C_{i})}$ is 
realized as a mapping torus and then conclude. As in the definition of \emph{type (b)}, we denote $\gamma_i:[0,1]\rightarrow C_i$ as a parametrization.

Let $I = [0,1]$ and denote $\pr_1: I \times \sphere^{n-3} \rightarrow I$. 
We have a pullback bundle 
$E := \pr_1^*(\N \gamma_i) \rightarrow I$. 
Note that the fibers $E_{\mid 0}$ and $E_{\mid 1}$ are canonically identified since $\gamma_i(0)=\gamma_i(1)$. 
We can define a similar bundle map as earlier, now 
exchanging the domain and codomain:
\[ \kappa^{-1}:\pr_1^*(\N\gamma_i) \rightarrow (\mathcal{R}^\bot\oplus \varepsilon^1_{\R})|_{\partial N(C_i)} \qquad \big((t,x), (y,c)\big)\mapsto \big( (\gamma_i(t),x), \kappa^{-1}(\gamma_i(t), \psi(t))(x,y,c)\big). \] 
By construction, the map $\kappa^{-1}|_{[0,1) \times \sphere^{n-3}}$ is injective, but $\kappa^{-1}|_{ \{0\} \times \sphere^{n-3}}$ and $\kappa^{-1}|_{\{1\}\times \sphere^{n-3}}$ have the same images. Examining the formula \eqref{Kappa_Isomorphism}, we see that, in an appropriate trivialization of $\pr_1^*(\N \gamma_i)$, we have obtained 
a mapping torus description of the bundle of interest:
\begin{align}\label{MappingTorus}
  ( \mathcal{R}^\bot\oplus \varepsilon^1_{\R})|_{\partial N(C_i)} \cong \big([0,1] \times \sphere^{n-3}\times (\R^{n-3}\oplus \R) \big)/ \sim, \; \text{where} \; (0,x, (y,c) ) \sim  (1, x, (- y, c)).
\end{align}
Note that the gluing map $F \in \GL(n-2,\R)$ such that $(0, x,y,c) \sim (1, x,F(y,c))$ is given by $F=(-\id_{\R^{n-3}}) \oplus \id_{\R}$. Thus, $F$ is in the identity component of $\GL(n-2, \R)$ since $n$ is odd. 
It follows that $\mathcal{R}^{\perp}_{gen} \oplus \varepsilon_{\R}^{1}$ is a trivial vector bundle over $\partial N(C_i)$ and thus can be extended to $N(C_{i})$ as claimed, which concludes \emph{Case (ii)} and hence the proof.
\end{proof}

Combining Lemma \ref{lem:section} and Lemma \ref{lem:trivial_Rperp} we conclude the following. 
\begin{corollary}[Lifting Separating Spheres]\label{Cor:SeparatingSphereLift}
Let $n \geq 7$ be odd and let $\Sigma_i \subset S_{gen}$ be a separating sphere. Then the vector bundle $\mathcal{R}^\perp_{\mathrm{gen}}\big|_{\Sigma_i}
\rightarrow \Sigma_i $
admits a non-vanishing section.
\end{corollary}
  
Now, we can easily demand the separating sphere $\Sigma_i \subset S_{gen}$ to be smoothly embedded. 
Via Corollary \ref{Cor:SeparatingSphereLift}, we can lift $\Sigma_i$ to $X_{gen}$, realizing a smoothly embedded sphere representing our homology class of interest. 

\begin{definition}
Let $\Sigma_i^* \subset X_{gen}$ be a smooth lift of $\Sigma_i \subset S_{gen}$. 
\end{definition}

We will write $[\Sigma_i^*]$ for the homology class $[\Sigma_i^*] \in H_{n-2}(X)$ obtained by pushforward from $X_{gen}$ to $X$. 
Now, we seek to compute 
$S\beta([\Sigma_i^*])$, which involves understanding the (once-stabilized) normal bundle of $\Sigma_i^*$ in $X$. Since $X_{gen}$ is open in $X$, we have the luxury to work exclusively in this submanifold. That is, 
the normal bundles of $\Sigma_i^*$ in the ambient spaces $X_{gen}$ and $X$ are canonically identified. 

To study the vector bundle, $\N_{X_{gen}}\Sigma_i^*\oplus \varepsilon^1_{\R}$, we now introduce one more object. 
Namely, we define an open $(2n-4)$-manifold $W$: the total space of the vector bundle $\mathcal{R}^{\perp}_{gen} \rightarrow S_{gen}$. Here, we denote by $\mathcal{R}^{\perp}_{gen}$ the $(n-3)$-vector bundle over $S_{gen}$ obtained by restricting the projection $\mathcal{R}^{\perp} \rightarrow \sphere^{n-1}$ to $S_{gen}$. 
Now, there is a natural inclusion $\iota: X_{gen} \hookrightarrow W$, since $X_{gen}=\sphere(\mathcal{R}^\bot_{gen})$ is exactly the associated sphere bundle. 
We may regard $\Sigma_i^*$ as embedded in $W$ after post-composing with $\iota$. 

Because the normal bundle of $X_{gen}$ in $W$ in trivial, the following relation holds between the normal bundles $\N_{X_{gen}}\Sigma_{i}^{*}$ of $\Sigma_{i}^{*}$ in $X_{gen}$ and $\N_{W}\Sigma_{i}^{*}$ of $\Sigma_i^*$  in $W$:
\begin{equation}\label{eq:normal_splitting_W}
    \N_{W}\Sigma_{i}^{*}\cong\N_{X}\Sigma_{i}^{*} \oplus \varepsilon_{\R}^{1} .
\end{equation}
Our goal thus reduces to the study of the normal bundle of $\Sigma_{i}^{*}$ seen inside $W$. 
We can simplify further by replacing $\Sigma_i^*$ by an isotopic companion $\Sigma_i^0$, namely the graph of the zero section of $W$ restricted to $\Sigma_i \subset S_{gen}$. 
There is then a clear relationship between $\Sigma_i^*$ and $\Sigma_i^0$ that we record formally.

\begin{lemma}\label{lem:isotopy}
$\Sigma_i^*$ and $\Sigma_i^0$ are isotopic in $W$. In particular, $\N_{W}\Sigma_{i}^{*} \cong \N_{W}\Sigma_{i}^{0}$.
\end{lemma}

We can finally compute $S\beta$ on the class $[\Sigma_i^*]$. 
\begin{corollary}\label{cor:Sbeta_Sigma} $S\beta([\Sigma_{i}^*])=0$.     
\end{corollary}

\begin{proof} 
Endow $W$ with the Riemannian metric induced by ambient space $\R^{2n}$. 
We denote $p: \Sigma_i^0 \rightarrow \Sigma_i$ for the projection. 
Observe that the normal bundle $\N_{W}\Sigma_i^0$ splits as:
\begin{equation}\label{eq:split_normal}
    \N_{W}\Sigma_{i}^{0} \cong p^*({\mathcal{R}^{\perp}_{gen}}|_{\Sigma_{i}} \oplus \varepsilon_{\R}^{1}).
\end{equation}
Indeed, as $\Sigma_{i}^{0}$ lies inside the zero section, the vertical fiber $\mathcal{R}^{\perp}_{gen}$ is contained in the normal bundle. The   factor $\varepsilon^1_{\R}$ comes from the normal bundle $\N_{S_{gen}}\Sigma_i\cong \N_{\sphere^{n-1}}\Sigma_i$, which is trivial. 

Since $\Sigma_i^*$ is smoothly embedded, we conclude that 
$S\beta([\Sigma_i^*])=0$ by combining equations \eqref{eq:normal_splitting_W}, \eqref{eq:split_normal}, Lemma \ref{lem:isotopy} and Lemma \ref{lem:trivial_Rperp}.  
\end{proof}

We have now all the ingredients to determine the topology of the fibers for $n$ odd.

\begin{theorem}[Fiber Classification, Odd Case]
Let $n \geq 7$ be odd and $X = \hat{\mathcal{B}}(\mathcal{P})$ be the base of pencil of $\mathcal{P}=\T_o\Ha^2_{\pr}$. Then $X$ is almost diffeomorphic to $\#_{i=1}^{2n-1}(\sphere^{n-2}\times \sphere^{n-3})$. 
\end{theorem}
\begin{proof} 
Theorem \ref{thm:hom_odd} asserts that $X$ is $(n-4)$-connected, has torsion-free homology, and has Betti number $b_{n-2}=2n-1$. By Theorem \ref{thm:nightmare_spheres}, Corollary \ref{cor:Sbeta_Sigma} and Lemma \ref{Lem:Dichotomy}, the invariant $S\beta$ vanishes for $X$. 
The result then follows from Theorem \ref{thm:WallSimplified}. 
\end{proof}

\begin{corollary}
Let $n\geq 7$ be odd, $\rho: \pi_1S\rightarrow \PSL(n,\R)$ be a Hitchin representation, and $\Omega_{\rho} \subset \mathcal{F}_{1,n-1}$ the domain \eqref{Omega_Thick}, with universal cover $\hat{\Omega}_{\rho} \subset V_2(\R^n)$. 
Then there is a smooth fiber bundle projection
$\hat{\Omega}_{\rho} \rightarrow \Ha^2$ with fiber almost diffeomorphic to $\#_{i=1}^{2n-1}(\sphere^{n-2}\times \sphere^{n-3})$. 
\end{corollary}

\section{Special Low-Dimensional Cases}\label{Sec:SpecialCases}

In this section, we treat the special cases $n\in\{3,4,5\}$ individually. The case $n=3$ is known and provides a sanity check for our methods. When $n=4$, the fiber $\hat{\mathfrak{F}}$ is a 3-manifold. This case is exceptional since $\hat{\mathfrak{F}}$ is not simply connected. When $n=5$, the fiber $\hat{\mathfrak{F}}$ is a 5-manifold and Wall's classification does not apply. Moreover, the calculation of homology does not quite fit the general pattern, and must be addressed separately. 

\subsection{\texorpdfstring{$n=5$}{n=5}}\label{Sec:n=5}

The study of the topology of $X$ for $n=5$ is similar to the general case, but some homological computations differ due to low-dimensionality. 
We will also see that the diffeomorphism type of $X$ is easier to determine and will follow directly from work of Smale \cite{Sma62}, 
later extended by Barden \cite{Bar65}, 
without having to refer to the more general and involved theory of Wall \cite{Wal67}. Indeed, there is a complete classification of simply-connected closed 5-manifolds up to diffeomorphism, unlike the more general case of highly connected closed odd-dimensional manifolds for $2n+1\geq 7$. \medskip 

Recall that $X =\hat{\mathcal{B}}(\mathcal{P})$ is the total space of an almost-fibration $p: X \rightarrow \sphere^{4}$ from Lemma \ref{Lem:StructureLemma}. There are $r(5) =6$ singular circles in this case by Lemma \ref{lem:number_circles}. 
Again we compute the homology of $X$ by applying Mayer-Vietoris, where $A$ and $B$ are as follows:
\begin{itemize}
    \item $A$ is a regular neighborhood of $X_{sing} =p^{-1}(S_{sing})$ in $X$. Here, we have singular locus 
    $X_{sing} \cong \bigsqcup_{i=1}^{4}(\sphere^1\times\sphere^{2}) \sqcup \bigsqcup_{i=1}^{2}(\sphere^{1} \tilde{\times} ~ \sphere^2)$ by Corollary \ref{cor:orientability_singular}. 
    \item $B = X_{gen}$ is an orientable $\sphere^{1}$-bundle over $S_{gen} \simeq \big(\bigvee_{i=1}^{6}\sphere^{2} \big) \vee \big(\bigvee_{i=1}^{5} \sphere^{3} \big)$ with Euler class $e(X_{gen}) = (\pm 2,\dots, \pm 2)\in H^2(S_{gen},\Z) \cong \Z^{6}$.
    \item $A \cap B$ is homotopy equivalent to an orientable $\sphere^{1}$-bundle over $ \bigsqcup_{i=1}^{6} (\sphere^1\times \sphere^{2})$.  
\end{itemize}
We first compute the homology of $B$ and $A\cap B$. 
\begin{lemma}[Homology of Generic Locus, $n=5$]\label{Lem:CohomologyGenericLocus_5} We have
\begin{align}\label{CohomologyGenericLocus_5}
    H_k(B,\Z) = \begin{cases}
        \Z & k =0,\\
        \Z_2 &  k =1, \\
        \Z^{5} & k=2, \\
        \Z^{11} & k=3,\\
        \Z^{5} & k=4.\\
    \end{cases}
\end{align}
\end{lemma}

This calculation is a straightforward application of the Gysin sequence and completely analogous to the calculations in Lemma \ref{Lem:CohomologyGenericLocus}, so we omit the details. \medskip 

Next, we compute the homology of $A \cap B$. Comparing with Lemma \ref{Lem:HomologyIntersection}, we see the homology differs from the general case.  
This discrepancy is why we treat $n=5$ separately.  

\begin{lemma}[Homology of Intersection, $n=5$]\label{Lem:HomologyIntersection_5} We have 
\begin{align}\label{CohomologyOfIntersection_5}
    H_k(A \cap B) = \begin{cases}
        \Z^{6} & k =0,\\
        \Z^6 \oplus \Z_2^6 & k= 1,\\
        \Z_2^{6} & k =2 ,\\
        \Z^6 & k \in \{3,4\} .
    \end{cases}
\end{align}
\end{lemma}
\begin{proof} We recall that $A\cap B$ is homotopy equivalent to the disjoint union of six $\sphere^1$-bundles over $\sphere^1\times \sphere^2$ with each Euler classes in $H^2(\sphere^1\times \sphere^2)\cong \Z$ equal to $2$. It is enough to compute the cohomology of one connected component of $A \cap B$, which we denote by $E$. 

Since $E$ is homotopy equivalent to a closed orientable 4-manifold, $H^4(E)\cong \mathbb{Z}$. 
Now, in the Gysin exact sequence, the map $\cup e:H^k(B)\rightarrow H^{k+2}(B)$ is injective for $k \leq 1$. Hence, for $k \leq 0$, the Gysin sequence decomposes into short exact sequences:
\[ 0 \longrightarrow H^k(B) \stackrel{\cup e}{\longrightarrow}  H^{k+2}(B) \stackrel{p^*}{\longrightarrow} H^{k+2}(E) \longrightarrow 0.\]
Taking $k \in \{-1,0\}$, one finds $H^1(E) \cong \Z$ and $H^2(E) \cong \Z_2$. 
When $k=1$, we have the exact sequence 
\[ 0 {\longrightarrow} \Z\cong H^1(\sphere^1\times \sphere^2)\stackrel{\cdot 2}{\longrightarrow} \Z \cong H^3(\sphere^1\times \sphere^2) \stackrel{p^*}{\longrightarrow} H^3(E) \stackrel{\partial}{\longrightarrow} \Z \cong H^2(\sphere^1\times \sphere^2) \longrightarrow 0.\]
Hence, $H^3(E)\cong \Z \oplus \Z_2$. The result follows by taking six direct sums. 
\end{proof}

With the homology of $A$ (recall Lemma \ref{lem:homology_singular}), $B,A \cap B$ in hand, we can now compute the homology of $X$. 
\begin{theorem}\label{thm:hom_5}
The simply connected 5-manifold $X = \hat{\mathcal{B}}(\mathcal{P})$ has torsion-free integer homology and second Betti number $b_2=9$. 
\end{theorem}
\begin{proof} 

$X$ is simply-connected by Lemma \ref{lem:simplyconnected}. By Poincar\'e duality, $H_{3}(X)$ is a free $\Z$-module with the same rank as $H_{2}(X)$, so it is sufficient to compute $H_{2}(X)$. 
The Mayer-Vietoris exact sequence is the following:
\begin{align*}
    \cdots \rightarrow  H_{3}(X) \xrightarrow{\partial_{3}} H_{2}(A\cap B) \xrightarrow[\Phi_{2}]{(i_{*},j_{*})} H_{2}(A) \oplus H_{2}(B) \xrightarrow{k_{*}-l_{*}} H_{2}(X) \xrightarrow{\partial_{2}} H_{1}(A\cap B) \rightarrow \cdots \,, 
\end{align*}
which furnishes the short exact sequence
\begin{align}\label{short_sequence_5}
    0 \longrightarrow \coker(\Phi_{2}) \longrightarrow H_{2}(X) \longrightarrow \im(\partial_{2}) \longrightarrow 0.
\end{align}
By exactness, the image of $\partial_{2}$ coincides with the kernel of the map 
\[ \Phi_{1}: \Z^6 \oplus \Z_2^{6} \cong H_{1}(A\cap B) \rightarrow H_{1}(A)\oplus H_{1}(B) \cong \Z^6 \oplus \Z_2. \]
Now, $\Phi_{1}$ is surjective since $H_1(X) =0$. Moreover, $\ker(\Phi_1) \cong \Z_{2}^{5}$ because $\Phi_1$ restricts to an injective map $\Z^6\rightarrow \Z^6$. 
Since $H_{2}(A\cap B)$ is purely torsion, $\coker(\Phi_{2})$ has the same rank as $H_{2}(A) \oplus H_{2}(B) \cong \Z^{9} \oplus \Z_{2}^2$. These observations imply that $H_{2}(X) \cong \Z^{9} \oplus T_{2}$, where the torsion subgroup $T_2$ satisfies $T_{2}/H \leq (\Z_{2})^{5}$, and $H \leq \Z_2^2$. Here, $H$ is the image of the torsion subgroup $\Z_2^2$ of $H_2(A)\oplus H_2(B)$ under $k_*-l_*$. Thus, we can write $T_{2}=\Z_{2}^{k} \oplus \Z_4^l$ for some integers $k, l \geq 0$. 
We can prove that $k=l=0$ and hence $T_2=0$ by the exact same reasoning as in Theorem \ref{thm:hom_odd} \emph{Case 2}.
\end{proof}

Finally, we obtain the diffeomorphism type of the fiber in this case. 

\begin{corollary}[The Fiber $n=5$]
The manifold $X_5 = \hat{\mathcal{B}}(\mathcal{P})$ is diffeomorphic to $\#_{i=1}^{9} (\sphere^2 \times \sphere^3)$. 
\end{corollary}
\begin{proof}
By Corollary \ref{cor:highly-connected} and Lemma \ref{Lem:StablyTrivial}, the manifold $X$ is closed, $1$-connected and spin. Manifolds with these properties have been classified up to diffeomorphism by Smale \cite[Theorem A]{Sma62}. The complete invariant is the second homology group of $X$ over the integers, which in our case is $H_{2}(X,\Z)\cong \Z^{9}$. Since the smooth manifold $\#_{i=1}^{9} (\sphere^2 \times \sphere^3)$ satisfies all of the desired conditions, the result follows.  
\end{proof}

\begin{corollary}
Let $\rho: \pi_1S\rightarrow \SL(5,\R)$ be a Hitchin representation, and $\Omega_{\rho} \subset \mathcal{F}_{1,4}$ the domain \eqref{Omega_Thick}, with universal cover $\hat{\Omega}_{\rho} \subset V_2(\R^5)$. 
Then there is a smooth fiber bundle projection
$\hat{\Omega}_{\rho} \rightarrow \Ha^2$ with fiber diffeomorphic to $\#_{i=1}^{9}(\sphere^{2}\times \sphere^{3})$. 
\end{corollary}

\subsection{\texorpdfstring{$n=4$}{n=4}}\label{Sec:n=4}
When $n=4$, the fiber of the Guichard-Wienhard domain of discontinuity in $\mathcal{F}_{1,3}(\R^{4})$ is a $3$-manifold, and $X = \hat{\mathcal{B}}(\mathcal{P})$ is not simply connected. 
Nonetheless, we are able to determine the topology of $X$ using the almost-fibration. In particular, we will describe $X$ as a graph manifold obtained by gluing two copies of the complement in $\sphere^{3}$ of the link $L$ formed by the singular circles along their boundaries. Each piece is a Seifert fibered manifold, which we are able to describe completely. \medskip 

Let us first observe what the Structure Lemma \ref{Lem:StructureLemma} tells us. We have a smooth surjective map $p: X \rightarrow \sphere^3$, whereby
$X= X_{gen} \sqcup X_{sing}$ decomposes as follows: 
\begin{itemize}
    \item $\sphere^1 \rightarrow X_{sing} \rightarrow S_{sing}$, where $S_{sing} =:L \subset \sphere^3$ is a two component link. This fibration is orientable, hence $X_{sing}$ is the disjoint union of two 2-tori.
    \item $\sphere^0 \rightarrow X_{gen} \rightarrow S_{gen}= (\sphere^3 \setminus L)$. This fibration is also orientable, so $X_{gen}$ is a disjoint pair of copies of $\sphere^3\setminus L$. 
\end{itemize}

\begin{figure}[ht]
\centering
\includegraphics[scale=0.30]{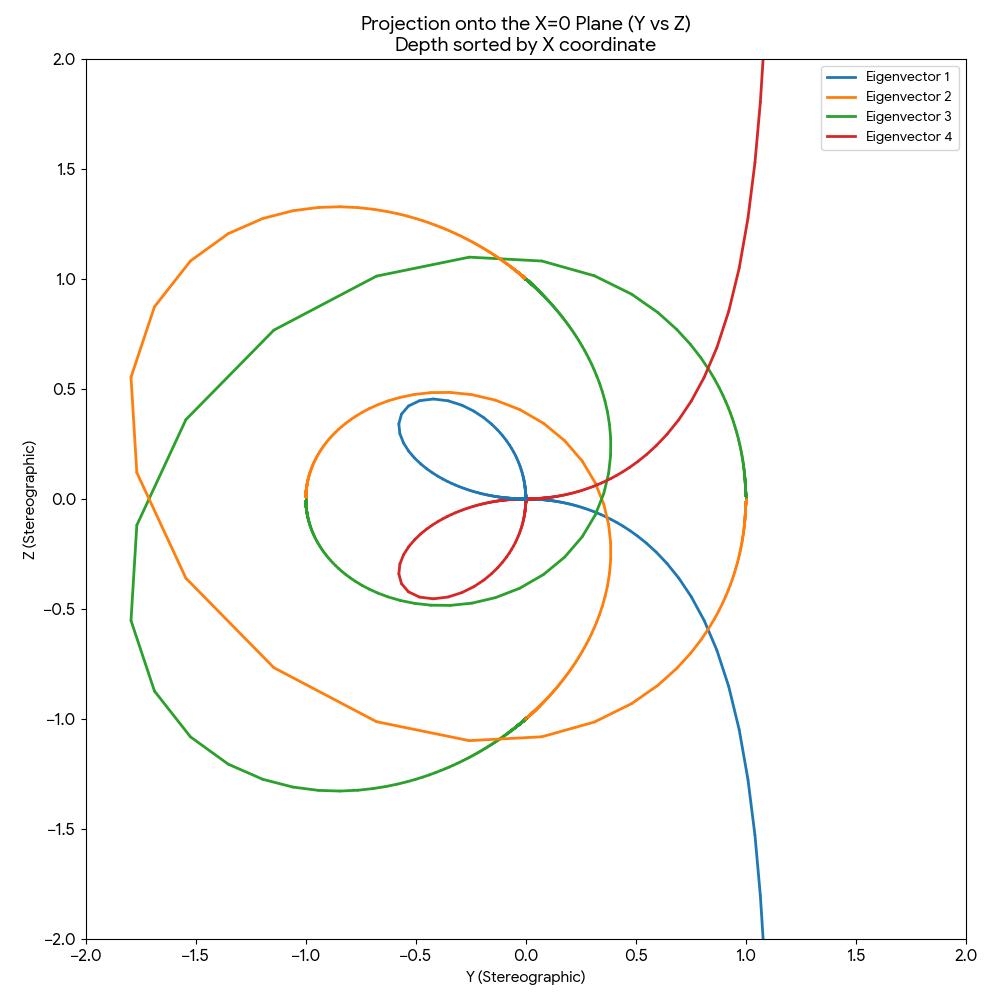}
\caption{\small{\emph{Link diagram of the singular circles after stereographic projection. Each color represents the antipodal paths traced by the unit eigenvectors of $\psi(t)$ relative to the $k$-th largest eigenvalue. The knot $K_1$ is red-blue and $K_2$ is orange-green.}}}
\label{fig:link}
\end{figure}

Next, we determine the link $L \subset \sphere^3$ formed by the singular circles. By Lemma \ref{lem:number_circles}, this link $L$ has two components $K_{1}$ and $K_{2}$. The component $K_1$ is traced out by the unit eigenvectors relative to the top and bottom eigenvalues of unit elements in the pencil, while $K_2$ is traced by the eigenvectors relative to the middle eigenvalues of the pencil. We shall work with the `simplified' pencil $\mathcal{P}_0$ described in Lemma \ref{lem:def_pencil}. Let us now describe these two knots $K_1, K_2$ explicitly. First, recall the basis $\{\psi_{1}, \psi_{2}\}$ of the pencil $\mathcal{P}_0$ from Lemma \ref{lem:def_pencil}. All elements of the pencil $\mathcal{P}_0$ excluding $\pm \psi_{1}$ can be written, up to scalars, as
\[
    \psi(t):=\psi_{2}+t\psi_{1}=\begin{pmatrix}
        \frac{3t}{2} & 1 & 0 & 0 \\
            1 & \frac{t}{2} & 1 & 0 \\
            0 & 1 & -\frac{t}{2} & 1 \\
            0 & 0 & 1 &-\frac{3t}{2}
    \end{pmatrix},
\]
with $t\in \mathbb{R}$. The eigenvector equation $\psi(t)v(t)=\lambda(t)v(t)$ yields the following recurrence in terms of the entries $v_{k}(t)$ of the eigenvector:
\begin{equation}\label{eq:recursion}
    \begin{cases}
        \left(\frac{3t}{2}-\lambda(t)\right)v_{1}(t)+v_{2}(t)=0 \\
        v_{1}(t)+\left(\frac{t}{2}-\lambda(t)\right)v_{2}(t)+v_{3}(t)=0 \\
        v_{2}(t)-\left(\lambda(t)+\frac{t}{2}\right)v_{3}(t)+v_{4}(t)=0\\
        v_{3}(t)-\left(\frac{3t}{2}+\lambda(t)\right)v_{4}(t)=0 \ ,
    \end{cases}
\end{equation}
which can be easily solved. Moreover, the characteristic polynomial of $\psi(t)$ is bi-quadratic, so the eigenvalues can be explicitly described as follows:
\[
    \lambda(t)=\pm \frac{1}{2}\sqrt{5t^2+6\pm2\sqrt{(2t^2+1)(2t^2+5)}}\ .
\]
Combining the formulas above, we obtain a parameterization of the knots $K_{i} \subset \sphere^{3}$. With the help of Python, we can visualize how the knots $K_1$ and $K_2$ link: we first project stereographically to $\mathbb{R}^{3}$ from the south pole $(0,0,0,1)$ and then we further project onto the $x=0$ plane in order to obtain the link diagram in Figure \ref{fig:link}. The code keeps track of which branch passes above or below each crossing by averaging the $x$-values of points near the intersections in the $\{x=0\}$-plane. We immediately see that the knot $K_{2}$ traced by the unit eigenvectors relative to the middle eigenvalues is trivial. For the other knot $K_{1}$, this projection is not sufficiently transverse, so we prove it directly now. 

\begin{lemma}\label{lem:trivial_knot} The knot $K_{1}\subset \sphere^3$ 
is trivial.    
\end{lemma}
\begin{proof} By Lemma \ref{lem:number_circles}, $K_{1}$ is invariant under the antipodal map.  We will show that its projection $\overline{K}_{1}$ to $\mathbb{RP}^{3}$ is isotopic to a projective line: this will conclude the proof, since, by the homotopy lifting property of the universal covering map, this implies that $K_{1}$ is isotopic to a great circle in $\sphere^{3}$ and it is thus the unknot. 

Now, the knot $\overline{K}_{1}$ is made up of the two arcs of projective classes of eigenvectors $v^{(1)}(t)$ and $v^{(4)}(t)$ of $\psi(t)$ relative to the top and bottom eigenvalues together with the points $[1,0,0,0]$ and $[0,0,0,1]$. The recursion in equation \eqref{eq:recursion} shows that $\overline{K}_{1}$ intersects the projective hyperplane $\{x_{1}=0\}$ only at $[0,0,0,1]$, so we can view the entire knot minus this point in the affine chart in which $x_{1}\neq 0$. We consider then the projective coordinate $X(t)=x_{2}(t)/x_{1}(t)$. 
For later, we introduce the function 
\[ \lambda_{max}(t)=\frac{1}{2}\sqrt{5t^2+6+2\sqrt{(2t^2+1)(2t^2+5)}}, \]
which is the top eigenvalue of $\psi(t)$, corresponding to the eigenvector $v^{(1)}(t)$. 

We claim that the coordinate $X(t)$ of the vector $v^{(1)}(t)$ is non-negative and decreasing. Indeed, 
the first equation in \eqref{eq:recursion} implies that for all $t\in \mathbb{R}$, 
\[
    -X(t)=\left(\frac{3t}{2}-\lambda_{max}(t)\right) \leq 0. 
\]
Moreover, we can directly compute that 
\[
    X'(t)=\left(\frac{d}{dt}\lambda_{max}(t)-\frac{3}{2} \right)<0,
\]
for all $t\in \mathbb{R}$. 

By a similar argument, the projective coordinate $X(t)$ of the vector $v^{(4)}(t)$ is non-positive and decreasing. We deduce that $\overline{K}_{1}$ intersects each affine plane $\{y\in \RP^3 \mid X(y)=c\}$ at exactly one point and we can thus isotope $\overline{K}_{1}$ to the projective line $\ell= \mathbb{P}( \langle e_{1}\rangle \oplus \langle e_{2}\rangle )$.
\end{proof}

The linking number of $L=K_{1}\cup K_{2}$ is computed from Figure \ref{fig:link} to be 3. This leads to the natural conjecture that $L$ is the torus link $T(2,6)$. 
We can now verify this is the case. 
\begin{lemma}\label{Lem:T(2,6)}
The singular locus $L=S_{sing}$ is isotopic to the torus link $T(2,6)$. 
\end{lemma}

\begin{proof}
See Figure \ref{fig:Reidemeister}. 
\begin{figure}[ht]
    \centering
    \includegraphics[width=0.50\linewidth]{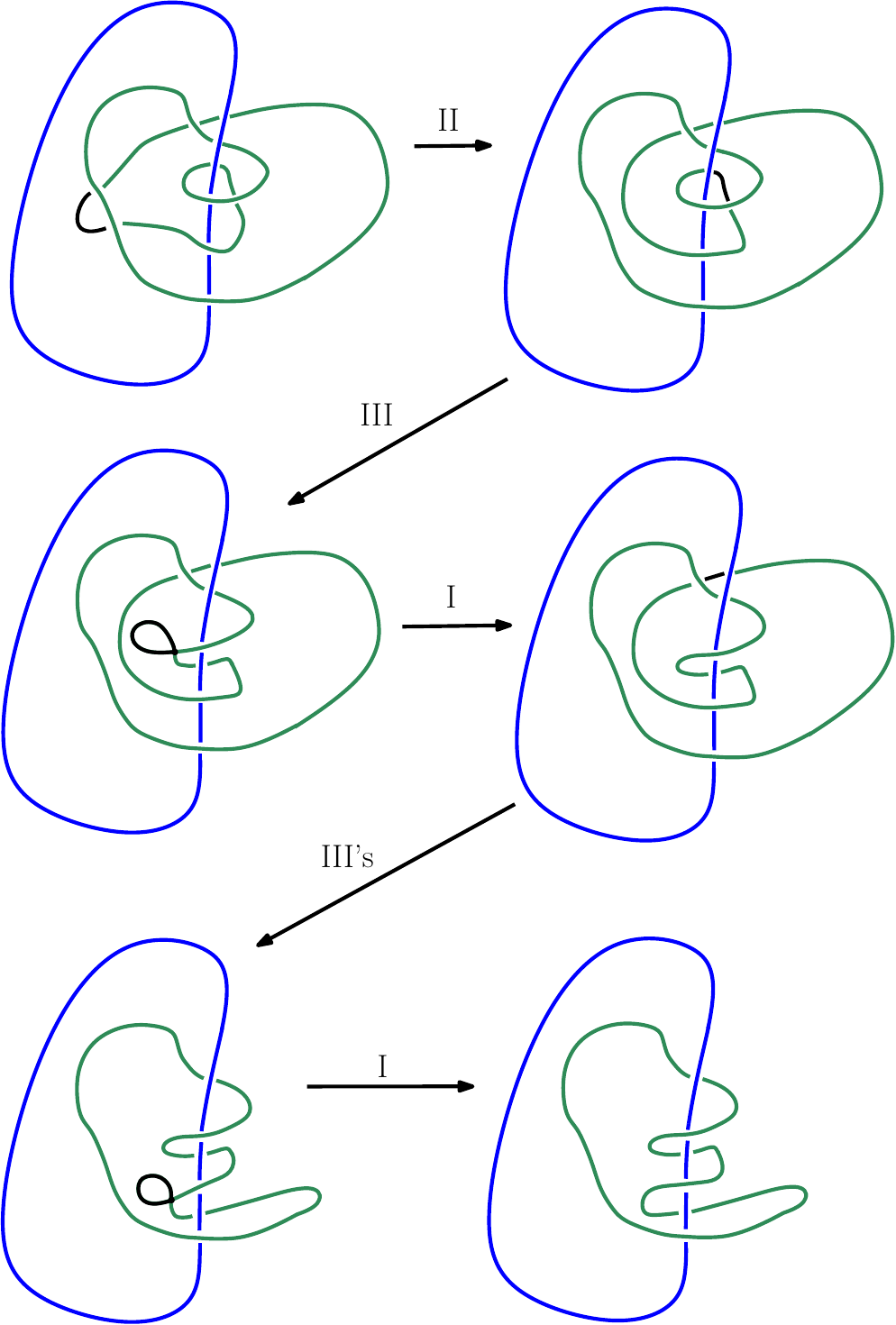}
    \caption{The Reidemeister moves to convert the singular link $L$ to the torus link $T(2,6)$. Here, in the initial diagram, $K_1$ is in blue and $K_2$ in green. We start with $L$ from Figure \ref{fig:link} after collapsing the two inner loops of $K_1$ and `straightening' it out. 
    In the penultimate step, we apply multiple III moves at once. See \cite[Figure 11.7]{Martelli} for more torus knots and links. }
    \label{fig:Reidemeister}
\end{figure}
\end{proof}

Lemma \ref{Lem:T(2,6)} gives us a lot of information about the topology of $X$. In fact, 
$X$ is obtained by gluing two copies $X_{+}$ and $X_{-}$ diffeomorphic to $\sphere^{3}\setminus N(L)$ along their boundary tori, where $N(L) \cong L \times \R^2$ is a regular neighborhood of $L$. Indeed, 
we can write $X_{gen}= X_{gen}^+\sqcup X_{gen}^-$ as connected components and set $X_{\pm}$ as the closure of $X_{gen}^{\pm}$ in $X$. Then $X$ is obtained from $X_+\sqcup X_-$ by gluing their mutual boundary $X_{sing}\cong \mathbb{T}^2\sqcup \mathbb{T}^2$.

3-manifold theory tells us that a torus link complement $\sphere^3 \setminus N(T(p,q))$ is Seifert fibered \cite{JSJ_dec_knot}. In our case, we have the following. 
\begin{proposition}
$X_{\pm}$ is isomorphic to the Seifert manifold $M(\Sigma_{0,2};(3,1))$: it fibers over a genus zero surface with two boundary components with one singular fiber of order $(3,1)$. 
\end{proposition}

\begin{proof} By \cite[Proposition 3.3]{JSJ_dec_knot}, the complement of $T(2,6)$ is the Seifert manifold $M(\Sigma_{0,2};(3,4))$. Apply the inverse of move $(10)$ in \cite[Proposition 10.3.11]{Martelli}.
\end{proof}

Analyzing more carefully the fibration $p: X \rightarrow \sphere^3$, we can understand how the boundary tori of $X_{\pm}$ are glued together, thus obtaining the following.

\begin{theorem}[The Fiber, $n=4$]\label{thm:double} $X$ is diffeomorphic to the double of $\sphere^3\setminus N(T(2,6))$.
\end{theorem}
\begin{proof} It is sufficient to show that the gluing maps between boundary tori are isotopic to the identity. We provide the proof for one boundary component, as the other case is identical.  

Let $U_{1}\subset \sphere^{3}$ be a tubular neighborhood of $K_{1}$ that does not intersect $K_{2}$. We view $U_{1}$ as foliated by tori $T_{r}$, with $r\in (0,\epsilon)$ being the distance from the core $K_{1}$, collapsing to the knot $K_{1}$ as $r\to 0$. 
By definition of the almost-fibration, we know that $p^{-1}(T_{r})=T_{r}\times \sphere^{0}$. We use coordinates $(\theta, \varphi)$ on $T_{r}$ so that the angle $\theta$ corresponds to rotation around a meridian of $K_{1}$ and the $\varphi$ corresponds to rotation around $K_{1}$ (longitudes). Let $\varphi=\varphi_{0}$ be fixed. 
The two points in the fiber of a point $x_{r}(\theta):=x_{r}(\theta, \varphi_{0})\in T_{r}$ are antipodal and will be denoted by $y_{\pm}(x_{r}(\theta))$. As $r\to 0$, the points $x_{r}(\theta)$ converge to $x_{0} \in K_{1}$ independently of $\theta$. By definition of the singular locus, there exists $\mathbf{c}=(c_{1}, c_{2}, c_{3}) \in \mathbb{R}^{3}\setminus \{\mathbf0\}$ such that $c_{1}x_{0}+c_{2}\psi_{1}(x_{0})+c_{3}\psi_{2}(x_{0})=0$. Now, the direction $\partial_{r}(\theta)$ of the gradient of the distance from $K_{1}$ is orthogonal to $K_{1}$, hence we can apply the discussion in the proof of Lemma \ref{Lem:Replace}, which says that each of $y_{\pm}(x_{r}(\theta))$ converge as $r\to 0$ to a point $y_{\pm}(x_{0}(\theta))$ in the singular fiber $p^{-1}(x_{0})=\sphere^{1}$ 
that is orthogonal to $x_{0}, \psi_{1}(x_{0}), \psi_{2}(x_{0})$ (which, we recall, are linearly dependent) and $c_{1}\partial_{r}(\theta)+c_{2}\psi_{1}(\partial_{r}(\theta))+c_{3}\psi_{2}(\partial_{r}(\theta))$. 
In particular, we have described $y_{\pm}(x_0)$ up to sign and found $y_{+}(x_0)=\pm y_{-}(x_0)$. 

As we vary $\theta$, the points $y_{\pm}(x_{0}(\theta))$ rotate around the singular fiber $p^{-1}(x_0)$, providing degree one maps $\mathcal{C}_{r,\varphi_0}^{\pm}\cong \sphere^1\rightarrow \sphere^1\cong p^{-1}(x_0)$ by $y_{\pm}(x_r(\theta))\mapsto y_{\pm}(x_0(\theta))$.
This means that the torus boundary components corresponding to $K_{1}$ are identified via a map
\begin{align*}
    \sphere^{1}\times \sphere^{1} &\rightarrow \sphere^{1}\times \sphere^{1} \\
        (\theta, \varphi) &\mapsto (\pm \theta, \varphi)
\end{align*}
which is isotopic to the identity, as claimed.    
\end{proof}

We can then identify $X$ as a standard Seifert fibered manifold. 

\begin{corollary}\label{Cor:SeifertFibered}
$X$ is the Seifert fibered manifold $M(\Sigma_{1,0};(3,1),(3,-1))$: it fibers over the torus $\mathbb{T}^2= \Sigma_{1,0}$ with Euler number $e=0$ and two singular fibers of order $(3,1)$ and $(3,-1)$. 
\end{corollary}
\begin{proof}
$X$ is the double of $\sphere^{3}\setminus N(T(2,6))$, hence it is obtained by gluing two copies of $\sphere^{3}\setminus N(T(2,6))$ with opposite orientations along their common boundary. The Seifert description of $\sphere^{3}\setminus N(T(2,6))$ with the opposite orientation is $M(0,2;(3,-1))$ \cite{Martelli}. Now, when we glue the two Seifert fibrations, the new base is the double of the original, which is $\Sigma_{1,0}$. We have one singular fiber from each copy of $X$, which have already been described. The Euler number $e=0$ is computed directly (cf. \cite[$\S$10.3.4]{Martelli}). 
\end{proof}

\begin{remark}
Since $\chi(B) < 0$, where $B=(\Sigma_{1,0};(3,1);(3,-1))$, and $e(M)= 0$, the closed Seifert fibered manifold $X$ carries $\Ha^2\times \R$-geometry \cite[Proposition 12.4.6]{Martelli}. 
\end{remark}

\subsection{\texorpdfstring{$n=3$}{n=3}}\label{Sec:n=3}

This section is a sanity check. We apply the same strategy to show how to reinterpret the well-known case of $X =\mathcal{F}_{1,2} = \Flag(\R^3)$. See also \cite[Example 6.9]{Dav25}. 

We consider $\mathcal{P} = \T_Q\Ha^2_{\pr}$, as usual. 
We fix the following basis for the pencil $\mathcal{P}$:
$$\psi_{1}=\begin{pmatrix}
                2 & 0 & 0 \\
                0 & 0 & 0 \\
                0 & 0 & -2
    \end{pmatrix}   \ \ \ \ \  \psi_{2}=\begin{pmatrix} 
                                         0 & 1 & 0 \\
                                         1 & 0 & 1 \\
                                         0 & 1 & 0
                                        \end{pmatrix}.$$  
                                        
The base of pencil $\hat{\mathcal{B}}(\mathcal{P})$ identifies with the total space $\mathbb{S}(\mathcal{R}^{\perp})$, which we can write as
\[
    \mathbb{S}(\mathcal{R}^{\perp})=\{(x,y)\in \mathbb{S}^{2}\times \mathbb{S}^{2} \ | \ \langle x,y \rangle = \langle \psi_{1}(x),y\rangle = \langle \psi_2(x), y \rangle = 0\} \ .
\]
We denote by $p:\sphere(\mathcal{R}^\bot) \rightarrow \sphere^2$ the projection onto the first factor. Generically, given $x \in \mathbb{S}^{2}$, the vectors $x$, $\psi_{1}(x)$ and $\psi_{2}(x)$ are linearly independent, so $p^{-1}(x)=\emptyset$. In other words, the ``generic locus'' $X_{gen}$, as we have called it in this work, is empty. 

We now describe the singular locus.
First, note the fiber over $x \in \sphere^2$ is non-empty if and only if $\mathcal{R}_{x}=\pi_{x^\bot}\{\psi_{1}(x), \psi_{2}(x)\}$ has dimension $1$, meaning $\dim \Span \{x, \psi_1(x),\psi_2(x)\}=2$. 
Hence, the singular locus $S_{sing} = p(\sphere(\mathcal{R}^\bot))$ is given by 
\[
    p(\mathbb{S}(\mathcal{R}^{\perp}))=\{x \in \mathbb{S}^{2} \ | \ \det(x,\psi_{1}(x),\psi_{2}(x))=0 \} \ .
\]
We can now compute this set explicitly: a point $(x_{1},x_{2},x_{3}) \in p(\mathbb{S}(\mathcal{R}^{\perp}))$ if and only if
\[
    0=\det\begin{pmatrix}
        x_{1} & 2x_{1} & x_{2} \\
        x_{2} & 0 & x_{1}+x_{3} \\
        x_{3} & -2x_{3} & x_{2}
    \end{pmatrix} = 2(x_{1}+x_{3})(2x_{1}x_{3}-x_{2}^2)=2(x_{1}+x_{3})((x_{1}+x_{3})^2-1) .
\]
Therefore, the singular locus $S_{sing}=p(\mathbb{S}(\mathcal{R}^{\perp}))$ consists of three disjoint circles:
\[
    \gamma_{0}=\{x_{1}+x_{3}=0\} \cap \mathbb{S}^{2} \ \ \ \text{and} \ \ \ \ \gamma_{\pm1}=\{x_{1}+x_{3}=\pm 1\} \cap \mathbb{S}^{2}.
\]
For each $x \in S_{sing}$, the preimage $p^{-1}(x)$ consists of the two unit vectors orthogonal to $x, \psi_{1}(x)$ and $\psi_{2}(x)$, thus $\hat{\mathcal{B}}(\mathcal{P})$ is the union of six disjoint circles:
\begin{align*}
    \gamma_{0a}&=\left\{ (x,y) \in \sphere^2 \times \sphere^2 \mid (x_{1},x_{2},x_{3}) \in \gamma_{0}, \; y=\frac{1}{\sqrt{2}}(x_{2},-2x_{1},-x_{2})  \right\} \\
    \gamma_{0b}&=\left\{ (x,y) \in \sphere^2\times \sphere^2 \mid (x_{1},x_{2},x_{3}) \in \gamma_{0}, \; y=\frac{1}{\sqrt{2}}(-x_{2}, 2x_{1},x_{2})  \right\} \\
    \gamma_{1a}&=\left\{ (x,y) \in \sphere^2 \times \sphere^2 \mid  (x_{1},x_{2},x_{3}) \in \gamma_{1}, \;y=(-x_{3},x_{2},-x_{1}) \right\} \\
      \gamma_{1b}&=\left\{ (x,y) \in \sphere^2 \times \sphere^2 \mid  (x_{1},x_{2},x_{3}) \in \gamma_{1}, \;y= (x_{3},-x_{2},x_{1}) \right\} \\
    \gamma_{-1a}&=\left\{ (x,y) \in \sphere^2 \times \sphere^2 \mid (x_{1},x_{2},x_{3}) \in \gamma_{-1}, \; y= (-x_{3},x_{2},-x_{1}) \right\} \\
    \gamma_{-1b}&=\left\{ (x,y) \in \sphere^2 \times \sphere^2 \mid (x_{1},x_{2},x_{3}) \in \gamma_{-1} , \; y= (x_{3},-x_{2},x_{1}) \right\}
\end{align*}
That is, $p^{-1}(\gamma_i)=\gamma_{ia} \sqcup \gamma_{ib}$ in our notation.

The group $\mathbb{Z}_{2}\times \mathbb{Z}_{2}$ generated by the duality $\sigma(x,y)=(y,x)$ and the antipodal map $a(x,y)=(-x,-y)$ acts on $\hat{\mathcal{B}}(\mathcal{P})$. Each map permutes the six circles in the following way (this can be verified by looking at the explicit parameterization of each circle given above):
\begin{align*}
    \sigma(\gamma_{1a})=\gamma_{-1a} \ \ \ \ \  \sigma(\gamma_{1b})=\gamma_{1b} \ \ \ \ \ \ \sigma(\gamma_{0a})=\gamma_{0b} \\
    a(\gamma_{1a})=\gamma_{-1a} \ \ \ \ \  a(\gamma_{1b})=\gamma_{-1b} \ \ \ \ \ \ a(\gamma_{0a})=\gamma_{0a} 
\end{align*}

The action of $\sigma$ and $a$ on the remaining three circles follows since they are involutions. We conclude that under the $\Z_2\times \Z_2$ action, the following pairs of circles are glued together: $(\gamma_{1, a},\gamma_{-1,a}), (\gamma_{1,b},\gamma_{-1,b}), (\gamma_{0a},\gamma_{0b})$. 
Since the $\mathbb{Z}_{2}\times\mathbb{Z}_{2}$-action is free, we conclude that $G_{Q}(\hat{\mathcal{B}}(\mathcal{P}))$ and hence the fiber $\mathfrak{F}$ in $\mathcal{F}_{1,2}$ is the disjoint union of three circles.

\bibliographystyle{alpha}
\small{\bibliography{Bibliography}}

\end{document}